\documentclass{amsart}
\usepackage[T1]{fontenc}
\usepackage[dvipsnames]{xcolor}
\usepackage{tabularx}
\usepackage{rotating}
\usepackage[all]{xy}
\usepackage{hypbmsec}
\usepackage{amssymb}
\usepackage{epsfig}
\usepackage{amsfonts}
\usepackage{euscript}

\theoremstyle{plain}
\newtheorem{theor}{Theorem}[section]
\newtheorem{prop}[theor]{Proposition}
\newtheorem{lem}[theor]{Lemma}
\newtheorem{cor}[theor]{Corollary}
\newtheorem{const}[theor]{Construction}
\newtheorem{alg}[theor]{Algorithm}
\newtheorem*{set}{Settings}
\theoremstyle{definition}
\newtheorem{de}[theor]{Definition}
\newtheorem{ex}[theor]{Example}
\newtheorem{prob}[theor]{Problem}
\newtheorem{q}[theor]{Question}
\theoremstyle{remark}
\newtheorem{re}[theor]{Remark}

\DeclareMathOperator{\Spec}{Spec}
\DeclareMathOperator{\KK}{\mathbb{K}}
\def\BG{{\mathbb G}}

\DeclareMathOperator{\ML}{ML}
\DeclareMathOperator{\FML}{FML}

\def\SAut{\mathrm{SAut}}
\def\Quot{\mathrm{Quot}}
\DeclareMathOperator{\Aut}{Aut}
\DeclareMathOperator{\Ker}{Ker\,}

\DeclareMathOperator{\VV}{\mathbb{V}}
\DeclareMathOperator{\LND}{\mathrm{LND}}
\DeclareMathOperator{\Susp}{\mathrm{Susp}}
\DeclareMathOperator{\Bisusp}{\mathrm{Bisusp}}
\DeclareMathOperator{\Dan}{\mathrm{Dan}}

\title[Non-cancellative varieties]{Non-cancellative varieties, maximal tori, and the Makar-Limanov invariant}
\author{Sergey Gaifullin and Mikhail Petrov}
\address{Lomonosov Moscow State University, Faculty of Mechanics and Mathematics, Department of Higher Algebra, Leninskie Gory 1, Moscow, 119991 Russia;
\linebreak and \linebreak
HSE University, Faculty of Computer Science, Pokrovsky Boulvard 11, Moscow, 109028 Russia.}
\email{sgayf@yandex.ru, mikhail.petrov@math.msu.ru}
\date{}



\thanks{The work was supported by the Theoretical Physics and Mathematics Advancement Foundation <<BASIS>>, project number 24-7-2-20-1.}

\begin{document}
\begin{abstract}
In this paper, we construct a wide class of new counterexamples to the generalized Zariski cancellation problem. The cylinders over these counterexamples have infinitely many non-conjugate maximal tori in their regular automorphism groups. We provide an example of a variety with maximal tori of different dimensions in its automorphism group. Additionally, we prove that a cylinder over a non-rigid trinomial variety without a line factor is generically flexible, and hence, has a trivial Makar-Limanov invariant.
\end{abstract}
\maketitle
\section{Introduction}

Let $\KK$ be an algebraically closed field of characteristic zero and $X$ be an affine algebraic variety over $\KK$. The variety $X$ is said to be {\it cancellative} if $X\times \KK\cong Y\times \KK$ implies $X\cong Y$ for any affine algebraic variety $Y$. One of the most challenging problems in Aﬃne Algebraic Geometry is the Zariski Cancellation Problem, which asks whether the affine space is cancellative. For fields of positive characteristic, the Zariski Cancellation Problem has a negative answer, see~\cite{Gu2} and~\cite{Gu}. For fields of characteristic zero, it remains open.

The first example of a non-cancellative variety was provided in 1989 by Danielew\-ski~\cite{Da}. He proved that cylinders $X_n\times \KK$ over the surfaces $X_n=\{x^ny=f(z)\}$, where $\deg f\geq 2$ and $f$ has no multiple roots, are isomorphic for all positive integers~$n$ while $X_1$ and $X_2$ are not isomorphic. Non-cancellative varieties are often said  to be {\it counterexamples to the generalized Zariski Cancellation Problem}, which asks whether $X\times \KK\cong Y\times \KK$ implies $X\cong Y$. Now some other examples of non-cancellative varieties are known; see, for example,~\cite{Du, Du2, DF, GS, DG, FKZ}. Although constructions are mostly geometric, many known examples of non-cancellative varieties have the same algebraic form: the variety $Y$ can be obtained from $X$ by replacing the combination $xy^2$ by $xy$ in all equations defining $X$.

In this paper we develop a purely algebraic technique for constructing non-cancellative varieties. This technique allows one to replace $xy^2$ by $xy$ in all equations when $X$ satisfies some conditions and obtain a variety $Y$ with an isomorphic cylinder. This allows us to construct many new examples of non-cancellative varieties; see Section~\ref{chetvert}. We use the Slice Theorem to prove the isomorphism of the cylinders. Thus, computing the kernel of a locally nilpotent derivation (LND) plays a key role in this technique. We use van den Essen's algorithm~\cite{Ess} and Gr\"obner basis calculations for this purpose. Section~\ref{kol} is devoted to describing such kernels.

In his fundamental work~\cite{Da}, Danielewski used non-cancellative surfaces $X$ to obtain non-conjugate two-dimensional tori in the group of regular automorphisms $\Aut(X\times\KK)$. Since the considered varieties are not toric, these two-dimensional tori are maximal in $\Aut(X\times\KK)$. Danielewski proved that there exists infinite sequence of pairwise non-conjugate algebraic 2-tori in $\Aut(X\times\KK)$. One more known example of a variety with infinite number of conjugate classes of maximal tori in the automorphism group is the variety $\mathrm{SL}(2)\times\KK$; see~\cite{Da} and~\cite[Example~8.14]{KrZ}.
The idea is that for each decomposition of a variety into a product $X\times \KK$ one obtains a maximal torus in $\Aut(X\times \KK)$ being the product of maximal tori in $\Aut(X)$ and $\Aut(\KK)$. Then one needs to prove that these tori are non-conjugate. We use the technique of so-called {\it isolated irreducible semi-invariants} for this purpose, see~\cite{BG} and Section~\ref{secsemiinv}. This technique allows us to prove that a wide class of non-cancellative varieties $X$ obtained in Section~\ref{chetvert} has infinitely many non-conjugate maximal tori in $\Aut(X\times \KK)$; see Section~\ref{secconj}. We obtain such examples in all dimensions $\geq 3$. In particular, each non-rigid trinomial hypersurface of Type I that is not isomorphic to a product of a rigid trinomial hypersurface and an affine space, has infinitely many non-conjugate tori of maximal dimension; see Corollary~\ref{ssylkan}. 

Moreover, we demonstrate a new phenomenon: the automorphism group of an affine variety can contain maximal tori of different dimensions.  We construct a family of varieties $X_n$ of dimension $n$ such that $\Aut(X_n)$ admits  maximal tori of dimensions $1, 2, \ldots, n-1$; see Theorem~\ref{raztor}. Let us mention the important result due to Popov, who proves that the Cremona group $\mathrm{Cr}_n(\KK)$ admits maximal tori of different dimensions~\cite[Theorem~2]{Po}. The varieties considered in our work are rational, and hence, their automorphism groups are included in Cremona group. But it is easy to show that in the Cremona group none of the considered tori are maximal; rather, they and are conjugate to subtori of the standard torus. 

The rest of the work is devoted to describing the Makar-Limanov invariant and proving the generic flexibility of a wide family of varieties including cylinders over trinomial varieties; see Section~\ref{trvtrv} for definitions. In \cite{ML96} Makar-Limanov introduced an invariant $\ML(X)$ of a variety $X$ (initially it was denoted $\mathrm{AK}(X)$), which was later named the {\it Makar-Limanov invariant}. This invariant is the  subalgebra of the algebra of regular functions  $\KK[X]$ given by the intersection of  kernels of all LNDs on $X$:
$$
\ML(X) = \bigcap_{\partial \in \LND(X)} \Ker\partial.
$$ 
 This invariant allows one to prove that the  Koras-Russell cubic 
 $$C=\{x+x^2y+z^2+t^3=0\}$$ 
 is not isomorphic to the three-dimensional affine space. Namely, it is easy to see that $\ML(\KK^n)=\KK$, and Makar-Limanov proved that $\ML(C)=\KK[x]$, hence, it is not isomorphic to $\KK$. Therefore, $C$ is not isomorphic to~$\KK^3$. Generalization of this result given by Kaliman and Makar-Limanov  in~\cite{KML} was the final step in the proof due to Koras and Russell of the Linearization Problem for $\KK^\times$-action on $\mathbb{K}^3$. 
The Makar-Limanov invariant became a powerful tool for distinguishing non-isomorphic varieties. Also it can be used for studying the automorphism group of a variety, see~\cite{ML96, KML, ML01, MJ, P, Ga}. 

Since the Koras-Russell cubic is not isomorphic to the three dimensional affine space, but is rather similar to it, it became a natural candidate to be a counterexample to the Zariski Cancellation Problem. The natural question was: does Makar-Limanov invariant allow to distinguish the cylinder over the Koras-Russell cubic $C\times \KK$ and the four-dimensional affine space? 
It is easy to see that $\ML(X\times\KK) \subseteq \ML(X)$. Therefore $\ML(C\times \KK)\subseteq \KK[x]$. 
In~\cite{D} Dubouloz showed that $\ML(C\times\KK)=\KK$. And hence, $\ML$ does not distinguish $C\times \KK$ and~$\KK^4$. Some results and questions in the area of distinguishing $C\times \KK$ and~$\KK^4$ can be found in~\cite{GSh} and~\cite{BGSh}.

 This paper was inspired by a discussion with N. Gupta and P. Ghosh. They showed a purely algebraic proof of Dubouloz's result that  $\ML(C\times\KK)=\KK$. Unfortunately, this proof seems to be unpublished, so we cannot provide a reference. The key point in this proof is a technique that, under certain conditions, allows one to lift an LND $\delta$ on $A=\Ker D$, where $D$ is an LND on $B$, to an LND on $B$. The author of this proof is apparently Bhatwadekar. So, the technique which we introduce in this paper and which is based on this example, we call Bhatwadekar's technique. We elaborate generalizations of this technique and apply to a wide class of varieties obtained by construction of $m$-suspension, which was introduced in~\cite{Ga}, and Bisuspensions; see Sections~\ref{SB}, \ref{IAGBT}, and \ref{secbis}. 

It turns out that it is more convenient to investigate the {\it field Makar-Limanov} invariant $\mathrm{FML}(X)$ instead of the Makar-Limanov invariant by this technique. Let us recall that the field Makar-Limanov invariant is a subfield in the field of rational functions on $X$ defined by
$$
\FML(X)=\bigcap_{D\in \LND(X)}\Quot(\Ker D)\subseteq \KK(X). 
$$
By a $\BG_a$-subgroup of $\Aut(X)$ we mean an algebraic subgroup isomorphic to the additive group of the ground field $\KK$. The subgroup $\SAut(X)$ in $\Aut(X)$ generated by all $\BG_a$-subgroups in $\Aut(X)$ is called the subgroup of {\it special automorphisms} of~$X$. This subgroup was investigated in~\cite{AFKKZ}. It is easy to see that the Makar-Limanov invariant coincides with the algebra of regular $\SAut(X)$-invariants $\KK[X]^{\SAut(X)}$, and the field Makar-Limanov invariant coincides with the  field of rational $\SAut(X)$-invariants $\KK(X)^{\SAut(X)}$. In particular, $\FML(X)=\KK$ implies $\ML(X)=\KK$. 
The condition $\FML(X)=\KK$ is equivalent to the existence of an open $\SAut(X)$-orbit on $X$; see~\cite{AFKKZ}. Then the variety $X$ is called {\it generically flexible}; see Section~\ref{fmlgf} for details. As an application of the above technique in
Section~\ref{cyltrv} we prove that cylinders over non-rigid trinomial varieties, that are not isomorphic to a product of a rigid trinomial variety and an affine space, are generically flexible.

The authors are grateful to Neena Gupta and Parnashree Ghosh for introducing them to the Bhatwadekar's technique of computing the Makar-Limanov invariant of the cylinder over the Koras-Russell cubic. The authors are also grateful to Ivan Arzhantsev and Mikhail Zaidenberg for fruitful discussions.

\section{Preliminaries}
\subsection{Derivations}
Recall that $\KK$ is an algebraically closed field of characteristic zero. Let $B$ be a $\KK$-domain. A linear map $\partial\colon B\rightarrow B$ is called a \emph{derivation} if it satisfies the Leibniz rule: $\partial(ab)=a\partial(b)+b\partial(a)$. A derivation is called {\it locally nilpotent} (LND) if for any $b \in B$ there exists $n\in \mathbb{N}$ such that $\partial^n(b)=0$. We denote the set of all LNDs on $B$ by $\LND(B).$ We assume that $B = \KK[X]$ is the algebra of regular functions on an irreducible affine variety $X$. So, by $\LND(X)$ we mean $\LND(\KK[X]).$

Each locally nilpotent derivation $\partial$ corresponds to an algebraic action of the group $\BG_a = (\KK, +)$ on $B$, and hence on $X$ either:
$$s\cdot b = \mathrm{exp}(s\partial)(b) = \sum_{i=0}^{+\infty} \frac{s^i\partial^i(b)}{i!},$$
where $s \in \BG_a,\ b\in B.$ Here we mean $\partial^0=\mathrm{id}$. The sum is well-defined since for each $b$ only a finite number of terms are nonzero. If $\partial$ is nonzero, this action defines an injective homomorphism to the group of regular automorphisms on $X$.  Denote the image of this homomorphism by $H_\partial$. Such subgroups of $\Aut(X)$ we call $\BG_a$-{\it subgroups}.  It is known that every homomorphism $\BG_a \to \Aut(X)$ is obtained by this construction for some LND $\partial$, see \cite{F} for details. If $\partial \in \mathrm{LND}(X)$ then the kernel of $\partial$ 
$$\mathrm{Ker}\ \partial = \{f\ |\ \partial(f) = 0\}$$
coincides with the set of regular invariants $\KK[X]^{H_\partial}$ with respect to the natural action of the group $H_\partial\subseteq \mathrm{Aut}(X)$ on $\KK[X].$

Let $G$ be a commutative group. Algebra $B$ is said to be \textit{$G$-graded} if there exists a decomposition into a direct sum of subspaces
$$B = \bigoplus_{g \in G} B_g,$$
such that for any $a \in B_n$ and $b \in B_m$ we have $ab \in B_{n+m}$.
A derivation $D$ on $B$ is called \textit{$G$-homogeneous of degree $d$} if for any $b \in B_n$ we have $D(b) \in B_{n+d}$. Let $B$ be a finitely generated $\mathbb{Z}$-graded algebra, and let $D$ be a derivation on $B$. Then there exists a decomposition $D = \sum\limits_{i=k}^{l} D_i$, where every $D_i$ is $\mathbb{Z}$-homogeneous of degree $i$. The following lemma was proved in~\cite{Re}.
\begin{lem}
If $D\in\LND(B)$ and $D = \sum\limits_{i=k}^{l} D_i$, then $D_k$ and $D_l$ are locally nilpotent.
\end{lem}

An element $r \in B$ is called a \textit{local slice} for $D \in \LND(B)$ if $D^2(r) = 0$ but $D(r) \neq 0$. An element $s \in B$ is called a \textit{slice} for $D \in \LND(B)$ if $D(s) = 1$.
\begin{re}
    For any nonzero LND there exists at least one local slice, but there may not be a slice.
\end{re}
For a local slice $r$ of $D$, the \textit{Dixmier map} induced by $r$ is a homomorphism $\pi_r: B \rightarrow B_{D(r)}$ defined as follows:
$$\pi_r(f) = \sum_{i \geq 0} \frac{(-1)^i}{i!}D^i(f)\frac{r^i}{(D(r))^i}.$$
If it is not clear, for which LND we consider the Dixmier map, we use denotation $\pi_{D,r}$ instead of $\pi_r$. 
We need the following statement.
\begin{prop}\cite[Principle~11]{F}\label{ssnf}
    Let $D\in\LND(B)$ be given, $D\neq 0$, and set $A=\Ker D$. Choose a local
slice $r\in B$ of $D$, and letr $\pi_r: B \rightarrow B_{D(r)}$ denote the Dixmier map defined by $r$. Then
\begin{itemize}
    \item[(a)] $\pi_r(B)\subseteq A_{D(r)}$;
    \item[(b)] $\pi_r(B)$ is a $\KK$-algebra homomorphism;
    \item[(c)] $\Ker\pi_r=rB_{D(r)}\cap B$;
    \item[(d)] $B_{D(r)} = A_{D(r)}[r]$;
    \item[(e)] The transcendence degree of $B$ over $A$ is 1.
\end{itemize}
\end{prop}

In particular case when $r$ is a slice of $D$ the point (d) gives the assertion which is called the Slice Theorem.
\begin{cor}[Slice Theorem]
    Let $D$ be an LND on $B$ and $s$ be a local slice for~$D$. Then $B = (\Ker D)[s]$.
\end{cor}
By definition, $\pi_r(a)=a$ for all $a\in A$. This implies the following two lemmas.

\begin{lem}\label{yadrodiks}
    $A=\pi_r(B)\cap B$.
\end{lem}
\begin{proof}
    For each $a\in A$ we have $a=\pi_r(a)\in \pi_r(B)\cap B$. Conversely, 
    $$\pi_r(B)\cap B\subseteq A_{D(r)}\cap B\subseteq \Quot(A)\cap B=B.$$
    The last equality holds by \cite[Principle~13]{F}.
\end{proof}
\begin{lem}\label{yadrofild}
$\Quot(A)=\Quot(\pi_r(B))$.
\end{lem}
\begin{proof}
We have $A\subseteq \pi_r(B)\subseteq A_{D(r)}$. Moving on to inclusions of quotient fields, we obtain
$\Quot(A)\subseteq\Quot(\pi_r(B))\subseteq \Quot(A)$.
\end{proof}

\subsection{Makar-Limanov invariant, field Makar-Limanov invariant and generic flexibility}\label{fmlgf}

Let us remind the definition of the Makar-Limanov invariant. 
\begin{de}
    The {\it Makar-Limanov invariant} of an algebra $B$ is the intersection of kernels of all LNDs on $B$. 
    $$\ML(B)=\bigcap\limits_{D\in \LND(B)}\Ker D.$$ 
    If $B=\KK[X]$ is the algebra of regular functions on an affine variety $X$, we use the notation $\ML(X)$ instead of $\ML(\KK[X])$.  
\end{de}

Following \cite{AFKKZ}, we consider the subgroup of special automorphisms of $X$. 
\begin{de}
    The subgroup $\SAut(X)$ in $\Aut(X)$ generated by all $\BG_a$-subgroups in $\Aut(X)$ is called the subgroup of {\it special automorphisms} of $X$.
\end{de}

Because of the correspondence between LNDs and $\BG_a$-subgroups, $\ML(X)=\KK[X]^{\SAut(X)}$. 

\begin{de}
    An affine variety $X$ is called {\it rigid} if $\ML(X)=\KK[X]$. 
\end{de}
In other words $X$ is rigid if $\KK[X]$ does not admit any nonzero LNDs, or equivalently if $X$ does not admit any nontrivial $\BG_a$-actions.

Each derivation of $B$ can be naturally extended to a derivation of the quotient field $\Quot(B)$. 
\begin{de}
    The {\it field Makar-Limanov invariant} of an irreducible variety $B$ is the intersection of kernels of all derivations of the field of rational functions $\KK(X)$ that are extensions of LNDs on~$\KK[X]$. 
    $$\FML(X)=\bigcap\limits_{D\in \LND(X)}\Ker \overline{D},$$ 
    where $\overline{D}$ is the extension of $D$ to $\KK(X)=\Quot(\KK[X])$. 
\end{de}
It is easy to see that $\FML(X)=\KK(X)^{\SAut(X)}$. 
By~\cite[Corollary~1.29]{F}, $\Ker\overline{D}=\Quot(\Ker D)$. This implies that 
$$
\FML(X)=\bigcap_{D\in \LND(X)}\Quot(\Ker D)\subseteq \KK(X). 
$$

\begin{de}
A point $x\in X$ is called {\it flexible} if the tangent space $\mathrm{T}_xX$ is spanned by the tangent vectors to orbits of $\BG_a$-subgroups in $\Aut(X)$. 
\end{de}
It is easy to see that a flexible point is regular; see for example~\cite[Lemma~4.2]{Ga2}.
\begin{de}
    A variety $X$ is called {\it flexible} if each regular point is flexible. 
\end{de}
An action of a group $G$ on a set $Y$ is called $m$-transitive if for every
two $m$-tuples of pairwise distinct points $(a_1,...,a_m)$ and $(b_1,...,b_m)$ on $Y$
there exists an element $g\in G$ such that for all $i$ we have $g\cdot a_i = b_i$.
An action is {\it infinitely transitive} if it is $m$-transitive for all positive
integers $m$.

The following wonderful fact inspired interest to flexible varieties. 
\begin{prop}\cite[Theorem 0.1]{AFKKZ} For an irreducible affine variety
$X$ of dimension $\geq 2$ the following conditions are equivalent:
\begin{itemize}
    \item[(i)] the group $\SAut(X)$ acts transitively on the regular locus $X^{\mathrm{reg}}$;
    \item[(ii)] the group $\SAut(X)$ acts infinitely transitively on $X^{\mathrm{reg}}$;
    \item[(iii)] the variety $X$ is flexible.
\end{itemize}
\end{prop}
In the same paper~\cite{AFKKZ} it is proved that each $\SAut(X)$-orbit is open
in its closure. In particular, if $\SAut(X)$-orbit is dense in $X$, it is open.
\begin{de}
If $\SAut(X)$ acts on X with an open orbit, the variety $X$ is called
{\it generically flexible}. 
\end{de}
Generic flexibility does not imply flexibility; see~\cite{Ga2} for more information in this field. 

The following proposition connects generically flexibility and triviality of $\FML$.
\begin{prop}\cite[Proposition 5.1.]{AFKKZ}
 An irreducible aﬃne variety $X$ possesses a flexible point if and only
if the group $\SAut(X)$ acts on $X$ with an open orbit, if and only if the field Makar-Limanov invariant $\FML(X)$
is trivial. In the latter case X is unirational.
\end{prop}

\subsection{Trinomial varieties}\label{trvtrv}

In this section we give a rigorous definition of a trinomial variety according to~\cite{H-W}.

\begin{const}\label{odyn} \cite[Construction 1.1]{H-W}. Suppose we are given positive integers~$r$ and~$n$, a nonnegative integer  $\theta$ and $q \in \left\{ {0, 1} \right\}$. Let us fix a partition  $n = n_{q} +  \ldots+ n_{r}$ of~$n$ into positive integer summands. 
Let us consider the polynomial algebra $B$ in $\theta+n$ variables. These variables we denote $T_{ij}$ and~$S_k$:
$$
B= \KK \left[ {T_{ij}, S_{k} \,|\, q \leq i \leq r, \!\ 1 \leq j \leq n_{i}, \!\ 1 \leq k \leq \theta} \right].
$$
For each $i = q, \ldots, r$ we fix a tuple of positive integers $l_{i} = \left( {l_{i1},  \ldots , l_{in_{i}}} \right)$. So we can consider the following monomial:
$$
	T_{i}^{l_{i}} = T_{i1}^{l_{i1}}  \ldots T_{in_{i}}^{l_{in_{i}}} \in B. 
$$
Now we define the {\it trinomial algebra} $R(A)$, which we construct by some data $A$.  These data are different for two types of trinomial algebras. 
\par \emph{Type I.} $ q = 1, \!\ A = (a_{1},  \ldots, a_{r}), \!\ a_{j} \in \KK, \!\ a_{i} \neq a_{j} $ if $i \neq j$.  Let us put  $I = \left\{ {1, \ldots, r - 1} \right\}$ and 
\begin{equation*}
	g_{i} = T_{i}^{l_{i}} - T_{i+1}^{l_{i+1}} - (a_{i+1} - a_{i}) \in B, \ i \in I.
\end{equation*}
\par \emph{Type II.} $ q = 0$, 
\begin{equation*}
	A = \begin{pmatrix}
		a_{10} \ a_{11} \ a_{12} \cdots a_{1r} \\
		a_{20} \ a_{21} \ a_{22} \cdots a_{2r} \\
	\end{pmatrix}
\end{equation*}
is such a matrix with elements from $\mathbb{K}$, that every two columns are linearly independent. Let us put $I = \left\{ {0,  \ldots , r - 2} \right\}$ and 
\begin{equation*}
	g_{i} = \det \begin{pmatrix}
		T_{i}^{l_{i}} \ T_{i+1}^{l_{i+1}} \ T_{i+2}^{l_{i+2}} \\
		a_{1i} \ a_{1i+1} \ a_{1i+2}\\
		a_{2i} \ a_{2i+1} \ a_{2i+2}
	\end{pmatrix} \in B, \ i \in I.
\end{equation*}
For both types $R(A) = B/(g_{i}\, |\, i \in I)$.
\end{const}
\begin{de}\label{trvarde} The variety $X(A) = \mathrm{Spec}(R(A))$ is called a {\it trinomial variety}.  The type of a trinomial variety is the type of the corresponding trinomial algebra.\end{de} 

\begin{re}
It is easy to see that the dimension of $X(A)$ equals $\theta+n-r+1$.
\end{re}

\begin{re}\label{initr} By definition, 
$X(A)\cong Y(A)\times\mathbb{A}^\theta$. To obtain the algebra of regular functions on  $Y(A)$ one should eliminate generators $S_k$  from $R(A)$. The type of $Y(A)$ coincides with the type of $X(A)$.
\end{re}
\begin{re}
    In the definition of trinomial variety of Type~I sometimes it is convenient to put $n_0=0$ and $T_0^{l_0}=1$. Then we can consider all the equations as the equationa of the form $T_i^{l_i}-T_{i-1}^{l_{i-1}}-(a_{i+1}-a_i)T_0^{l_0}$.
\end{re}

One can define an effective action of an algebraic torus $\mathbb{T}\cong (\KK^\times)^{n+\theta-r}$ of complexity 1 on the variety $X(A)$. Moreover, any trinomial variety with a torus action of complexity 1 is connected with thinomial varieties via the Cox construction. Namely~\cite[Theorem~1.8]{H-W} states that the Cox ring of an irreducible, normal, rational, affine variety (and more general, $A_2$-variety) with only constant invertible functions, finitely generated divisor class group and
a torus action of complexity one is isomorphic to $R(A)$ for some $A$. 

In~\cite[Theorem~4]{E-G-S} the following criterion of rigidity of a trinomial variety is obtained.
\begin{prop}\label{EGSh}
A trinomial variety $X(A)$ is not rigid if and only if one of the following conditions holds:
\begin{enumerate}
\item[1)] $\theta \neq 0$,
\item[2)] The variety $X(A)$ has type I and there exists such an $a \in \lbrace1,\ldots,r\rbrace$, that for all $i \in \lbrace1,\ldots,r\rbrace\setminus\lbrace a\rbrace$ there is $j(i) \in\lbrace1,\ldots,n_i\rbrace$, such that $l_{ij(i)}=1$,
\item[3)] The variety $X(A)$ has type II and ane of the following conditions holds:

\begin{enumerate}
	\item[a)] There are two such numbers $a, b \in\lbrace0,\ldots,r\rbrace$, that for each $i \in\lbrace0,\ldots,r\rbrace\setminus\lbrace a,b\rbrace$ there is $j(i) \in \lbrace1,\ldots,n_i\rbrace$, such that $l_{ij(i)}=1$,
	\item[b)]  There are three such numbers $a, b, c \in\lbrace 0,\ldots,r\rbrace$, that for each  
 $$k \in\lbrace0,\ldots,r\rbrace\setminus\lbrace a,b,c\rbrace$$ there is $j(k) \in \lbrace1,\ldots,n_k\rbrace$, such that $l_{kj(k)}=1$, and for each $i \in \lbrace a,b\rbrace$ there is $v(i) \in\lbrace1,\ldots,n_i\rbrace$, such that $l_{iv(i)}=2$ and for all $w \in\lbrace1,\ldots,n_i\rbrace$ the numbers $l_{iw}$ are even.
\end{enumerate}
	
\end{enumerate}
	
\end{prop}

\subsection{$m$-suspensions}\label{secsusp}

Let $X$ be an affine algebraic variety. Given a noninvertible regular function $f\in\KK[X]$, we can define a new affine variety 
$$Y=\mathrm{Susp}(X,f)=\VV\left(uv-f(x)\right)\subset \KK^2\times X$$
called a {\it suspension} over $X$.
In~\cite{KZ} infinitely transitivity of $\mathrm{Aut}(Y)$-action on the smooth locus of $Y$ is proved in case $X\cong \mathbb{K}^n$. LNDs on suspensions are investigated in \cite{AKZ}. A suspension over a flexible affine variety is flexible; see \cite[Theorem~0.2(3)]{AKZ}.

In~\cite{Ga} the following generalization of the consept of suspension is considered.  

\begin{de}
Let $f$ be a nonzero noninvertible regular function on the affine variety~$X$. We fix positive integers $m$ and $k_1,\ldots, k_m$. By \emph{$m$-suspension} over $X$ we mean a subvariety in $\KK^m\times X$ cut by $y_1^{k_1}\cdots y_m^{k_m}-f$, where $y_1,\ldots,y_m$ are coordinate functions on~$\KK^m$. We denote this variety by $\mathrm{Susp}(X,f,k_1,\ldots,k_m)$. 
\end{de}

One can easily check that an $m$-suspension over an irreducible variety is irreducible. $m$-suspensions were investigated in \cite{Ga, G, BG}. 
We need the following easy lemma, which combine ideas of \cite{Ar}, \cite{AKZ}, \cite{Ga} and \cite{G}, but seems that it have not be formulated explicitly.  

\begin{lem}\label{lprl}
    Suppose $X$ is an irreducible affine variety and $f\in\KK[X]$. Consider $Y=\Susp(X,f,1,k_2,\ldots, k_m)$. Then for each LND $\delta$ on $X$ the derivation $\overline{\delta}$ on $Y$ given by
    $$
    \overline{\delta}(y_1)=\delta(f),\qquad \overline{\delta}(y_i)=0\text{ for }i\geq 2,$$ 
    $$\overline{\delta}(h)=y_2^{k_2}\ldots y_m^{k_m}\delta(h)\text{ for }h\in\KK[X]\subseteq\KK[Y].
    $$
    is a correct defined LND of $\KK[Y]$. Moreover, $\Quot(\Ker\overline{\delta})=\Quot(\Ker\delta)(y_2,\ldots, y_m)$.
\end{lem}
\begin{proof}
    It is easy to see that $\overline{\delta}(y_1y_2^{k_2}\ldots y_m^{k_m})=\overline{\delta}(f)$. Therefore, $\overline{\delta}$ is a correct defined derivation of $\KK[Y]$.
    Since $\overline{\delta}(y_i)=0$ for $i\geq 2$ and $\overline{\delta}(h)=y_2^{k_2}\ldots y_m^{k_m}\delta(h)$ for $h\in \KK[X]$, the derivation $\overline{\delta}$ is locally nilpotant on $\KK[X][y_2,\ldots, y_m]$. Then $\overline{\delta}(y_1)=\delta(f)\in\KK[X]$ provides locally nilpotence of $\overline{\delta}$.

    Now let $r$ be a local slice of $\delta$. Then $r$ is a local slice of $\overline{\delta}$. So, by Lemma~\ref{yadrofild}, $$\Quot(\Ker\overline{\delta})=\Quot(\pi_{\overline{\delta},r}(\KK[Y])).$$ 
    We have $\pi_{\overline{\delta},r}(y_i)=y_i$ for $i>1$. For $F\in\KK[X]$,
    $$\pi_{\overline{\delta},r}(F)=\sum_{i=0}^{\infty}\frac{(-1)^{i}\delta^i(F)y_2^{ik_2}\ldots y_m^{ik_m}r^i}{(y_2^{k_2}\ldots y_m^{k_m}\delta(r))^i}=
    \sum_{i=0}^{\infty}\frac{(-1)^{i}\delta^i(F)r^i}{(\delta(r))^i}=\pi_{\delta, r}(F).
    $$
    \begin{multline*}
    \pi_{\overline{\delta},r}(y_1)=y_1+\sum_{i=1}^{\infty}\frac{(-1)^{i}\delta^i(f)y_2^{(i-1)k_2}\ldots y_m^{(i-1)k_m}r^i}{(y_2^{k_2}\ldots y_m^{k_m}\delta(r))^i}=\\=y_1-\frac{f}{y_2^{k_2}\ldots y_m^{k_m}}+\sum_{i=0}^{\infty}\frac{(-1)^{i}\delta^i(f)y_2^{(i-1)k_2}\ldots y_m^{(i-1)k_m}r^i}{(y_2^{k_2}\ldots y_m^{k_m}\delta(r))^i}=\\=y_1-\frac{f}{y_2^{k_2}\ldots y_m^{k_m}}+
    \frac{\pi_{\delta, r}(f)}{y_2^{k_2}\ldots y_m^{k_m}}=\frac{\pi_{\delta, r}(f)}{y_2^{k_2}\ldots y_m^{k_m}}.
    \end{multline*}
    So, \begin{multline*}
    \Quot(\Ker\overline{\delta})=\Quot(\pi_{\overline{\delta},r}(\KK[Y]))\subseteq \Quot(\pi_{\delta, r}(\KK[X]))(y_2,\ldots, y_m)=\\
    =\Quot(\Ker\delta)(y_2,\ldots, y_m). 
    \end{multline*}
    The converse inclusion is obvious. 
\end{proof}

\begin{re}
    It may be $\Ker\overline{\delta}\supsetneq(\Ker\delta)[y_2,\ldots, y_m]$. If $f\in\Ker \delta$, then $y_1\in\Ker\overline{\delta}$.
\end{re}

The following lemma is a direct corollary of Lemma~\ref{lprl}.

\begin{lem}\label{lmlsunov}
    Suppose $X$ is an irreducible affine variety and $f\in\KK[X]$. Consider $Y=\Susp(X,f,1,k_2,\ldots, k_m)$. Then $\FML(Y)\subseteq \FML(X)(y_2,\ldots, y_m)$.
\end{lem}

If there are more then one unit among $k_1,\ldots, k_m$, then we can extend $\delta$ to $\overline{\delta}$ by some different ways. In particular, we are interested by the case of $k_1=\ldots=k_m=1$.
Then each $\delta\in\mathrm{LND}(\KK[X])$ gives $m$ LNDs $\overline{\delta}_1,\ldots, \overline{\delta}_m$ of $\KK[Y]$, where 
$$
\overline{\delta}_j(y_j)=\delta(f),\qquad \overline{\delta}_j(y_i)=0\text{ for }i\neq j,
$$
$$
\overline{\delta}_j(h)=y_1\ldots y_{j-1}y_{j+1}\ldots y_m\delta(h)\text{ for }h\in\KK[X]\subseteq\KK[Y].
$$

\begin{lem}\label{intker}
    Let $Y=\Susp(X,f,1,1,\ldots, 1)$. Then 
    $$
    \bigcap_{i=1}^m \Quot(\Ker\overline{\delta}_i)=\Quot(\Ker\delta)\subseteq\KK(X)\subseteq\KK(Y).
    $$
\end{lem}
\begin{proof}
We have
 $$\bigcap_{i=1}^m \Quot(\Ker\overline{\delta}_i)=\bigcap_{i=1}^m\Quot(\Ker\delta)(y_1\ldots, y_{i-1},y_{i+1},\ldots, y_m)=\Quot(\Ker\delta).$$
\end{proof}

\subsection{Isolated irreducible semi-invariants}\label{secsemiinv}

Let an algebraic torus $T$ acts on an affine variety $X$. 
Recall that a regular function $f\in\KK[X]$  is called a semi-invariant with respect to $T$ if for all $t\in T$ we have $t\cdot f=\lambda f$ for some $\lambda=\lambda (t)\in\KK^\times$. A character $\lambda\colon T\rightarrow \KK^\times$ is called a weight of $f$. We call a semi-invariant $f$ {\it irreducible} if it is irreducible as an element of $\KK[X]$. It is easy to show that a $T$-semi-invariant~$f$ is irreducible if it is not invertible and cannot be decomposed into the product of two non-invertible $T$-semi-invariants. Let us denote the lattice of $T$-characters by~$M$.

\begin{de}
We call an irreducible $T$-semi-invariant $f$ of weight $\omega$ {\it isolated} if there
exists a linear function $\alpha$ on $M_\mathbb{Q} = M \otimes_{\mathbb{Z}} \mathbb{Q}$ such that
\begin{enumerate}
    \item $\alpha(\omega)>0$;
    \item if $\alpha(\omega')>0$ for the weight $\omega'$ of an irreducible $T$-semi-invariant $h$, then $h=\gamma f$,
where $\gamma$ is an invertible $T$-semi-invariant;
    \item if $\omega''$ is the weight of an invertible $T$-semi-invariant, then $\alpha(\omega'')= 0$.
\end{enumerate}
We say that $\alpha$ is an $f$-separating function.
\end{de}

\begin{lem}\cite[Lemma~3.2]{BG}
There exist finitely many isolated irreducible $T$-semi-invariants up to scaling
by elements of $\KK[X]^\times$.
\end{lem} 

\begin{lem}\cite[Lemma~3.4]{BG}
Let $S$ be a set of $T$-semi-invariant generators of~$\KK[X]$.

(i) Assume that $f\in S$ is a non-invertible element and denote its weight by~$\omega$. Suppose that there exists a linear function $\alpha$ on $M_{\mathbb{Q}}$ such that $\alpha(\omega)>0$
 and $\alpha(\omega')\leq 0$
 for all weights $\omega'$
 of the other generators $s\in S$. Then $f$ is an irreducible isolated $T$-semi-invariant.

 (ii) If $S$ does not containes two elements $s$ and $\lambda s$ for $\lambda\in\KK[X]^\times$, then every isolated irreducible $T$-semi-invariant has the form $\gamma f$ for some $f$ satisfying (i) and 
$\gamma\in\KK[X]^\times$.
\end{lem} 

Let us denote the set of weights of isolated irreducible $T$-semi-invariants on the $T$-variety $X$ by $W(X,T)\subseteq M\cong \mathbb{Z}^n$. 

\begin{re}
Suppose we have two varieties $X_1$ and $X_2$ with $T_1$-action on $X_1$ and $T_2$-action on $X_2$, and there exist  isomophismes $\varphi\colon X_1\rightarrow X_2$ and $\psi\colon T_1\rightarrow T_2$ such that 
$\varphi(t\cdot x)=\psi(t)\cdot\varphi(x)$.
Let us identify $M_1=\mathfrak{X}(T_1)\cong \mathbb{Z}^n$ and $M_2=\mathfrak{X}(T_2)\cong \mathbb{Z}^n$ via some isomorphism.
Then there exists  such $\eta\in\mathrm{GL}_n(\mathbb{Z})$, that $$\eta(W(X_1,T_1))=W(X_2,T_2).$$ 
\end{re}

\begin{de}
    Let us call two subsets $A$ and $B$ in $\mathbb{Z}^n$ {\it equivalent} if there exists such $\eta\in\mathrm{GL}_n(\mathbb{Z})$, that $\eta(A)=B.$
\end{de}

This gives a tool for proving that two $T$-varieties are not $T$-equivariant isomorphic. In particular we obtain the following statement, which we use in Section~\ref{secconj}.

\begin{lem}\label{nesnct}
Let $T_1, T_2$ be two subtori in $\mathrm{Aut}(X)$. If the sets $W(X,T_1)$ and $W(X,T_2)$ are non-equivalent, then $T_1$ and $T_2$ are not conjugate in $\mathrm{Aut}(X)$.
\end{lem}

\begin{ex}
    Let us define two $1$-dimensional tori $T_1$ and $T_2$ in automorphism group of $$\mathrm{SL}_2=\mathrm{Spec}\left(\KK[x,y,z,w]/(xw-yz-1)\right).$$
    $$
    t_1\cdot (x,y,z,w)=(t_1x,y,z,t_1^{-1}w),\qquad
    t_2\cdot (x,y,z,w)=(t_2x,t_2y,t_2^{-1}z,t_2^{-1}w).
    $$
    We have $W(\mathrm{SL}_2,T_1)=\{1,-1\}$, $W(\mathrm{SL}_2,T_2)=\varnothing$. Therefore, $T_1$ and $T_2$ are not conjugate in $\Aut(\mathrm{SL}_2)$.
\end{ex}

\subsection{Factorial gradings}
\begin{de}
    An affine variety $X$ is said to be \textit{factorially $G$-graded} if its coordinate algebra $\KK[X]$ is equipped with a $G$-grading and every homogeneous non-zero non-unit element $f$ can be written as a product of homogeneously irreducible elements and this factorization is unique up to the order of the factors and multiplication by homogeneous units. An element $p$ is said to be {\it homogeneously irreducible} if it is not a unit and $p = a \cdot b$ for some homogeneous $a$ and $b$ implies at least one of them being a homogeneous unit.
\end{de}

Suppose that a variety $X$ admits a $G$-grading $\zeta$ (may be a trivial one) on $\KK[X]$ by an abelian group $G$ such that $\deg f=w$ for some $f\in \KK[X]$. Then, we can define a natural $(\mathbb{Z}^{m} \times G)/\langle(-k_1, \dots, -k_m, w)\rangle$-grading on $\KK[Y]$, where $Y=\Susp(X,f,k_1,\ldots, k_m)$. Indeed, consider the $\mathbb{Z}^m$-grading on $\KK^m$ with homogeneous coordinates $y_1,\ldots, y_m$ and the degrees $\deg y_j=e_j$, where $\{e_1,\ldots, e_m\}$ form a basis in $\mathbb{Z}^m$. Then $\deg(y_1^{k_1}\ldots y_m^{k_m})=(k_1,\ldots,k_m)$. Then we gather this grading $\eta$ on $\KK[y_1,\ldots, y_m]$ and $G$-grading $\zeta$ on $\KK[X]$ to obtain a $(\mathbb{Z}^{m} \times G)/\langle(-k_1, \dots, -k_m, w)\rangle$-grading on $\KK[Y]$ defined by 
$$
\deg y_j=(\deg_{\eta}y_j,0),\qquad \deg F=(0,\deg_{\zeta}F),\text{ for }F\in \KK[X]. 
$$
Since $\deg (y_1^{k_1}\ldots y_m^{k_m})=(k_1,\ldots,k_m, 0)$ and $\deg f=(0, \ldots, 0, w)$, we equate these degrees and get a $(\mathbb{Z}^{m} \times G)/\langle(-k_1, \dots, -k_m, w)\rangle$-grading on $\KK[Y]$.

\begin{lem}\label{facgrlem}
    Suppose $\KK[X]$ admits a factorial $G$-grading by an abelian group $G$ such that $\deg(f)=w$, where $f$ is a homogeneously irreducible element. Consider $Y=\Susp(X,f,k_1,\ldots, k_m)$.  Then the corresponding $(\mathbb{Z}^{m} \times G)/\langle(-k_1, \dots, -k_m, w)\rangle$-grading on $\KK[Y]$ is factorial. 
\end{lem}

\begin{proof}
Any element $F \in \KK[Y]$ can be uniquely expressed in the following normal form:
$$F = \sum_{s} h_s y_1^{s_1}\ldots y_{m}^{s_m}, \quad h_s \in \KK[X],$$
where in every summand $s_i < k_i$ for at least one index $i$.
If $F$ is homogeneous, the above expression consists of exactly one summand. Indeed, suppose two distinct summands $h_a y_1^{a_1}\dots y_m^{a_m}$ and $h_b y_1^{b_1}\dots y_m^{b_m}$ have the same degree. This implies
$$(a_1, \dots, a_m, \deg_G h_a) - (b_1, \dots, b_m, \deg_G h_b) = c \cdot ( -k_1, \dots, -k_m, w)$$
for some integer $c$. Considering the $\mathbb{Z}^m$ components, we obtain $a_j - b_j = -c k_j$ for all $j$. If $c > 0$, then $b_j = a_j + c k_j \ge k_j$ for all $j$, which contradicts the normal form condition for the second summand. Symmetrically, $c < 0$ implies $a_j \ge k_j$ for all $j$, which is also impossible. Hence, $c = 0$, which yields $a_j = b_j$ for all $j$ and $\deg_G h_a = \deg_G h_b$. Thus, any homogeneous element is of the form $h y_1^{a_1}\dots y_m^{a_m}$ where $h \in \KK[X]$ is $G$-homogeneous.

We claim that the homogeneously irreducible elements in $\KK[Y]$ are exactly the following:
\begin{itemize}
\item $G$-irreducible homogeneous elements in $\KK[X]$ that are not associated with~$f$;
\item Elements of the form $\lambda y_j$ for $\lambda \in \KK^\times$.
\end{itemize}

First, suppose $y_i = a \cdot b$ for some homogeneous $a, b \in \KK[Y]$. Expressing $a$ and~$b$ in normal form, we have $a = y_1^{a_1}\ldots y_m^{a_m} p$ and $b = y_1^{b_1} \ldots y_m^{b_m} q$. This yields $y_i = y_1^{a_1 + b_1}\ldots y_m^{a_m+b_m} pq$. By reducing this product to normal form, we obtain $y_i = y_1^{c_1}\ldots y_m^{c_m} f^l pq$ for some integer $l \ge 0$. Due to the uniqueness of the normal form over $\KK[X]$, this requires $c_i = 1$, $c_j = 0$ for all $j \neq i$, and $f^l pq = 1$. Since~$f$ is not a unit, $l = 0$. Hence, $(a_1, \ldots, a_m) + (b_1, \ldots, b_m) = (c_1, \ldots, c_m)$. This implies that for each $j$, either $a_j = c_j$ and $b_j = 0$, or vice versa. Thus, every $y_i$ is homogeneously irreducible in $\KK[Y]$.

Next, suppose $p \in \KK[X]$ is $G$-irreducible, homogeneous, and $p \neq \lambda f$. Suppose $p = a \cdot b$ for some homogeneous $a, b \in \KK[Y]$. Let $a = y_1^{a_1}\ldots y_m^{a_m} p_1$ and $b = y_1^{b_1}\ldots y_m^{b_m} p_2$ in normal form. Since the degrees of $p$ and $a \cdot b$ are equal, we get:
$$(0, \dots, 0, \deg_G p) - ( a_1 + b_1, \dots, a_m + b_m, \deg_G(p_1 p_2)) = c \cdot (-k_1, \dots, -k_m, w)$$
for some integer $c$. Looking at the $\mathbb{Z}^m$ coordinates, we have $-(a_j + b_j) = -c k_j$, so $a_j + b_j = c k_j$. Since $a_j, b_j \ge 0$, we must have $c \ge 0$. This implies $p = f^c p_1 p_2$ in $\KK[X]$. Because $p$ is homogeneously irreducible in $\KK[X]$ and $p \neq \lambda f$, we must have $c = 0$, which forces $a_j = b_j = 0$ for all $j$. We thus obtain the decomposition $p = p_1 p_2$ in $\KK[X]$. Since $p$ is irreducible in $\KK[X]$, either $p_1$ or $p_2$ is a unit, proving that $p$ remains irreducible in $\KK[Y]$.
Finally, we prove unique factorization. Take a homogeneous element $F \in \KK[Y]$. Any factorization of $F$ into homogeneous irreducible elements takes the form:
$$F = \lambda \cdot y_1^{s_1}\ldots y_{m}^{s_m} \cdot p_1 \ldots p_r$$
where $p_i \in \KK[X]$ are irreducible. Putting $a_i = s_i - l k_i$, where $l$ is the maximal integer such that $a_i \ge 0$ for all $i$, we can convert this to the normal form:
$$F = \lambda \cdot y_1^{a_1} \ldots y_m^{a_m} \cdot f^l \cdot p_1 \ldots p_r$$

To show uniqueness, assume another representation in normal form:
$$\lambda \cdot y_1^{a_1} \ldots y_m^{a_m} \cdot f^{l_1}\cdot p_1 \ldots p_r = \mu \cdot y_1^{b_1} \ldots y_m^{b_m} \cdot f^{l_2}\cdot q_1 \ldots q_s$$
The uniqueness of the normal form over $\KK[X]$ implies that $a_j = b_j$ for all $j$ and
$$\lambda \cdot f^{l_1}\cdot p_1 \ldots p_r = \mu \cdot f^{l_2}\cdot q_1 \ldots q_s$$
in $\KK[X]$. Since $\KK[X]$ is factorially $G$-graded, we conclude $l_1 = l_2$ and the sets of homogeneous irreducibles $\{p_i\}$ and $\{q_j\}$ coincide up to permutation and multiplication by elements of $\KK^\times$. This finishes the proof.
\end{proof}

If $w=\deg f = 0\in G$, then 
$$(\mathbb{Z}^{m} \times G)/\langle(-k_1, \dots, -k_m, w)\rangle\cong \mathbb{Z}_d \times \mathbb{Z}^{m-1} \times G,$$ where $d = \gcd(k_1, \ldots, k_m)$. The following corollary is straightforward.
\begin{cor}
In conditions of Lemma~\ref{facgrlem} consider two $m$-suspensions $Y=\Susp(X,f,k_1,\ldots, k_m)$ and $Z=\Susp(X,f,s_1,\ldots, s_r)$. Put $d=\gcd(k_1,\ldots, k_m)$ and $l=\gcd(s_1,\ldots, s_r)$. Then $\KK[Y]$ is $(\mathbb{Z}^{m-1}\times \mathbb{Z}_d\times G)$-factorial if and only if $\KK[Z]$ is $(\mathbb{Z}^{r-1}\times \mathbb{Z}_l\times G)$-factorial.
\end{cor}

Suppose $\KK[X]$ admits a factorial $G$-grading by an abelian group $G$ such that $\deg(f)=w$, where $f$ is a nonzero homogeneous element. Consider $Y=X\times \KK^m$.  Introduce  $(\mathbb{Z}^{m} \times G)/\langle(-k_1, \dots, -k_m, w)\rangle$-grading on $\KK[Y]=\KK[X][z_1,\ldots,z_m]$ by
$\deg z_i=e_i$. 

\begin{lem}\label{sfar}
     The considered $(\mathbb{Z}^{m} \times G)/\langle(-k_1, \ldots, -k_m, w)\rangle$-grading on $\KK[Y]$ is factorial and $h=f+z_1^{k_1}\ldots z_m^{k_m}$ is a homogeneously irreducible element. 
\end{lem}
\begin{proof}
Denote $v=z_1^{k_1}\ldots z_m^{k_m}$. It is easy to see that each homogeneous element has the form 
$
z_1^{a_1}\ldots z_m^{a_m}g(v),
$
where there exists $a_j<k_j$ and $g(t)$ is a polynomial in $\KK[X][t]$ with homogeneous coefficients. By degrees arguments $a_i$ is uniquely defined. Since $z_j$ is homogeneously irreducible, we need only to prove that the algebra $\KK[X][t]$ is factorially graded by $\mathbb{Z}\times G/\langle(-1,w)\rangle$. 

Consider the localization $L$ of $\KK[X]$ in all non-zero homogeneous elements. Then $L[t]$ is a Euclidean domain, and hence, it is factorial. Using this by standard arguments similar to Gauss's lemma one can prove that $\KK[X][t]$ is homogeneously factorial. 

Since $h(t)=f+t$ is irreducible in $L[t]$ and primitive, it is irreducible in $\KK[Y]$. 
    
\end{proof}

The $m$-suspension technique naturally extends on trinomial varieties. Indeed, each trinomial variety $Z$ can be contained in the sequence of varieties $Z_0=\KK^{n_0+n_1+\theta}$, $Z_1,\ldots, Z_{k-1}=Z$, where $Z_{j+1}=\Susp(Z_{j},f_j, l_{j1},\ldots, l_{j_n})$, where $f_j=\lambda_j T_0^{l_0}+\mu_jT_1^{l_1}$. By Lemma~\ref{sfar}, $\KK^{n_0+n_1+\theta}$ is factorially graded by 
$$\mathbb{Z}^{n_0+n_1}/\langle(-k_{01},\ldots,-k_{0n_0},k_{11},\ldots, k_{1n_1})\rangle,$$ 
and all $f_i$ are homogeneously irreducible. As a consequence of Lemma \ref{facgrlem}, we obtain an alternative proof of a well-known result due to Hausen and Wrobel; see \cite[Theorem~1.2]{H-W}, see also \cite[Theorem~1.4]{H-H}. Let $X$ be a trinomial variety. Define the finest grading $\eta$ for which all $T_{ij}$ and $S_i$ are homogeneous. This grading is obtained by equating of weights of all monomials $T_i^{l_i}$; see \cite[Construction~1.1]{H-W} for details.
\begin{cor}\label{comi}
       Let $X$ be a trinomial variety. Then the grading $\eta$ on $\KK[X]$ is factorial.
\end{cor}

\section{Kernels of LNDs}\label{kol}
In this section we elaborate a technique to compute kernels of LNDs in some particular situation. The main idea is as follows. Suppose we are given by an irreducible affine variety $X$. Assume that $X$ is given by a system of equations such that a variable $x$ occurs in equations only as part of the product $xy^2$. Then we consider an LND on $\KK[X][u]$ of the form $D=y\frac{\partial}{\partial u}+\delta$, where $\delta$ is an LND on $X$. Then in some rather wide situation $\Ker D$ is finitely generated and $\Spec(\Ker D)$ is given by the same equations as $X$, in which the product $xy^2$ is replaced by $xy$.

It is known that a kernel of an LND $D$ may be not finitely generated. In case when we know that $\Ker D$ is finitely generated, one can use van den Essen Kernel Algorithm to explicitly compute generators of the kernel, see~\cite[Section~1.4]{Ess}. Moreover, if this algorithm finishes working, the kernel is finitely generated.  Let us give a description of this algorithm.

\begin{alg}
    Suppose $B=\KK[X]=\KK[b_1,\ldots, b_n]\cong\KK[x_1,\ldots, x_n]/(h_1,\ldots, h_l)$. Let $D$ be an LND on $X$. Fix a local slice $r$ of $D$ and $p=D(r)\in\Ker D$. 

1) Let $c_1=\pi_r(b_1),\ldots, c_n=\pi_r(b_n)\in(\Ker D)_p$. We can take such nonnegative integers $d_1,\ldots, d_n$ that $a_i=p^{d_i}c_i\in \Ker D$. Define $A_0=\KK[a_1,\ldots, a_n]$. 

2) For each $i\geq 1$ having $A_{i-1}$ we can compute $A_i=\{g\in B\mid gp\in A_{i-1}\}$. If $A_i=A_{i-1}$, then $\Ker D$ is finitely generated and $\Ker D=A_{i}$. 

\end{alg}

\begin{re}
    In a practical situation, it is advantageous to choose $d_i$ as small as possible. However, the optimal choice of these parameters only affects the number of algorithm steps.
\end{re}

So, if we would like to check whether $\Ker D=A_0$, we should check whether $A_1=A_0$. 
Let us formulate more explicitly what one should check to prove $\Ker D=A_0$.

\begin{alg}
1) Let us add new variables $v_1,\ldots, v_n$ and fix lexicographic order with $$x_1>\ldots>x_n>v_1>\ldots>v_n$$ 
on monomials in $\KK[x_1,\ldots, x_n,v_1,\ldots, v_n]$.

2) Let $a_j=g_j(b_1,\ldots, b_n)$, and $p=f(b_1,\ldots, b_n)$, where $g_j,f\in\KK[x_1,\ldots, x_n]$. Find a Gr\"obner basis $\mathfrak{F}$ of the ideal 
$$J=(v_1-g_1,\ldots, v_n-g_n, h_1,\ldots, h_l, f)\subseteq\KK[x_1,\ldots, x_n,v_1,\ldots, v_n],$$
where $g_j$, $h_i$, and $f$ are polynomials in $x_1,\ldots, x_n$.

3) Let $\{Q_1,\ldots, Q_k\}=\mathfrak{F}\cap \KK[v_1,\ldots, v_n]$. Then $\Ker D=A_0$ if and only if $$Q_i(a_1,\ldots, a_n)p^{-1}\in \KK[a_1,\ldots, a_n]\text{ for all } 1\leq i\leq k.$$
\end{alg}

We would like to apply this algorithm to compute kernels of some particular LNDs. 

Let us fix some settings which we use further.

\begin{set}[$*$]
Let $C=\KK[x,y,z_1,\ldots, z_k]/(h_1,\ldots, h_l)$ be a domain, where for each $1\leq j\leq l$ 
$$
h_j(x,y,z_1,\ldots,z_k)=\overline{h}_j(xy^2,y,z_1,\ldots,z_k).
$$ 
We denote $Z=\mathrm{Spec}\,C$.
Suppose $\delta$ is an LND of $C$ such that $y\nmid \delta(x)=g(y,z_1,\ldots, z_k)$, $\delta(y)=0$, $y^2\mid \delta(z_j)\in\KK[y,z_1,\ldots, z_k]$ for all $1\leq j\leq l$. Let
$B=C[u]$, $X=\mathrm{Spec}\,B$ and $D=y\frac{\partial}{\partial u}+\delta$. We mean that $D$ is a derivation such that for $c\in C$ we have $D(c)=\delta(c)$ and $D(u)=y$. It is clear that $D$ is an LND. 
\end{set}

In these settings $u$ is a local slice for $D$. And we can compute
$\pi_u(u)=0$, $\pi_u(y)=y$, $\widetilde{z_j}=\pi_u(z_j)=z_j+yw_j$ for some $w_j\in B$,
$$\pi_u(x)=x-\frac{gu}{y}+\frac{\sum_{i=1}^k\frac{\partial g}{\partial z_i}\delta(z_i)u^2}{2! y^2}-\ldots=\frac{1}{y}(gu+yw)=\frac{1}{y}\widetilde{x},\qquad w\in B.
$$
So, we can take $A_0=\KK[\widetilde{z_1}, \ldots, \widetilde{z_k},\widetilde{x},y].$
Now we have
\begin{multline*}
J=(v_1-\widetilde{z_1},\ldots, v_k-\widetilde{z_k}, v_{k+1}-\widetilde{x}, v_{k+2}-y, h_1,\ldots, h_l, y)\subseteq\\
\subseteq\KK[x,y,z_1,\ldots,z_k, v_1,\ldots, v_{k+2}].
\end{multline*}
Since $y\in J$ and $y\mid(\widetilde{z_j}-z_j)$, we can replace $\widetilde{z_j}$ by $z_j$. Analogically, we can replace $\widetilde{x}$ by $gu$. So,
$$
J=(v_1-z_1,\ldots, v_k-z_k, v_{k+1}-gu, v_{k+2}, h_1,\ldots, h_l, y).
$$
Let us recall that $h_j(x,y,z_1,\ldots, z_k)=\overline{h_j}(xy^2,y,z_1,\ldots, z_k)$. Denote $\widehat{h_j}(z_1,\ldots, z_k)=\overline{h_j}(0,0,z_1,\ldots, z_k)$. Also denote $\widehat{g}(z_1,\ldots, z_k)=g(0,z_1,\ldots, z_k)$. Then 
$$
J=(v_1-z_1,\ldots, v_k-z_k, v_{k+1}-\widehat{g}(v_1,\ldots, v_k)u, v_{k+2}, \widehat{h_1}(v_1,\ldots, v_k),\ldots, \widehat{h_l}(v_1,\ldots, v_k), y).
$$

\begin{lem}\label{pl1}
    In settings $(*)$ let us define 
    $$
    I=(v_{k+1}-\widehat{g}(v_1,\ldots, v_k)u, \widehat{h_1}(v_1,\ldots, v_k),\ldots, \widehat{h_l}(v_1,\ldots, v_k))\subseteq \KK[v_1,\ldots, v_{k+1},u].
    $$
   If
   $ I\cap \KK[v_1,\ldots, v_{k+1}]=(\widehat{h_1}(v_1,\ldots, v_k),\ldots, \widehat{h_l}(v_1,\ldots, v_k)),$
    then 
    $$\Ker D=A_0=\KK[\widetilde{z_1}, \ldots, \widetilde{z_k},\widetilde{x},y].$$
    
\end{lem}
\begin{proof}
    Let us show that if 
    $I\cap \KK[v_1,\ldots, v_{k+1}]=(\widehat{h_1}(v_1,\ldots, v_k),\ldots, \widehat{h_l}(v_1,\ldots, v_k))$,
    then
    $J\cap \KK[v_1,\ldots, v_{k+2}]=(\widehat{h_1}(v_1,\ldots, v_k),\ldots, \widehat{h_l}(v_1,\ldots, v_k), v_{k+2}).$
    To compute generators of the latest ideal, we should compute a Gr\"obner basis $\mathfrak{F}$ for $J$ with respect to the lexicographic order with
    $$
    x>y>z_1>\ldots>z_k>u>v_1>\ldots>v_{k+2},
    $$
    and then remove all elements of $\mathfrak{F}$ depending on $x,y,z_1,\ldots, z_k,u$. To compute $\mathfrak{F}$ one can use Buchberger's algorithm. But it is well-known fact that if leading terms of $F$ and $G$ are coprime, then $s$-polynomial of the pair $(F,G)$ need not be considered in Buchberger's algorithm. Since the leading term
    $\mathrm{LT}(v_i-z_i)=-z_i$ and 
    $$\mathrm{LT}(v_{k+1}-\widehat{g}(v_1,\ldots, v_k)u),\mathrm{LT}(\widehat{h_l}(v_1,\ldots, v_k))\in \KK[v_1,\ldots, v_{k-1},u],$$ 
    we have $\mathfrak{F}=\mathfrak{G}\cup \{v_1-z_1,\ldots, v_{k}-z_k,v_{k+2},y\}$, where $\mathfrak{G}$ is a Gr\"obner basis for $I$.
    If 
    $$I\cap \KK[v_1,\ldots, v_{k+1}]=(\widehat{h_1}(v_1,\ldots, v_k),\ldots, \widehat{h_l}(v_1,\ldots, v_k)),$$ we have 
    $$J\cap \KK[v_1,\ldots, v_{k+2}]=(\widehat{h_1}(v_1,\ldots, v_k),\ldots, \widehat{h_l}(v_1,\ldots, v_k), v_{k+2}).$$ 

    Now 
    \begin{multline*}    
    \widehat{h_j}(\widetilde{z_1},\ldots, \widetilde{z_k})=\overline{h_j}(0,0,\widetilde{z_1},\ldots, \widetilde{z_k})=\overline{h_j}(\widetilde{x}y,y,\widetilde{z_1},\ldots, \widetilde{z_k})+yQ_i(\widetilde{x},y,\widetilde{z_1},\ldots, \widetilde{z_k})=\\
    =\overline{h_j}(\pi_u(xy^2),y,\pi_u(z_1),\ldots, \pi_u(z_k))+yQ_i(\widetilde{x},y,\widetilde{z_1},\ldots, \widetilde{z_k})=\\
    =\pi_u(\overline{h_j}(xy^2,y,z_1,\ldots, z_k))+yQ_i(\widetilde{x},y,\widetilde{z_1},\ldots, \widetilde{z_k})=\pi_u(0)+yQ_i(\widetilde{x},y,\widetilde{z_1},\ldots, \widetilde{z_k})=\\=yQ_i(\widetilde{x},y,\widetilde{z_1},\ldots, \widetilde{z_k}).
    \end{multline*}
    Since $v_{k+2}$ gives $y=y\cdot 1$, by the above algorithm $\Ker D=A_0$.
\end{proof}

Let us denote $\widetilde{h_j}(x,y,z_1,\ldots, z_k)=\overline{h_j}(xy,z_1,\ldots, z_k)=h_j\left(\frac{x}{y},y,z_1,\ldots, z_k\right)$. Then $\pi_u(h_j(x,y,z_1,\ldots, z_k))=\widetilde{h_j}(\widetilde{x},y,\widetilde{z_1},
,\ldots,\widetilde{z_k})$. Denote also 
$$
\widetilde{C}=\KK[x,y,z_1,\ldots,z_k]/(\widetilde{h_1}(x,y,z_1,\ldots, z_k),\ldots,\widetilde{h_l}(x,y,z_1,\ldots, z_k)).
$$

\begin{de}
In conditions $(*)$ suppose 
\begin{itemize}
    \item $I\cap \KK[v_1,\ldots, v_{k+1}]=(\widehat{h_1}(v_1,\ldots, v_k),\ldots, \widehat{h_l}(v_1,\ldots, v_k))$
    \item $y\in \widetilde{C}$ is not a zero devisor.
\end{itemize}
Then we say that we are in conditions $(**)$. 
\end{de}

By Lemma~\ref{pl1}, in conditions $(**)$ we have $\Ker D=A_0=\KK[\widetilde{z_1}, \ldots, \widetilde{z_k},\widetilde{x},y]$. It would be useful to find not only generators of $\Ker D$, but also the relations. 

\begin{lem}\label{pl2}
    In conditions of $(**)$, we have
    $$
    \Ker~D=\KK\left[\widetilde{x}, y,\widetilde{z_1},\ldots,\widetilde{z_k}\right]/\left(\widetilde{h_i}\left(\widetilde{x},y,\widetilde{z_1},\ldots,\widetilde{z_k}\right)\mid 1\leq i\leq l\right).
    $$
\end{lem}
\begin{proof}
By Proposition~\ref{ssnf}(c), we have $\Ker \pi_u=uB_y\cap B=uB$. So, $\pi_u$ defines an embedding $\pi_u|_C\colon C\rightarrow (\Ker D)_y$. Hence, 
$$
\mathrm{Im}\, \pi_u|_C=\KK\left[\frac{\widetilde{x}}{y}, y,\widetilde{z_1},\ldots,\widetilde{z_k}\right]/\left(\overline{h_i}\left(\widetilde{x}y,y,\widetilde{z_1},\ldots,\widetilde{z_k}\right)\mid 1\leq i\leq l\right).
$$
We know that, 
$$\Ker D=\mathrm{Im}\,\pi_u\cap B=\KK\left[\widetilde{x}, y,\widetilde{z_1},\ldots,\widetilde{z_k}\right]/\mathcal{P}$$
for some ideal $\mathcal{P}\subseteq \KK\left[\widetilde{x}, y,\widetilde{z_1},\ldots,\widetilde{z_k}\right]$. We need to show that 
$$\mathcal{P}=\left(\widetilde{h_i}\left(\widetilde{x},y,\widetilde{z_1},\ldots,\widetilde{z_k}\right)\mid 1\leq i\leq l\right).$$
Take $f\in \mathcal{P}$. Then $f=\sum p_i\left(\frac{\widetilde{x}}{y}, y,\widetilde{z_1},\ldots,\widetilde{z_k}\right) \widetilde{h_i}(\widetilde{x},y,\widetilde{z_1},\ldots,\widetilde{z_k})$. Therefore, for some positive integer $n$, we obtain
$fy^n=\sum q_i\left(\widetilde{x}, y,\widetilde{z_1},\ldots,\widetilde{z_k}\right) \widetilde{h_i}(\widetilde{x},y,\widetilde{z_1},\ldots,\widetilde{z_k}).$
That is $fy^n=0$ in $$\KK[\widetilde{x},y,\widetilde{z_1},\ldots,\widetilde{z_k}]/(\widetilde{h_1}(\widetilde{x},y,\widetilde{z_1},\ldots, \widetilde{z_k}),\ldots,\widetilde{h_l}(\widetilde{x},y,\widetilde{z_1},\ldots, \widetilde{z_k}))\cong\widetilde{C}.$$
Since $y$ is not zero divisor in $\widetilde{C}$, we obtain $f\in \left(\widetilde{h_i}\left(\widetilde{x},y,\widetilde{z_1},\ldots,\widetilde{z_k}\right)\mid 1\leq i\leq l\right)$. Thus,
$$\mathcal{P}\subseteq\left(\widetilde{h_i}\left(\widetilde{x},y,\widetilde{z_1},\ldots,\widetilde{z_k}\right)\mid 1\leq i\leq l\right).$$
The inverse inclusion is obvious.
\end{proof}
\begin{re}
    In conditions ($**$) we have $\widetilde{C}\cong\Ker D$. Hence, $\widetilde{C}$ is a domain.  
\end{re}
\begin{de}
    We say that conditions $(*')$ hold if conditions $(*)$ hold and $\widetilde{C}$ is a domain.
\end{de}

In many examples considered in this paper, the following conditions are satisfied. 

\begin{lem}\label{pl5}
    In conditions $(*')$ suppose that the ideal $(\widehat{g},\widehat{h_1},\ldots, \widehat{h_l})$ in $\KK[z_1,\ldots,z_k]$ is not proper. Then we are in conditions $(**)$.
\end{lem}
\begin{proof}
    Since $1\in (\widehat{g},\widehat{h_1},\ldots, \widehat{h_l})$, there exist such polynomials $P_1,\ldots, P_l, Q\in \KK[z_1,\ldots,z_k]$, that $1=P_1\widehat{h_1}+\ldots+P_l\widehat{h_l}-Q\widehat{g}$.
    The element 
\begin{multline*} 
\sum_{j=1}^l P_j(v_1,\ldots, v_k)\widehat{h_j}(v_1,\ldots, v_k)u+Q(v_1,\ldots, v_k)(v_{k+1}-\widehat{g}(v_1,\ldots, v_k)u)=\\
=Q(v_1,\ldots, v_k)v_{k+1}+u
\end{multline*}
belongs to $I$. Moreover, 
    $$v_{k+1}-\widehat{g}(v_1,\ldots, v_k)u=v_{k+1}\sum_{j=1}^l P_j(v_1,\ldots, v_k)\widehat{h_j}-
    \widehat{g}(v_1,\ldots, v_k)(Q(v_1,\ldots, v_k)v_{k+1}+u).$$
    This implies $I=(\widehat{h_1}(v_1,\ldots, v_k),\ldots, \widehat{h_l}(v_1,\ldots, v_k), Q(v_1,\ldots, v_k)v_{k+1}+u)$. 

    Applying Buchberger's algorithm we obtain a Gr\"obner basis~$\mathfrak{G}$ of $I$ with respect to the  lexicographic order with
    $$
    u>v_1>\ldots>v_{k+2}.
    $$
    Then~$\mathfrak{G}=\{Q(v_1,\ldots, v_k)v_{k+1}+u\}\cap \mathfrak{G}'$, where $\mathfrak{G}'$ is a Gr\"obner basis of $$(\widehat{h_1}(v_1,\ldots, v_k),\ldots, \widehat{h_l}(v_1,\ldots, v_k))\subseteq \KK[v_1,\ldots, v_{k}].$$
    So, $$I\cap \KK[v_1,\ldots,v_{k+1}]=(\widehat{h_1}(v_1,\ldots, v_k),\ldots, \widehat{h_l}(v_1,\ldots, v_k)).$$
\end{proof}

\begin{ex}\label{pe1}
    Let $Z$ be a Danielewski surface 
    $$Z=\{xy^n=f(z)\}, \qquad n\geq 2, f\text{ has not multiple roots}.$$ 
    We can define $\delta\colon (x,y,z)\to (f'(z),0,y^n)$. It is easy to see that conditions $(*)$ are satisfied. Since $f$ has not multiple roots, we have $\gcd (f(v_1),f'(v_1))=1$.
    Therefore, by Lemmas~\ref{pl5} and~\ref{pl2} we have
    $$
    \Ker D=\KK[\widetilde{x},y,\widetilde{z}]/(\widetilde{x}y^{n-1}-f(\widetilde{z})).
    $$
    Here $\widetilde{z}=z-uy^n$ and $\widetilde{x}=yx-uf'(z)+\frac{1}{2}u^2f''(z)y^{n-1}-\frac{1}{6}u^3f'''(z)y^{2n-2}+\ldots$
\end{ex}

\begin{re}
    The condition that $f$ does not have multiple roots is indeed necessary. For example, if $Z=\{xy^2=z^2\}$, we have
    $
    I=(v_2-2v_1u, v_1^2). 
    $
    Let us compute the s-polynomial 
    $$
    s(v_2-2v_1u,v_1^2)=v_1(v_2-2v_1u)+2uv_1^2=v_1v_2.
    $$
    Therefore, $I\cap \KK[v_1,v_2]\neq (v_1^2)$. So, we can not apply Lemma~\ref{pl1}. By algorithm we obtain that $\frac{\widetilde{x}\widetilde{z}}{y}\in\Ker D$, so $\Ker D\supsetneq \KK[\widetilde{x},y,\widetilde{z}]$. 
\end{re}

The following proposition is a generalization of Example~\ref{pe1}.
\begin{prop}\label{p1p}
    Let $V$ be an irreducible affine variety with 
$$\KK[V]=\KK[z_1,\ldots, z_{k}]/(h_1(z_1,\ldots, z_{k}),\ldots, h_{l}(z_1,\ldots, z_{k})).$$ Suppose $\xi$ is an LND of $\KK[V]$ and $f\in \KK[V]$ such that $\xi(f)\neq 0$. Suppose that there exist $p$ and $q$ in $\KK[V]$ such that $pf+q\xi(f)=1$. Denote 
$$
Z=\Susp(V,f,1,k_1,\ldots, k_m),\qquad k_1\geq 2.
$$
I.e.
$
\KK[Z]=\KK[x,y_1,\ldots, y_m,z_1,\ldots, z_k]/(h_1,\ldots, h_{l+1})
$, where $$h_{l+1}=xy_1^{k_1}\ldots y_m^{k_m}-f.
$$
Denote $D=y_1\frac{\partial}{\partial u}+\delta$, where $\delta=\overline{\xi}$; see Lemma~\ref{lprl}. Then we have
$$
\Ker D=\KK[\widetilde{x},y_1,\ldots, y_m, \widetilde{z_1},\ldots, \widetilde{z_k}]/H,$$
where
$$
H=(h_1(\widetilde{z_1},\ldots,\widetilde{z_k}),\ldots, h_{l}(\widetilde{z_1},\ldots,\widetilde{z_k}), \widetilde{x}y_1^{k_1-1}\ldots y_m^{k_m}-f(\widetilde{z_1},\ldots,\widetilde{z_k})).
$$
\end{prop}
\begin{proof}
    Conditions $(*')$ are satisfied. In notation of the settings $(*)$ we have $y=y_1$, $z_{k+1}=y_2,\ldots, z_{k+m-1}=y_m$. We obtain $\delta(x)=\xi(f)$. Therefore,
    $$1\in (\xi(f), \widehat{h_1},\ldots,\widehat{h_{l}}, f)=
    (\widehat{g}, \widehat{h_1},\ldots,\widehat{h_{l+1}}).
    $$
    By Lemma~\ref{pl5}, conditions $(**)$ are satisfied.  
Therefore, by Lemmas~\ref{pl1} and~\ref{pl2} we have
$$
\Ker D=\KK[\widetilde{x},y_1,\ldots, y_m, \widetilde{z_1},\ldots, \widetilde{z_k}]/H.$$
\end{proof}

\begin{cor}\label{yadco}
    In conditions of Proposition~\ref{p1p}, we have
    $$
    \Spec{\Ker D}=\Susp(V,f,1,k_1-1,k_2,\ldots,k_m).
    $$
\end{cor}

\begin{ex}\label{pe2}
Let $Z$ be an $m$-suspension $Z=\Susp(\KK,f(z),1,k_1,\ldots, k_m)$, where $k_1\geq 2$, and $f$ has no multiple roots. I.e. 
$
Z=\{xy_1^{k_1}\ldots y_m^{k_m}=f(z)\}.
$
Applying Proposition~\ref{p1p} for $V=\KK$ and $\xi=\frac{\partial}{\partial z}$, we obtain 
    $$
    \Ker D=\KK[\widetilde{x},y_1,\ldots, y_m,\widetilde{z}]/(\widetilde{x}y_1^{k_1-1}y_2^{k_2}\ldots y_m^{k_m}-f(\widetilde{z})).
    $$
\end{ex}

\begin{ex}
    Let $V=\KK[a,b,c,d]/(a(c^3+d^5)^2-b^2-1)$. Consider $$\xi\colon (a,b,c,d)\rightarrow (2b,(c^3+d^5)^2,0,0)\in \LND(V).$$ 
    Denote $Z=\Susp(V,a,1,2)$. I.e.
    \begin{multline*}
    \KK[Z]=\KK[x,y,a,b,c,d]/(xy^2-a,a(c^3+d^5)^2-b^2-1)=\\
    =\KK[x,y,b,c,d]/(xy^2(c^3+d^5)^2-b^2-1).
    \end{multline*}
    If we take $p=(c^3+d^5)^2$, $q=-b$, then $pa+q\xi(a)=1$. As usually we take $X=Z\times\KK$, $\delta=\overline{\xi}\colon (x,y,b,c,d)\rightarrow (2b,0,(c^3+d^5)^2y^2,0,0)$ and $D=y\frac{\partial}{\partial u}+\delta$ Hence, by Proposition~\ref{p1p}, 
    $$
    \Ker D=\KK[\widetilde{x},y,\widetilde{b},c,d]/(\widetilde{x}y(c^3+d^5)^2-\widetilde{b}^2-1).
    $$
\end{ex}

\begin{cor}\label{pnc}
Let $W$ be an irreducible affine variety with 
$$\KK[W]=\KK[z_1,\ldots, z_{k-1}]/(h_1(z_1,\ldots, z_{k-1}),\ldots, h_{l-1}(z_1,\ldots, z_{k-1})).$$
Let $V=W\times \KK$ and $z_{k}$ be the additional coordinate. Consider $f=\alpha z_k^d+1\in\KK[V]$, where $\alpha\in\KK[W]$. Now denote
$
Z=\Susp(V,f,1,k_1,\ldots, k_m), k_1\geq 2.
$
I.e.
$$
\KK[Z]=\KK[x,y_1,\ldots, y_m,z_1,\ldots, z_k]/(h_1,\ldots, h_l), \text{ where }h_l=xy_1^{k_1}\ldots y_m^{k_m}-f.
$$
Define $\delta$ by $\delta(x)=d\alpha z_k^{d-1}$, $\delta(z_k)=y_1^{k_1}\ldots y_m^{k_m}$, $\delta(y_i)=\delta(z_j)=0, j<k$.
Then for $D=y_1\frac{\partial}{\partial u}+\delta$ we have
$$
\Ker D=\KK[\widetilde{x},y_1,\ldots, y_m, z_1,\ldots, z_{k-1},\widetilde{z_k}]/(h_1,\ldots, h_{l-1}, \widetilde{x}y_1^{k_1-1}\ldots y_m^{k_m}-\alpha\widetilde{z_k}^d-1).
$$
\end{cor}
\begin{proof}
 Let us take $p=d$ and $q=-z_k$. Denote $\xi=\frac{\partial}{\partial z_k}\in\mathrm{LND}(\KK[V])$. Then we have
 $$pf+q\xi(f)=d(\alpha z_k^d+1)+(-z_k)d\alpha z_k^{d-1}=1.$$
 Therefore, we are under the conditions of Proposition~\ref{p1p}. This proposition implies the assertion.
\end{proof}

Now we would like to introduce one more class of varieties, which we call {\it Multi Danielewski varieties} or {\it $n$-Danielewski varieties}. This class generalizes $m$-suspensions with $k_1=1$,  Double Danielewski surfaces, see~\cite{GS}, and Danielewski varieties, see \cite{Du} and \cite{Ga}. 
\begin{de}\label{multid}
    Let $Z$ be an irreducible variety. Let us cionsider  $(n+m)$-dimensional affine space with coordinates $x_1,\ldots, x_n, y_1,\ldots, y_m$. Also, we fix polymonials $f_1,\ldots, f_n$, where $f_j\in\KK[Z][x_1,\ldots, x_{j-1}, y_1,\ldots, y_m]$ and an $n\times m$-matrix $K=(k_{ij})$, $k_{ij}\in\mathbb{Z}_{\geq 0}$.
    Now we consider the variety $Y\subseteq \KK^{n+m}\times Z$ given by 
    \begin{equation}\label{nm}
    \begin{cases}
       x_1y_1^{k_{11}}\ldots y_{m}^{k_{1m}}=f_1;\\
       x_2y_1^{k_{21}}\ldots y_{m}^{k_{2m}}=f_2;\\
       \vdots\\
       x_ny_1^{k_{n1}}\ldots y_{m}^{k_{nm}}=f_n.
    \end{cases}
    \end{equation}
    The variety $Y$ we call an {\it $n$-Danielewski variety}. Let us use the following notation $Y=\mathrm{Dan}_n(Z,K,f_1,\ldots,f_n)$. 

    Let us consider $f_j$ as a polynomial in $x_{j-1}$ with coefficients in $$\KK[x_1,\ldots, x_{j-2}, y_1,\ldots, y_m].$$ We will assume that all $f_j$ are monic as polynomials in $x_{j-1}$.
\end{de}

We can include $Y$ on the following chain of varieties 
$$Y_0=\KK^m\times Z, Y_1,\ldots, Y_n=Y,$$ 
where $Y_i$ is given in $\KK^{i+m}\times Z$ by the equations from the first to the $i$-th in the system (\ref{nm}).

Let $Z$ be an irreducible affine variety with 
$$\KK[Z]=\KK[z_1,\ldots, z_{k}]/(h_1(z_1,\ldots, z_{k}),\ldots, h_{l}(z_1,\ldots, z_{k})).$$ 
We denote 
$$h_{l+j}=x_jy_1^{k_{j1}}\ldots y_m^{k_{jm}}-f_j(x_1,\ldots, x_{j-1}, y_1,\ldots, y_m,z_1,\ldots,z_k),\qquad 1\leq j\leq n-1.$$

\begin{lem}\label{snl318}
    The algebra
    $$\KK[Y]=\KK[z_1,\ldots,z_k,x_1,\ldots,x_n,y_1,\ldots,y_m]/(h_1,\ldots,h_{l+n})
    $$
    is a domain.
\end{lem}
\begin{proof}
    We would like to prove Lemma~\ref{snl318} by induction on $n$. When $n=0$, $\KK[Y]=\KK[Z]$ is a domain by the condition of the lemma.

    Now let us suppose that we have proved that 
$$R=\KK[Y_j]=\KK[z_1,\ldots,z_k,x_1,\ldots,x_j,y_1,\ldots,y_m]/(h_1,\ldots,h_{l+j})$$ is a domain.
    Let us use the same technique as in~\cite{GS}.

    The following lemma is \cite[Lemma~2.4(2)]{DGO}, see also \cite[Lemma~3.2]{GS}.
    \begin{lem}\label{vspomlem1}
    Let $R$ be an integral domain and $a,b\in R\setminus\{0\}$. If $b$ is not a zero-divisor on $R/(a)$, then the ring $R[t]/(bt-a)$ is an integral domain. 
    \end{lem}
    Also let us recall \cite[Lemma~3.2]{GS}.
    \begin{lem}\label{vspomlem2}
        Let $R$ be an integral domain and $a, b\in R\setminus\{0\}$. If $a$ is not a zero-divisor on $R/(b)$, then $b^n$ is not a zero-divisor on $R/(a)$ for any integer $n\geq 1$.
    \end{lem}
    Suppose $k_{j+1,i}\neq 0$. Then we put  $a=f_{j+1}\in\KK[Y_j]$, $b=y_i$. We have 
    $$R/(b)\cong \left(\KK[Y_{j-1}]/(y_i, f_j)\right)[x_j].$$ Since $f_{j+1}$ is monic polynomial in $x_j$, it is not a zero-divizor on $R/(b)$. By Lemma~\ref{vspomlem2}, $y_j$ is not a zero-divisor on $\KK[Y_j]/(f_{j+1})$. Since it holds for all $i$ such that $k_{j+1,i}\neq 0$, we have $y_1^{k_{j+1,1}}\ldots y_m^{k_{j+1,m}}$ is not a zero-divisor on $\KK[Y_j]/(f_{j+1})$. Now we put $a=f_{j+1}$, $b=y_1^{k_{j+1,1}}\ldots y_m^{k_{j+1,m}}$, $t=x_{j+1}$ and apply Lemma~\ref{vspomlem1}. We obtain 
    $$
    \KK[Y_j][x_j]/(y_1^{k_{j+1,1}}\ldots y_m^{k_{j+1,m}}x_{j+1}-f_{j+1})
    $$
    is a domain.
\end{proof}

Let $\xi$ be an LND on $Z$. We define $\delta_0\in\LND(\KK[Y_0])$ by $\delta_0(y_j)=0$ and $\delta_0(F)=\xi(F)$ for $F\in\KK[Z]$. Then we can define an LND $\delta_i$, $1\leq i\leq n$ of $Y_i$ by
$$
\delta_i(x_i)=\delta_{i-1}(f_i), \qquad \delta_i(F)=y_1^{k_{i1}}\ldots y_{m}^{k_{im}}\delta_{i-1}(F)\text{ for }F\in \KK[Y_{i-1}].
$$

For any $F\in\KK[Z][x_1,\ldots, x_n, y_1,\ldots, y_m]$ we denote by $\widehat{F}$ the polynomial $$F|_{y_1=0}\in \KK[Z][x_1,\ldots, x_n, y_2,\ldots, y_m].$$

\begin{prop}\label{P323}
Suppose $k_{n1}\geq 2$, 
and the ideal  
$$\mathfrak{J}=\left(\widehat{\xi(f_1)}\prod_{j=1}^{n-1}\widehat{\frac{\partial f_{j+1}}{\partial x_j}}, \widehat{f_1},\ldots, \widehat{f_{n}}\right)\subseteq \KK[Z][x_1,\ldots,x_{n-1},y_2,\ldots,y_m]$$ is not proper. 

Then for the LND $D=y_1\frac{\partial}{\partial u}+\delta_n$ on $X=Y_n\times\KK$ we have
$$
\Ker D=\KK[\widetilde{x_1},\ldots, \widetilde{x_n},y_1,\ldots, y_m, \widetilde{z_1},\ldots, \widetilde{z_k}]/H,
$$
where the ideal $H$ is generated by $$h_i(\widetilde{z_1},\ldots,\widetilde{z_k}),\qquad 1\leq i\leq l,$$ 
$$\widetilde{x_j}y_1^{k_{j1}}\ldots y_m^{k_{jm}}-f_j(\widetilde{x_1},\ldots, \widetilde{x}_{j-1}, y_1,\ldots, y_m,\widetilde{z_1},\ldots,\widetilde{z_k}),\qquad 1\leq j\leq n-1,$$ 
and
$$
\widetilde{x_n}y_1^{k_{n1}-1}y_2^{k_{n2}}\ldots y_m^{k_{nm}}-f_n(\widetilde{x_1},\ldots, \widetilde{x_{n-1}}, y_1,\ldots, y_m,\widetilde{z_1},\ldots,\widetilde{z_k}).
$$
\end{prop}
\begin{proof}
    In settings $(*)$ we define $\delta=\delta_n$, $y=y_1$, $x=x_n$,
    $$z_{k+1}=y_2,\ldots, z_{k+m-1}=y_m, z_{k+m}=x_1,\ldots, z_{k+m+n-2}=x_{n-1}.$$

    Since $f_i$ is a monic polynomial in $x_{i-1}$, we have $y\nmid \frac{\partial f_i}{\partial x_{i-1}}$. It is easy to prove by induction that $y\nmid \delta_i(x_i)$. Since $k_{n1}\geq 2$, we have $y^2\mid \delta_n(x_j)$ for all $j<n$ and $y^2\mid\delta_n(z_i)$. So, all conditions $(*)$ are satisfied. 

    Let us denote $h_{l+j}=x_jy_1^{k_{j1}}\ldots y_m^{k_{jm}}-f_j$ and $\widehat{f_j}=f_j\mid_{y_1=0}=-\widehat{h_{l+j}}$.

 It is easy to see that $\widehat{g}=\delta_n(x_n)=\widehat{\xi(f_1)}\prod_{j=1}^{n-1}\widehat{\frac{\partial f_{j+1}}{\partial x_j}}$. Therefore, the ideal $(\widehat{g}, \widehat{h_1},\ldots, \widehat{h_{l+n}})$ is not proper in $\KK[z_1,\ldots, z_k,x_1,\ldots, x_n, y_1,\ldots, y_m]$. By Lemmas~\ref{pl5}, \ref{pl1}, and \ref{pl2}, we obtain the goal. 
\end{proof}
\begin{re}\label{rmdv}
    Let us denote the following ideals 
    $$J_1=(\widehat{\xi(f_1)}, \widehat{f_1})\subseteq \KK[Z][y_2,\ldots, y_m];$$
    $$J_i=\left(\widehat{f_1},\ldots,\widehat{f_i}, \widehat{\frac{\partial f_i}{\partial x_{i-1}}}\right)\subseteq \KK[Z][x_1\ldots, x_{i-1},y_2,\ldots, y_m], \qquad 2\leq i\leq n-1.$$
    If all the ideals $J_1,\ldots, J_{n-1}$ are not proper, then the ideal $\mathfrak{J}$ from Proposition~\ref{P323} is not proper. 
\end{re}

Let us give some sufficient conditions for $(**)$ other than Lemma~\ref{pl5}.

\begin{lem}\label{pl3}
    Suppose that for $g=\delta(x)$ the following condition holds: $y\mid gw$ for some $w\in C$ implies $y\mid w$. Then in settings $(*')$ we have  
    $$\Ker D=A_0=\KK[\widetilde{z_1}, \ldots, \widetilde{z_k},\widetilde{x},y]/\left(\overline{h_i}\left(\widetilde{x}y,y,\widetilde{z_1},\ldots,\widetilde{z_k}\right)\mid 1\leq i\leq l\right).$$
\end{lem}
\begin{proof}
    Let us consider the ideal 
    $$
    I=(v_{k+1}-g(v_1,\ldots, v_k)u, \widehat{h_1}(v_1,\ldots, v_k),\ldots, \widehat{h_l}(v_1,\ldots, v_k))\subseteq \KK[v_1,\ldots, v_{k+1},u]
    $$
    from Lemma~\ref{pl1}. Assume that $$I\cap\KK[v_1,\ldots, v_{k+1}]\neq (\widehat{h_1}(v_1,\ldots, v_k),\ldots, \widehat{h_l}(v_1,\ldots, v_k)).$$ Let us apply Buchberger algorithm to compute a Gr\"obner basis of $I$. When we compute the s-plynomial of two polynomials $v_{k+1}a(v_1,\ldots, v_k)+b(v_1,\ldots, v_k)u$ and $c(v_1,\ldots, v_k)$, we obtain a polynomial of the form $v_{k+1}a(v_1,\ldots, v_k)+b_1(v_1,\ldots, v_k)u$. When we compute the s-plynomial of two polynomials 
    $$v_{k+1}a_1(v_1,\ldots, v_k)+b_1(v_1,\ldots, v_k)u$$ and $$v_{k+1}a_2(v_1,\ldots, v_k)+b_2(v_1,\ldots, v_k)u,$$ we obtain a polynomial of the form 
    $v_{k+1}a_3(v_1,\ldots, v_k)+b_3(v_1,\ldots, v_k)u$. If we apply a reduction of a polynomial $$v_{k+1}a(v_1,\ldots, v_k)+b(v_1,\ldots, v_k)u$$
    by a plynomial $c(v_1,\ldots, v_k)$ or 
    $v_{k+1}a_1(v_1,\ldots, v_k)+b_1(v_1,\ldots, v_k)u$, we obtain a polynomial of the form $v_{k+1}a(v_1,\ldots, v_k)+b(v_1,\ldots, v_k)u$. Therefore, all elements of the obtained Gr\"obner basis that are not in $(\widehat{h_1}(v_1,\ldots, v_k),\ldots, \widehat{h_l}(v_1,\ldots, v_k))$ have the form $v_{k+1}a(v_1,\ldots, v_k)+b_1(v_1,\ldots, v_k)u$. 

    If we have an element of this Gr\"obner basis contained in
    $$\KK[v_1,\ldots, v_{k+1}]\setminus (\widehat{h_1}(v_1,\ldots, v_k),\ldots, \widehat{h_l}(v_1,\ldots, v_k)),$$ 
    then it has the form $F=v_{k+1}a(v_1,\ldots, v_k)$.

    Then, according to the algorithm, we are interested in $$F(\widetilde{z_1},\ldots, \widetilde{z_k}, \widetilde{x})=\widetilde{x}a(\widetilde{z_1},\ldots, \widetilde{z_k})=gua(z_1,\ldots, z_k)+yS,$$ which is divisible by $y$. Therefore, $y\mid ga(z_1,\ldots, z_k)$. But, since the condition of the lemma, this implies that $y$ divides $a(z_1,\ldots, z_k)$ in $C$. That is $a(z_1,\ldots, z_k)\in J$, i.e. $a(v_1\ldots, v_{k})\in I$. Hence, $a(v_1\ldots, v_{k})\in (\widehat{h_1}(v_1,\ldots, v_k),\ldots, \widehat{h_l}(v_1,\ldots, v_k))$. This implies $$F\in(\widehat{h_1}(v_1,\ldots, v_k),\ldots, \widehat{h_l}(v_1,\ldots, v_k)),$$ which contradicts the assumption. This shows that we do not have elements of the Gr\"obner basis of $I$ in 
    $$\KK[v_1,\ldots, v_{k+1}]\setminus (\widehat{h_1}(v_1,\ldots, v_k),\ldots, \widehat{h_l}(v_1,\ldots, v_k)).$$ 
    So,
    $$I\cap\KK[v_1,\ldots, v_{k+1}]= (\widehat{h_1}(v_1,\ldots, v_k),\ldots, \widehat{h_l}(v_1,\ldots, v_k)).$$ 
    By Lemmas~\ref{pl1} and~\ref{pl2} we obtain the assertion.
\end{proof}

\begin{re}\label{pr2}
    The conditions of Lemma~\ref{pl3} are satisfied, for example, if $C$ is a UFD. But this is not a necessary condition. 
    
    Suppose a $G$-grading on $C$ is fixed, where $G$ is an abelian group. Suppose also that $C$ is known to be $G$-factorial. Let $y$ and $g$ be $G$-homogeneous elements such that $\gcd(y,g)=1$. Then $y\mid gw$ implies $y\mid w$. 
\end{re}

\begin{ex}\label{prtrnkr}
    Let $Z=\{xy^2=z_1^2+z_2^3\}$. By \cite[Theorem~1.2(iv)]{H-W} this variety is factorial and $y$ is an irreducible element. Therefore by Lemma~\ref{pl3} and Remark~\ref{pr2} for 
    $$D\colon C[u]\rightarrow C[u],\qquad D(x,y,z_1,z_2,u)=(2z_1,0,y^2,0,y)$$ 
    we have 
    $$
    \Ker D=\KK[\widetilde{x},y,\widetilde{z_1},z_2]/(\widetilde{x}y-\widetilde{z_1}^2-z_2^3).
    $$
    Here $\widetilde{x}=xy-2z_1u+yu^2$, $\widetilde{z_1}=z_1-uy$.
\end{ex}

\begin{ex}
    Let $Z=\{xy^2=z_1^3+z_2^3\}$. Then $C=\KK[Z]$ admits the following $G=\mathbb{Z}\times \mathbb{Z}_3$-grading given by $\deg x=(2,0)$, $\deg y=(-1,0)$, $\deg z_1=(0,1)$, $\deg z_2=(0,0)$. By \cite[Theorem~1.2(ii, iii)]{H-W}, see also Corollary~\ref{comi}, this variety is $G$-graded and $y$ is a $G$-irreducible element. Therefore by Lemma~\ref{pl3} and Remark~\ref{pr2} for $D\colon C[u]\rightarrow C[u]$, $D(x,y,z_1,z_2,u)=(3z_1^2,0,y^2,0,y)$ we have 
    $$
    \Ker D=\KK[\widetilde{x},y,\widetilde{z_1},z_2]/(\widetilde{x}y-\widetilde{z_1}^3-z_2^3).
    $$
    Here $\widetilde{x}=xy-3z_1^2u+3z_1u^2y-u^3y^2$, $\widetilde{z_1}=z_1-uy$.
\end{ex}

\begin{de}\label{nd316}
    Suppose $Z_0=\Susp(V,f,1,k_1,\ldots, k_m)$, $k_1\geq 2$ and $\xi$ is an LND on $V$ such that $\xi(f)\neq 0$. Consider $\delta_0=\overline{\xi}$. Let us define 
    $$Z=\{xy_1^{k_1}\ldots y_m^{k_m}+y_1F=f\}\subseteq \KK^{m+1}\times V,$$ 
    where $F\in \Ker\!\xi\,[y_1,\ldots,y_m]$. 
    The variety $Z$ we call a {\it modification} of $Z_0$.
\end{de}

Note that we can define an LND $\delta$ on $Z$ by the same formulas as $\delta_0$, i.e.
$$\delta(x)=\xi(f),\qquad \delta(y_i)=0, \qquad \delta(g)=y_1^{k_1}\ldots y_m^{k_m}\xi (g) \text{ for } g\in\KK[V].$$ Since $y_1F\in\Ker\!\xi\,[y_1,\ldots, y_m]\subseteq \Ker \delta$, these formulas give a correct LND on $\KK[Z]$.

\begin{lem}\label{pl4}
    Suppose for the LND $\delta_0$ on $Z_0$  the conditions $(**)$ are satisfied. Then for the LND $\delta$ on $Z$  these conditions are also satisfied.
\end{lem}
\begin{proof}
        It is easy to see that conditions $(*')$ are satisfied for both varieties. By definition, ideals $I$ for $\delta$ and $\delta_0$ coincide. 
\end{proof}






\begin{ex}
    Let $Z=\{xy_1^{k_1}\ldots y_m^{k_m}+y_1F(y_1,\ldots, y_m)=f(z)\}$, where $k_1\geq 2$ and $f$ has no multiple roots. This variety is obtained by modification from the variety from Example~\ref{pe2}. Therefore, for the same 
    $\delta\colon(x,y_1,\ldots, y_m,z)\rightarrow (f'(z), 0,\ldots, 0, y_1^{k_1}\ldots y_m^{k_m})$, the conditions $(**)$ are satisfied. So, for $D=y_1\frac{\partial}{\partial u}+\delta$ we have
    $$
    \Ker D=\KK[\widetilde{x}, y_1,\ldots, y_m,\widetilde{z}]/(\widetilde{x}y_1^{k_1-1}y_2^{k_2}\ldots y_m^{k_m}+y_1F(y_1,\ldots, y_m)-f(\widetilde{z})).
    $$
\end{ex}

\begin{ex}\label{KKRC}
Let us consider the Koras-Russel cubic $Z=\{xy^2+y=z^2+w^3\}$. It is obtained by a modification of the following $2$-suspension over the affine plane $Z_0=\{xy^2=z^2+w^3\}$, which is a trinomial hypersurface. By Example~\ref{prtrnkr} for $\delta_0\colon(x,y,z,w)\rightarrow (2z,0,y^2,0)\in \LND(Z_0)$ conditions $(**)$ are satisfied. So, for  $D_0=y\frac{\partial}{\partial u}+\delta_0\in \LND(Z_0\times\KK)$ we have 
$$
\Ker D_0=\KK[\widetilde{x}, y,\widetilde{z}, w]/(\widetilde{x}y-\widetilde{z}^2-w^3).
$$
By Lemma~\ref{pl4} for $\delta\colon(x,y,z,w)\rightarrow (2z,0,y^2,0)\in \LND(Z)$ conditions $(**)$ are satisfied. Putting
 $D=y\frac{\partial}{\partial u}+\delta\in \LND(Z\times\KK)$ 
we obtain
$$
\Ker D=\KK[\widetilde{x}, y,\widetilde{z}, w]/(\widetilde{x}y+y-\widetilde{z}^2-w^3)=
\KK[\overline{x}, y,\widetilde{z}, w]/(\overline{x}y-\widetilde{z}^2-w^3),
$$
where $\overline{x}=\widetilde{x}+1$.
\end{ex}

\section{Constructing counter-examples for the generalized Zariski cancellation problem}\label{chetvert}

Let us remind the statements of the Zariski cancellation problem and the generalized Zariski cancellation problem.

\begin{prob}[Zariski cancellation problem]
    Let $X$ be an irreducible affine algebraic variety. Does $X\times \KK\cong \KK^{n+1}$ imply
    $X\cong \KK^n$?
\end{prob}

\begin{prob}[generalized Zariski cancellation problem]
    Let $X$ and $Y$ be irreducible affine algebraic varieties. Does $X\times \KK\cong Y\times \KK$ imply
    $X\cong Y$?
\end{prob}

The Zariski cancellation problem is still open in case $\mathrm{char}\,\KK=0$, while there exists some counter-examples for the generalized Zariski cancellation problem. The first and the most known counter-example give Danielewski surfaces introduced in~\cite{Da}. In \cite{Du} Du\-bou\-loz construct counter-examples in orbitrary dimension $\geq 2$. Also it worth to mention that in positive characteristic there is a counter-example for the Zariski cancellation problem, see~\cite{Gu}.

A variety $X$ is {\it cancellative} if $X\times \KK\cong Y\times \KK$ imply $X\cong Y$ for any $Y$, i.e. if $X$ is not a counter-example to the generalized Zariski cancellation problem.

It seems to be a popular idea to use the Slice Theorem to prove that Danielewski surfaces give counter-examples for the generalized Zariski cancellation problem, see \cite[Section~10.1.1]{F}. To prove that $X\times \KK\cong Y\times \KK$ one needs to find an LND $D$ of $\KK[X\times\KK]$ with a slice such that $\Ker D\cong\KK[Y]$. We consider LNDs of the form $y\frac{\partial}{\partial u}+\delta$ as in the previous section and use the technique elaborated in the previous section to compute kernels of such LNDs. 

\begin{prop}\label{2pp1}
    In settings $(*)$, let $\sigma=F(y,z_1,\ldots, z_k)\in\Ker D$ for some polynomial $F\in\KK[y,z_1,\ldots, z_k]$. Then $\sigma\in\mathrm{pl}(D)$ if and only if
$$F\in(y,\widehat{g},\widehat{h_1},\ldots,\widehat{h_l})\subseteq \KK[y,z_1,\ldots, z_k].$$
\end{prop}
\begin{proof}
Suppose $\sigma\in\mathrm{pl}(D)$. Then 
    $$\sigma\in (D(u),D(x),D(y),D(z_1),\ldots, D(z_k))\subseteq \KK[X].$$  But $D(u)=y$, $D(y)=0$, and $y\mid D(z_i)$. Therefore, $$(D(u),D(x),D(y),D(z_1),\ldots, D(z_k))=(y, g(z_1,\ldots, z_k))\subseteq \KK[X].$$
    The condition $\sigma\in (y, g(z_1,\ldots, z_k))\subseteq \KK[X]$ is equivalent to 
    $$
    F\in (y, g,h_1,\ldots, h_l)=(y, \widehat{g},\widehat{h_1},\ldots, \widehat{h_l})\subseteq \KK[x,y,z_1,\ldots, z_k].
    $$
    Since $F\in\KK[y,z_1,\ldots,z_k]$, we have $F\in (y, \widehat{g},\widehat{h_1},\ldots, \widehat{h_l})\subseteq \KK[y,z_1,\ldots, z_k].$

    Conversely, suppose $F\in (y,\widehat{g},\widehat{h_1},\ldots,\widehat{h_l})\subseteq \KK[y,z_1,\ldots, z_k]$. Then there exists $S,P,Q_1,\ldots, Q_l\in \KK[y,z_1,\ldots, z_k]$ such that $F=Sy+P\widehat{g}+Q_1\widehat{h_1}+\ldots+Q_l\widehat{h_l}$.
We have $\widetilde{x}=ug+ya$, $\widetilde{z_i}=z_i+yb_i$. Hence,
$\widetilde{P}=\pi_u(P)=P+yc$, $\widetilde{Q_i}=\pi_u(Q_i)=Q_i+yd_i$.
Therefore,
\begin{multline*}
y\sigma=D(u\sigma)=D(u(Sy+P\widehat{g}+Q_1\widehat{h_1}+\ldots+Q_l\widehat{h_l}))=\\
=D(Syu+Pu\widehat{g}+u\widehat{h_1}Q_1+\ldots+uQ_l\widehat{h_l})=\\
=D\left(\widetilde{P}\widetilde{x}+yr+u\sum_{i=1}^l h_iQ_i\right)=D(yr)=yD(r)
\end{multline*}
for some $r\in\KK[X]$.
Therefore, $\sigma=D(r)$.
\end{proof}

Let us state two straightforward corollaries of Proposition~\ref{2pp1}.

\begin{cor}\label{2pc1}
 In settings $(*)$, let $\sigma=f(z_1,\ldots, z_k)\in\Ker D$ for some polynomial $f\in\KK[z_1,\ldots, z_k]$. Then $\sigma\in\mathrm{pl}(D)$ if and only if
$f\in(\widehat{g},\widehat{h_1},\ldots,\widehat{h_l})\subseteq \KK[z_1,\ldots, z_k]$.
\end{cor}

\begin{cor}\label{2pc2}
    In settings $(*)$, the LND $D$ admits a slice if and only if the ideal $(\widehat{g},\widehat{h_1},\ldots,\widehat{h_l})\subseteq \KK[z_1,\ldots, z_k]$ is not proper, i.e. $(\widehat{g},\widehat{h_1},\ldots,\widehat{h_l})= \KK[z_1,\ldots, z_k]$ .
\end{cor}

Remain that in settings $(*)$  denote $\widetilde{h_i}=\overline{h}_j(xy,y,z_1,\ldots,z_k)$, $1\leq i\leq l$. Let us put
$$\KK[Y]=\widetilde{C}=\KK[x,y,z_1,\ldots, z_k]/(\widetilde{h_1},\ldots, \widetilde{h_l}).$$ 
Applying Lemma~\ref{pl5}, Corollary~\ref{2pc2}, Lemma~\ref{pl1}, and Lemma~\ref{pl2}, we obtain the following theorem.
\begin{theor}\label{slt}
    In settings $(*')$, if the ideal $(\widehat{g},\widehat{h_1},\ldots,\widehat{h_l})\subseteq \KK[z_1,\ldots, z_k]$ is not proper, then $Z\times\KK\cong Y\times\KK$.
\end{theor}

Now we give some examples of varieties $Z$ and $Y$ such that $Z\times \KK\cong Y\times \KK$. Later we prove for some pairs $Z$ and $Y$ that they are not isomophic. 

\begin{cor}\label{nt0}
    Let $V$ be an irreducible affine variety and
let $\xi$ be an LND of $\KK[V]$. Fix $f\in \KK[V]$ a nonzero noninvertible function. Suppose that there exist $p$ and $q$ in $\KK[V]$ such that $pf+q\xi(f)=1$. Denote 
$$
Z=\Susp(V,f,1,k_1,\ldots, k_m),\qquad k_i\in\mathbb{Z}_{>0}.
$$
and $U$ is an $(m+1)$-suspension
$
U=\Susp(V,f,1,1,\ldots, 1).
$
Then $Z\times\KK\cong U\times \KK$.
\end{cor}
\begin{proof}
    Suppose $k_1\geq 2$. Obviously, it is sufficient to prove that $Z\times \KK\cong Y\times \KK$, where $Y=\Susp(V,f,1,k_1-1,\ldots, k_m)$. 

Consider
 $\delta=\overline{\xi}\in\LND(Z)$; see Lemma~\ref{lprl}. Then conditions $(*')$ are satisfied for $x=y_1$, $y=y_2$; and
    we have $g=\xi(f)$, $\widehat{h_1}=f$. Then Theorem~\ref{slt} implies $Z\times\KK\cong Y\times\KK$.
\end{proof}

\begin{cor}\label{nc0}
Suppose $W$ is an irreducible affine variety. Consider 
$$f=\alpha z^n+1\in\KK[W\times \KK],$$ where $\alpha\in\KK[W]$ and $z$ is the coordinate on $\KK$. Then 
    $$
    \Susp(W\times \KK,f,1,k_1,\ldots, k_m)\times \KK\cong \Susp(W\times \KK,f,1,1,\ldots, 1)\times \KK
    $$
    for any $k_j\in\mathbb{Z}_{>0}$.
\end{cor}
\begin{proof}
  It was shown in the proof of Corollary~\ref{pnc} that for $f=\alpha z^n+1$ the condition $pf+q\xi(f)=1$ is satisfied for $p=n$ and $q=-z$.
\end{proof}

The following statement is proved in another way in~\cite{Du}.

\begin{cor}\label{nc1}
    Consider $f(z)\in\KK[z]$ without multiple roots. Then 
    $$
    \Susp(\KK,f(z),1,k_1,\ldots, k_m)\times \KK\cong \Susp(\KK,f(z),1,1,\ldots, 1)\times \KK
    $$
    for any $k_j\in\mathbb{Z}_{>0}$.
\end{cor}
\begin{proof}
We can consider $\xi=\frac{\partial}{\partial z}$. Since $f(z)$ has no multiple roots, there exist $p$ and~$q$ in $\KK[z]$ such that $pf+q\xi(f)=pf+qf'=1$. Then the statement follows from Corollary~\ref{nt0}.
\end{proof}

\begin{ex}\label{e1}
 Let $Z_n=\Susp(\KK,f(z),1,n)$, $n\geq 2$, and $f$ has no multiple roots. So $Z$ is the Danielewski surface $\{xy^n=f(z)\}$. By Corollary~\ref{nc1}
 we have 
 $Z_n\times \KK\cong Z_{1}\times \KK$.
\end{ex}

\begin{ex}\label{e2}
Let us consider the trinomial hypersurface $Z$ of Type I given in $\KK^{n_1+n_2}$ by the equation
$$
T_{11}T_{12}^{l_{12}}\ldots T_{1n_1}^{l_{1n_1}}=T_2^{l_2}+1.
$$
We have $Z=\Susp(\KK^{n_2},T_2^{l_2}+1,1,l_{21},\ldots, l_{1n_1})$. Note that $f=T_2^{l_2}+1$ has the form $\alpha T_{21}^{l_{21}}+1$, where $\alpha=T_{22}^{l_{22}}\ldots T_{2n_2}^{l_{2n_2}}\in\Ker\frac{\partial}{\partial T_{21}}$. Therefore, by Corollary~\ref{nc0}
$$
\{T_{11}T_{12}^{l_{12}}\ldots T_{1n_1}^{l_{1n_1}}=T_2^{l_2}+1\}\times \KK\cong \{T_{11}T_{12}\ldots T_{1n_1}=T_2^{l_2}+1\}\times \KK.
$$
\end{ex}

Let $V$ be an irreducible affine variety and $\xi$ be an LND on it. Consider $Z_0=\Susp(V,f,1,k_1,\ldots, k_m)$. We can consider $\delta_0=\overline{\xi}\in\LND(Z_0)$. Now let
 $Z$ be a modification of $Z_0$, see Definition~\ref{nd316}, i.e. 
 $$Z=\{xy_1^{k_1}\ldots y_m^{k_m}+y_1F=f\}\subseteq \KK^{m+1}\times V,$$
 where $F\in\Ker\xi[y_1,\ldots,y_m]$. We obtain $\delta\in LND(Z)$. Denote $D=y_1\frac{\partial}{\partial u}+\delta\in\LND(Z\times \KK)$. 

\begin{lem}\label{stmodl}
    Suppose the LND $D_0=y\frac{\partial}{\partial u}+\overline{\xi}$ on $Z_0\times \KK$ has a slice. Then the LND $D$ also admits a slice. And $Z\times \KK\cong Y\times\KK$, where  
    $$Y=\{xy_1^{k_1-1}\ldots y_m^{k_m}+y_1F=f\}\subseteq \KK^{m+1}\times V.$$
\end{lem}
\begin{proof}
Note that images $\mathrm{Im}\,D$ and $\mathrm{Im}\,D_0$ are contained in $\KK[V][y_1,\ldots,y_m]$. And the images of all the generators for $D$ and $D_0$ coincide.  Therefore, $\mathrm{Im}\,D=\mathrm{Im}\,D_0$. This proves the first part.   

By Corollary~\ref{2pc2}, since $D$ admits a slice, the ideal $(\widehat{g}, \widehat{h})=(\xi(f),f)$ is not proper in $\KK[V][y_1,\ldots, y_m]$. Therefore, by Theorem~\ref{slt} we obtain $Z\times \KK\cong Y\times\KK$.
\end{proof}
\begin{ex}
    Let $Z=\{xy^2+y=z^2-1\}$, $Y=\{xy+y=z^2-1\}\cong \{xy=z^2-1\}$. Then $Z\times\KK\cong Y\times \KK$. 
\end{ex}

\begin{ex}\label{emdv}
    Let $X$ be a Multi Danielewski variety $\Dan(Z,K,f_1,\ldots, f_n)$, see Definition~\ref{multid}. In notations of Proposition~\ref{P323} if $k_{n1}\geq 2$ and the ideal  
$$\mathfrak{J}=\left(\widehat{\xi(f_1)}\prod_{j=1}^{n-1}\widehat{\frac{\partial f_{j+1}}{\partial x_j}}, \widehat{f_1},\ldots, \widehat{f_{n}}\right)\subseteq \KK[Z][x_1,\ldots,x_{n-1},y_2,\ldots,y_m]$$ is not proper, then by Theorem~\ref{slt}, $X\times\KK\cong Y\times \KK$, where $Y=\Dan(Z,P,f_1,\ldots, f_n)$, where
$$
\begin{cases}
    p_{ij}=k_{ij}, \text{ for } (i,j)\neq (n,1);\\
    p_{n1}=k_{n1}-1.
\end{cases}
$$
\end{ex}
Let us give an explicit example of application of Example~\ref{emdv}.
\begin{ex}
Let $z$ be the coordinate on the line $\KK$. Then
   \begin{multline*}
   \Dan_3\left(\KK,\begin{pmatrix}1&2\\3&4\\5&6\end{pmatrix},z^2-1, x_1^2-1, x_2^2-z^2x_1^2\right)\cong
    \\
    \cong \Dan_3\left(\KK,\begin{pmatrix}1&2\\3&4\\5&5\end{pmatrix},z^2-1, x_1^2-1, x_2^2-z^2x_1^2\right).
\end{multline*}
Indeed, for $\xi=\frac{\partial}{\partial z}$, we have 
$$\mathfrak{J}=(8zx_1x_2,z^2-1,x_1^2-1,x_2^2-z^2x_1^2)=\KK[z,x_1,x_2,x_3,y_1,y_2].$$
\end{ex}

\begin{re}
    The case of Double Danielewski surfaces, i.e. 2-Danielewski varieties $\Dan\left(\KK,\begin{pmatrix}d\\e\end{pmatrix}, f_1,f_2\right)$, was investigated in  \cite{GS}. In our notations, in this paper $\xi=\frac{\partial}{\partial z}$, where $z$ is the coordinate on $\KK$.  It was proven that $$ X\times\KK\cong \Dan\left(\KK,\begin{pmatrix}d\\e-1\end{pmatrix}, f_1,f_2\right)\times \KK,$$ 
    if $e\geq 2$ and ideals $(\widehat{\xi(f_1)},\widehat{f_1})\subseteq \KK[z]$ and $(\widehat{f_1},\widehat{f_2},\widehat{\frac{\partial f_2}{\partial x_1}})\subseteq \KK[x_1,z]$ are not proper, see \cite[Theorem~3.14]{GS}. As we have noticed in Remark~\ref{rmdv}, these conditions are sufficient to conditions $e\geq 2$ and $\mathfrak{J}$ is not proper. So, the result of \cite[Theorem~3.14]{GS} coincides with the result of Example~\ref{emdv} in the particular case of Double Danielewski surfaces. 
\end{re}

Now let us proceed with showing for some pairs of varieties $X$ and $Y$ such that $X\times \KK\cong Y\times\KK$ by the above arguments, that $X\ncong Y$. Such pairs are counter-examples for generalized Zariski cancellation problem. We start with some known examples. 
\begin{ex}
The first and the most popular non-cancellative varieties are Danielewski surfaces. Let $Z_n=\{xy^n=f(z)\}$, where $f$ has no multiple roots. Then it is known, see Example~\ref{e1}, that $Z_n\times \KK\cong Z_1\times \KK$. To show that $Z_n$ is a  non-cancellative variety, it is sufficient to prove that $Z_1\ncong Z_2$. It was mentioned in the original work due to Danielewski~\cite{D}, and was proved in~\cite{Fi}. See also~\cite{ML01} for a proof based on describing Makar-Limanov invariant of these varieties. 
\end{ex}

\begin{ex}\label{e4}
In \cite{Du} the varieties of the form 
$$
\{xy_1^{k_1}\ldots y_m^{k_m}=P(z)\},
$$
where $\deg P\geq 2$ and $P$ has $\deg P$ distinct roots, were considered. They are called {\it Danielewski varieties}. In \cite{Du}, two Danielewski varieties
$
X=\{xy_1^{k_1}\ldots y_m^{k_m}=P(z)\}
$
and
$
X'=\{xy_1^{k'_1}\ldots y_m^{k'_m}=P(z)\}
$
are shown to be non-isomorphic if $\{k_1,\ldots, k_m\}\in \mathbb{Z}_{>1}^m$ and $\{k'_1,\ldots, k'_m\}\in \mathbb{Z}_{\geq 1}^m\setminus \mathbb{Z}_{>1}^m$, but $X\times \KK\cong X'\times \KK$. This gives couner-examples for generalized Zariski cencellation problem in arbitrary dimension. The isomorphism $X\times \KK\cong X'\times \KK$ follows also from Corollary~\ref{nc1}.

Note that results of \cite{Ga} imply that if $ \{k_1,\ldots, k_m\}\neq\{k_1',\ldots, k_m'\}$, then $X$ and~$X'$ are not isomorphic. Indeed, let $\varphi\colon \KK[X]\rightarrow \KK[X']$ be an automorphism. From~\cite{Ga} follows that the $m$-dimensional torus~$T$ acting by multiplication of $x$, $y_1,\ldots y_m$ by characters is a maximal torus in automorphism group $\Aut(X)$. From the explicit description of $\Aut(X)$ obtained in~\cite{Ga} it follows that all maximal tori of $\Aut(X)$ are conjugate to this one by an exponents of LNDs. Therefore, restrictions to $\ML(X)$ of all maximal tori in $\Aut(X)$ coincide. Hence, all functions $y_i$ are semi-invariant with respect to any torus action on the variety $X$. Since these are all such functions in $\ML(X)$ up to multiplication by constants, $\varphi^*$ takes each $y_i\in\KK[X]$ to $\lambda_iy_j\in\KK[X']$. But conjugation by $\varphi$ should move $T$-action to $T'$-action. This is possible only if $\varphi$ takes $y_i$ to $\lambda_iy_j$ with $k_i=k_j'$. 
\end{ex}
\begin{ex}\label{e6}
In \cite[Example~2.7]{DF} it is proved that the varieties
$$
X=\{xy^2=zw^2+1\}
$$
and
$$
Y=\{xy^2=zw^2+yw+1\}
$$
are not isomorphic. Note that $X=\Susp(\KK^2,zw^2+1,1,2)$. We can take $Z_0=X$ and $\delta_0=\overline{\frac{\partial}{\partial z}}\in\LND(Z_0)$. I.e. $\delta_0\colon(x,y,z,w)\rightarrow (w^2,0,y^2,0)$. By Corollary~\ref{nc0}, $X\times \KK\cong\{xy=zw^2+1\}\times\KK$. In particular, $D_0=y\frac{\partial}{\partial u}+\delta_0$ admits a slice. Since $yw\in\Ker\delta_0$, $Y$ is a modification of $Z_0$. By Lemma~\ref{stmodl}, we obtain 
$$Y\times\KK\cong \{xy=zw^2+yw+1\}\times\KK=\{(x-w)y=zw^2+1\}\times\KK\cong \{xy=zw^2+1\}\times\KK.$$
Thus $X\times\KK\cong Y\times \KK$.
\end{ex}

\begin{ex}\label{examtri}
As we have shown in Example~\ref{e2}, 
    $$
\{T_{11}T_{12}^{l_{12}}\ldots T_{1n_1}^{l_{1n_1}}=T_2^{l_2}+1\}\times \KK\cong \{T_{11}T_{12}\ldots T_{1n_1}=T_2^{l_2}+1\}\times \KK.
$$
Note that by~\cite{G}, see also \cite{E-G-S}, all non-rigid trinomial hypersurfaces of type I have the form 
$$
X=\{T_{11}T_{12}^{l_{12}}\ldots T_{1n_1}^{l_{1n_1}}=T_2^{l_2}+1\}. 
$$
We can assume that $n_1\geq 2$, otherwise $X\cong \KK^{n_1+n_2-1}$.
So to prove that a non-rigid trinomial hypersurface of type I is non-cancellative, one should prove that among varieties 
$$
\{T_{11}T_{12}^{l_{12}}\ldots T_{1n_1}^{l_{1n_1}}=T_2^{l_2}+1\}. 
$$
with fixed $n_1$ and $n_2$ there are two non-isomorphic. We can do it in assumption that all $l_{2j}\geq 2$. In~\cite{G} it is proved that 
$$
\ML(\{T_{11}\ldots T_{1n_1}=T_{2}^{l_2}+1\})=\KK,
$$
and if $l_{2j}, k_i\geq 2$ for all $i,j$, then 
$$
\ML(\{T_{11}T_{21}^{k_2}\ldots T_{1n_1}^{k_{n_1}}=T_{2}^{l_2}+1\})=\KK[T_{21},\ldots, T_{1n_1}].
$$
This implies that these varieties are non-isomorphic. 
\end{ex}

It is a natural question whether trinomial hypersurfaces of the form 
$$
X=\{T_{11}T_{12}^{l_{12}}\ldots T_{1n_1}^{l_{1n_1}}+T_{21}T_{22}^{l_{22}}\ldots T_{2n_2}^{l_{2n_2}}+1=0\} 
$$
are cancellative? It is close to the question, when two hypersurfaces of this form are isomorphic. This is an open question. A particular result was obtained in \cite{DF}. Let us denote $X_{pq}=\{xy^p+zw^q=1\}$. It is proved in \cite{DF} that if $p+q=p'+q'$, then $X_{pq}\cong X_{p'q'}$.

\begin{ex}\label{muldanvarneis}
     In Example~\ref{emdv} it it shown that cylinders over n-Danielewski varieties $\Dan(Z,K,f_1,\ldots, f_n)$ and $\Dan(Z,P,f_1,\ldots, f_n)$ are isomorphic, where  
     $$
\begin{cases}
    p_{ij}=k_{ij}, \text{ for } (i,j)\neq (n,1);\\
    p_{n1}=k_{n1}-1.
\end{cases}
$$
and $\mathfrak{J}$ is not proper. In~\cite{GS} isomorphism classes of Double Danielewski surfaces are described. In some very slight assumptions on coefficients $d$ and $e$ it is proved that 
$$
\Dan\left(\KK,\begin{pmatrix}d\\e\end{pmatrix}, f_1,f_2\right)\ncong \Dan\left(\KK,\begin{pmatrix}d\\e-1\end{pmatrix}, f_1,f_2\right).
$$
This proves that Double Danielewski surfaces are non-cancellative. In a forthcoming paper P. Evdokimova describes isomorphism classes of Multi Danielewski varieties. This result also states that in some very slight assumptions on coefficients of $K$ holds
$$
\Dan(Z,K,f_1,\ldots, f_n)\ncong\Dan(Z,P,f_1,\ldots, f_n).
$$
So, Multi Danielewski varieties are also non-cancellative.
\end{ex}

\section{Non-conjugate maximal tori}\label{secconj}

It is proved in~\cite{D} that the variety $X=\{xy^2=z^2-1\}\times \KK$ has infinitely many non-conjugate in $\Aut(X)$ two-dimensional tori. It is not difficult to prove that $X$ is not toric. Therefore, $\Aut(X)$ containes infinitely many non-conjugate maximal tori. One more known example of a variety with infinite number of non-conjugate maximal tori is $X=\{xy=zw-1\}\times \KK$; see~~\cite{D} and~\cite[Example~8.14]{KrZ}. We are going to prove these facts using technique of isolated irreducible semi-invariants, see Section~\ref{secsemiinv}, and to extend this statement to a wide class of varieties.

\begin{prop}\label{pncj}
    Suppose we have two varieties $X_1$ and $X_2$ with the action of tori $T_1$ on $X_1$ and $T_2$ on $X_2$.  Let $X_1\times \KK\cong X_2\times \KK=Y$. We can consider tori $\widehat{T_i}\cong T_i\times \KK^\times$, $i=1,2$ acting on $X_i\times \KK$ by $(t_i,t)\cdot (x_i,z)=(t_ix_i, tz)$, where $t_i\in T_i$, $t\in \KK^\times$, $x_i\in X_i$, $z\in \KK$.  Assume that $\dim T_1=\dim T_2=n$ and subsets $W(X_1,T_1)$ and $W(X_2,T_2)$ are not equivalent. Then $\widehat{T_1}$ and $\widehat{T_2}$ are not conjugate in $\Aut(Y)$.
\end{prop}
\begin{proof}
    Let us consider a system of weights $S\subseteq M=\mathfrak{X}(T)\cong\mathbb{Z}^n$. We say that a weight $\omega\in S$ is  {\it independent} if $\omega\notin \mathrm{span}(S\setminus \omega)$. Then each finite system of weights has the form $S=S_0\cup \{\omega_1,\ldots, \omega_d\}$, where $\omega_j$ are independent weights of $S$, and the system $S_0$ does not contain any independent weights. For each nonzero weight $\omega$ we can define a positive integer $k(\omega)$ such that $\omega=k(\omega)\gamma$, where $\gamma$ is a primitive vector of $M$.  It is easy to see that two systems $S$ and $S'$ are equivalent if and only if system $S_0$ is equivalent to $S'_0$ and there is a bijection $\varphi\colon \{\omega_1,\ldots, \omega_d\}\rightarrow \{\omega'_1,\ldots, \omega'_{d'}\}$ such that $k(\omega_j)=k(\varphi(\omega_j))$.  

    We have $W(Y,\widehat{T_i})=(W(X_i,T_i),0)\cup (0,\ldots,0,1)\subseteq \mathbb{Z}^{n+1}$. That is, system $W(Y,\widehat{T_i})$ can be obtained from $W(X_i,T_i)$ by adding one primitive independent weight. Since subsets $W(X_1,T_1)$ and $W(X_2,T_2)$ are not equivalent, subsets $W(Y,\widehat{T}_1)$ and $W(Y,\widehat{T}_2)$ are not equivalent. By Lemma~\ref{nesnct}, tori $\widehat{T_1}$ and $\widehat{T_2}$ are not conjugate in $\Aut(Y)$.
\end{proof}

\begin{ex}
    Danielewski surface $Z_n={xy^n=z^2-1}$ admits a natural $T_n\cong\KK^\times$ action given by $t\cdot x=t^n x, t\cdot y=t^{-1}y, t\cdot z=z$. Then $W(Z_n, T_n)=\{-1,n\}\subseteq \mathbb{Z}$. Since  all the cylinders $Z_n\times \KK$ are isomorphic and all the subsets $\{-1,n\}\subseteq \mathbb{Z}$ are not equivalent, by Proposition~\ref{pncj} all tori $\widehat{T_i}$ for $i\in \mathbb{Z}_{>0}$ are not conjugate in $\Aut(Z_i\times\KK)$.
\end{ex}

Using Corollary~\ref{nc0} we can prove the following generalization of this example.
\begin{theor}\label{teoremanonconj}
    Suppose $W$ is an irreducible affine variety. Consider 
$$f=\alpha z^k+1\in\KK[W\times \KK],$$ where $\alpha\in\KK[W]$ and $z$ is the coordinate on $\KK$. 
Assume there exists a torus $T\subseteq \Aut(W\times \KK)$ such that $f$ is a $T$-invariant. Assume also that $T$ can not be included in a bigger torus $\widehat{T}\subseteq \Aut(W\times \KK)$ such that $f$ is $\widehat{T}$-semi-invariant with a nonzero weight. 
Then 
    $$
    X=\Susp(W\times \KK,f,1,k_1,\ldots, k_m)\times \KK, \text{ where }k_j\in\mathbb{Z}_{>0},
    $$
    admits infinitely many non-conjugate maximal tori.
\end{theor}

\begin{proof}
    By definition,  
    $$
    \KK[X]=\KK[W\times \KK][x,y_1,\ldots, y_m,u]/(xy_1^{k_1}\ldots y_m^{k_m}-f)
    $$ 
    admits the following action of $(m+1)$-dimensional torus $S$:
    $$
    (t_1,\ldots, t_m, t_{m+1})\cdot y_i=t_i y_i,\qquad (t_1,\ldots, t_m, t_{m+1})\cdot x=t_1^{-k_1}\ldots t_m^{-k_m}x,
    $$
    $$
    (t_1,\ldots, t_m, t_{m+1})\cdot u=t_{m+1}u,
    $$
    $$ 
    (t_1,\ldots, t_m, t_{m+1})\cdot h=h, \text{ for all } h\in \KK[W\times \KK].
    $$
    We obtain $S\times T\subseteq \Aut(X)$.

    Let us prove that $u$, $x$ and $y_i$ for all $i$ are isolated irreducible $S$-semi-invariants. Indeed, let $\alpha_i(a_1,\ldots, a_{m+1})=a_i$. Then $\alpha_i$ is a $y_i$-separating function for $1\leq i\leq d$, $\alpha_{m+1}$ is a $u$-separating function, and $-\alpha_1$ is an $x$-separating function. Let $\mathbb{T}$ be a maximal torus containing $S\times T$. Then $x, y_1,\ldots, y_m,u$ are isolated irreducible $\mathbb{T}$-semi-invariants. Therefore, $f$ is a $\mathbb{T}$-semi-invariant. 

    Since $\KK[X]^S=\KK[W\times\KK]$, the subalgebra $\KK[W\times\KK]$ is $\mathbb{T}$-invariant. Therefore, $\mathbb{T}$-action induces an action of a torus $\widehat{T}\supseteq T$ on $W\times\KK$. By conditions of the theorem, $f$ is $\widehat{T}$-invariant, and hence, $f$ is $\mathbb{T}$-invariant. So, we have 
    $$
    \omega(x)+k_1\omega(y_1)+\ldots+k_m\omega(y_m)=0,
    $$
    where $\omega(h)$ denotes the $\mathbb{T}$-weight of a semi-invariant function $h$. Up to proportionality, this is the unique equation in $\omega(x), \omega(y_1),\ldots, \omega(y_m)$. Therefore, for various tuples $k_1,\ldots, k_m$ the sets $\{\omega(x), \omega(y_1),\ldots, \omega(y_m)\}\subseteq M=\mathfrak{X}(\mathbb{T})$ are not equivalent.

Assume that the number of conjugation classes of maximal tori on $X$ is finite. For each conjugation class, let us consider all sets consisting of weights of $m+1$ isolated irreducible semi-invariants. These sets are defined up to application an element of $\mathrm{GL}_n(\mathbb{Z})$, but the number of non-equivalent such sets is also finite. 
    
   By Corollary~\ref{nc0}, $X$ does not depend on $k_1,\ldots, k_m$. Therefore,  sets 
   $$\{\omega(x), \omega(y_1), \ldots, \omega(y_m)\}\subseteq M$$ 
   for various tuples $k_1,\ldots, k_m$ give infinite number of non-equivalent sets of weights of $m+1$ isolated irreducible semi-invariants. So, we obtain a contradiction. Thus, the number of conjugation classes of maximal tori on $X$ is infinite. 
\end{proof}
\begin{cor}\label{ssylkan}
    Let $Z$ be the trinomial hypersurface of Type {\rm I} given in $\KK^{n_1+n_2}$, $n_1\geq 2$, by the equation
$$
T_{11}T_{12}^{l_{12}}\ldots T_{1n_1}^{l_{1n_1}}=T_2^{l_2}+1.
$$
We assume that if $n_2=1$, then $l_{21}\geq 2$.
Denote $X=Z\times\KK$. Then there exists infinitely many non-conjugate maximal tori of dimension $n_1+n_2-1=\dim X-1$ in $\Aut(X)$.
\end{cor}
\begin{proof}
    We have $Z=\Susp(\KK^{n_2},T_2^{l_2}+1,1,l_{21},\ldots, l_{1n_1})$. Let $W=\KK^{n_2-1}$ with coordinates $T_{22},\ldots, T_{2n_2}$. So, $W\times \KK=\KK^{n_2}$ with coordinates $T_{21},T_{22},\ldots, T_{2n_2}$. Then $f=T_2^{l_2}+1$ has the form $\alpha T_{21}^{l_{21}}+1$, where $\alpha=T_{22}^{l_{22}}\ldots T_{2n_2}^{l_{2n_2}}\in\KK[W]$, see also Example~\ref{e2}.
    Let us consider the action of $n_2$-dimensional torus on $W\times \KK$ by 
    $$(t_1,\ldots, t_{n_2})\cdot (T_{21},\ldots, T_{2n_2})=(t_1T_{21},\ldots, t_{n_2}T_{2n_2}).$$ 
    Denote by $T$ the neutral component of the stabilizer of $f$. Now let us consider two cases.

    {\it Case 1. $n_2=1$.} We have $T=\{1\}$ and $f=T_{21}^{l_{21}}+1$. Since $l_{21}\geq 2$, there is no nontrivial algebraic action of a one-dimensional torus on $\KK[T_{21}]$ such that $f$ is a semi-invariant.

    {\it Case 2. $n_2\geq 2$.} The elements $T_{21},\ldots, T_{2n_2}$ are isolated  irreducible $T$-semi-invariants. Therefore, if $T$ is included in $\widehat{T}$, then each $T_{2j}$ is $\widehat{T}$-semi-invariant. Therefore, $f=T_2^{l_2}+1$ is a sum of two $\widehat{T}$-semi-invariants. If $f$ is semi-invariant, then $\widehat{T}$-weights of $T_2^{l_2}$ and $1$ are coincide. That is $f$ is $\widehat{T}$-invariant.  
    
    In both cases we have proved that $f$ is not a $\widehat{T}$-semi-invariant of non-zero weight.
    Therefore, by Theorem~\ref{teoremanonconj} there are infinitely many non-conjugate tori in $\Aut(X)$. Since by construction each such torus contains the $(n-1)$-dimensional torus $S\times T$, and they are not toric, all these tori are $(n-1)$-dimensional. 
\end{proof}

Note that in the proof of Theorem~\ref{teoremanonconj} we used the condition $f=\alpha z^k+1$ only when we applied Corollary~\ref{nc0}. So, if we use~Corollary~\ref{nc1}, we obtain the same statement in the case of
Danielewski varieties. 

\begin{cor}
    Let $Z$ be a Danielewski variety, i.e. 
    $Z=\Susp(\KK, P(z), 1, k_1,\ldots, k_m)$, where $\deg P\geq 2$ and $P$ has distinct roots, see Example~\ref{e4}. Let $X=Z\times\KK$. Then there exist infinitely many non-conjugate $m$-dimensional tori in $\Aut(X)$. 
\end{cor}
\begin{proof}
   The only thing we need to check is that the trivial torus $T=\{1\}$ cannot be included in a bigger torus $\widehat{T}\subseteq\Aut(\KK)$ such that $f=P(z)$ is $\widehat{T}$-semi-invariant with a nonzero weight. This is true since $\deg P\geq 2$ and $P(z)$ has distinct roots. 
\end{proof}


\section{Maximal tori of different dimensions in automorphism group}
\label{secvar}

In this section we give some examples of irreducible affine algebraic varieties that have maximal tori of different dimensions in automorphism group. The technique is the same as in the previous section. All these varieties are cylinders over some counter-examples to the generalized Zariski cancellation problem. If $$X=Y\times \KK\cong Z\times \KK$$ and varieties $Y$ and $Z$ have   maximal tori of different dimensions, then we obtain two tori of different dimension in $\Aut(X)$. The following lemma states that these tori are maximal in $\Aut(X)$. 

\begin{lem}\label{lemum}
    Let $T\subseteq \Aut(Y)$ is a maximal torus. Then $T\times \KK^\times\subseteq \Aut(Y\times\KK)$ acting by 
    $$
    (t,s)\cdot (y,a)=(t\cdot y, sa), \qquad y\in Y, a\in \KK, t\in T, s\in\KK^\times
    $$
    is a maximal torus.
\end{lem}
\begin{proof}
    Suppose $T\times \KK^\times$ can be included in a bigger torus $\mathbb{T}$. Then $\mathbb{T}$ is contained in the centalizer of $ \KK^\times$. If $u$ is the coordinate on $\KK$, then $s\cdot u=su$, hence $\KK[Y\times\KK]^{\KK^\times}= \KK[Y]$ is $\mathbb{T}$-invariant. Therefore restriction of $\mathbb{T}$ to $\KK[Y]$ gives a homomorphism $\mathbb{T}\rightarrow \Aut(\KK[Y])$. The image is a torus $\widehat{T}$. We have $T\hookrightarrow \mathbb{T}\rightarrow \widehat{T}$. This composition is a injection from $T$ to $\widehat{T}$. Since $T$ is a maximal torus of $\Aut(Y)$, we have $\widehat{T}=T$. 

    Now, let $t\in \mathbb{T}$. Then $t\cdot u\in u\KK[Y]$. Since $u$ is an irreducible element, we have $t\cdot u=\lambda u$, where $\lambda\in \KK[Y]^\times$. Therefore, $\mathbb{T}|_{\KK[Y]}=T$ and $u$ is a $\mathbb{T}$ semi-invariant. Hence, $\mathbb{T}=T\times \KK^\times$.
\end{proof}

As examples of varieties $Y$ and $Z$ we use varieties from Lemma~\ref{stmodl}. One of the ways to prove that a torus is maximal is to use the following lemma.

\begin{lem}\label{ldd}
Suppose $\Lambda\cong\KK^\times$ is an algebraic subgroup in $\Aut(X)$. Let us consider the corresponding $\mathbb{Z}$-grading of $\KK[X]$. Suppose $S$ is a set of homogeneous generators of $\KK[X]$ such that for a prime number $p$ and $f\in S$ we have $p\nmid \deg f$ and $p$ divides degrees of all the other elements of $S$. Then $f$ is an irreducible element of $\KK[X]$ and for each torus $\mathbb{T}\supseteq \Lambda$ we have $t\cdot f=\lambda f$, $t\in \mathbb{T}$, $\lambda\in \KK[X]^\times$.     
\end{lem}
\begin{proof}
If $f=gh$, then $g$ and $h$ are homogeneous and the degree of one of them is not divisible by $p$. If $p\nmid \deg g$, then $f\mid g$. So, $f$ is irreducible.

    Since $\Lambda\subseteq \mathbb{T}$, action of $t\in \mathbb{T}$ preserves the degrees of elements. Therefore, $t\cdot f$ is a homogeneous element of degree $\deg f$. Since degrees of all elements of $S$ except~$f$ are divisible by $p$, we obtain $f\mid t\cdot f$. Since $f$ and $t\cdot f$ are irreducible, we have $t\cdot f=\lambda f$, $\lambda\in\KK[X]^\times$. 
\end{proof}

\begin{prop}
    Let us consider the following varieties:
$$
Z=\{xy^4+y^2z^3+z^3w^5+1=0\}, \qquad Y=\{xy^2+z^3w^5+1=0\}. 
$$
Then

(i) $X=Z\times\KK\cong Y\times \KK$;

(ii) There exist maximal tori of dimensions $2$ and $3$ in $\Aut(X)$.  
\end{prop}
\begin{proof}
(i) Let us consider $Z_0=\{xy^4+z^3w^5+1\}$. By Example~\ref{e2}, 
$$Z_0\times\KK\cong \{xy+z^3w^5+1=0\}\times\KK,$$
and the LND $\delta_0\colon (x,y,z,w)\mapsto (3z^2w^5,0,y^4,0)$ has a slice. 
Then by Lemma~\ref{stmodl} 
$Z\times \KK\cong Z'\times\KK$, where 
$$
Z'=\{xy^3+y^2z^3+z^3w^5+1=0\}.
$$
Analogically $Z'\times \KK\cong Z''\times \KK$, where
$$
Z''=\{xy^2+y^2z^3+z^3w^5+1=0\}.
$$
But 
$$
Z''=\{(x+z^3)y^2+z^3w^5+1=0\}\cong \{xy^2+z^3w^5+1=0\}=Y.
$$

(ii)
We have the following $\KK^\times$-action on $Z$:
$$
t\cdot (x,y,z,w)=(t^{-60}x,t^{15}y, t^{-10}z, t^{6}w).
$$
Let us prove that this one-dimensional torus is maximal in $\Aut(Z)$. Suppose $\KK^\times\subseteq \mathbb{T}$. By Lemma~\ref{ldd}, applying to $p=2, 3$ and $5$, we have $t\cdot w=\lambda w$, $t\cdot z=\mu z$, $t\cdot y=\nu y$, where $\lambda,\mu,\nu\in\KK[Z]^\times=\KK^\times$. Let $t\cdot x=f(x,y,z,w)$ for a polynomial $f$. Since this action should respects the relation, we have 
$$
f\nu^4 y^4+\nu^2\mu^3y^2z^3+\mu^3\lambda^5 z^3w^5+1
$$
is divisible by $g=xy^4+y^2z^3+z^3w^5+1$ in $\KK[x,y,z,w]$. Let us put 
$$f\nu^4 y^4+\nu^2\mu^3y^2z^3+\mu^3\lambda^5 z^3w^5+1=sg.$$ Then
\begin{equation}\label{ssf}
(f\nu^4-x) y^4+(\nu^2\mu^3-1)y^2z^3+(\mu^3\lambda^5 -1)z^3w^5=(s-1)g. 
\end{equation}
Substituting $y=0$ into (\ref{ssf}) we obtain $$(\mu^3\lambda^5 -1)z^3w^5=(s(x,0,z,w)-1)(z^3w^5+1).$$
This implies $\mu^3\lambda^5=1$. Thus, (\ref{ssf}) takes the form
\begin{equation*}
(f\nu^4-x) y^4+(\nu^2\mu^3-1)y^2z^3=(s-1)g.
\end{equation*}
Therefore, $y^2 \mid (s-1)$ and 
\begin{equation}\label{ssff}
(f\nu^4-x) y^2+(\nu^2\mu^3-1)z^3=\frac{(s-1)}{y^2}g=lg.
\end{equation}
Substituting $y=0$ into (\ref{ssff}) we obtain
$$
(\nu^2\mu^3-1)z^3=l(x,0,z,w)(z^3w^5+1). 
$$
So, we conclude $\nu^2\mu^3=1$. We obtain 
\begin{equation}\label{ssfff}
(f\nu^4-x) y^2=lg.
\end{equation}
Therefore, $g\mid (f\nu^4-x)$. That is $f\nu^4=x$ in $\KK[Z]$. 

We have shown that $x,y,z$ and $w$ are $\mathbb{T}$-semi-invariants with weights $\nu^{-4}, \nu, \mu$ and $\lambda$ such that $\mu^3\lambda^5=1$ and $\nu^2\mu^3=1$. These conditions imply that there exists such $t\in \KK^\times$ that $\lambda=t^6$, $\mu=t^{-10}$, $\nu=t^{15}$, $\nu^{-4}=t^{-60}$. Thus, $\mathbb{T}$ is a maximal one-dimensional torus in $\Aut(Z)$. 
By Lemma~\ref{lemum} there exists a maximal two-dimensional torus in $\Aut(X)$. 

The variety $Y$ admits the two-dimensional maximal torus in $\Aut(Y)$ acting by 
$$
(t_1,t_2)\cdot (x,y,z,w)=(t_1^2x, t_1^{-1}y, t_2^5z, t_2^{-3}w). 
$$
By Lemma~\ref{lemum} there exists a maximal three-dimensional torus in $\Aut(X)$. 
\end{proof}

\begin{theor}\label{raztor}
    Let $Z$ be a Danielewski variety of the form
    $$
    Z=\Susp(\KK,f,1,k_1,\ldots, k_m),
    $$
    where $k_i\in \mathbb{Z}_{>0}$, $\deg f\geq 2$, and $f$ has distinct roots. Denote $X=Z\times \KK$. Then for each $1\leq n\leq m$ there exists a maximal torus in $\Aut(X)$ of dimension $n$.
\end{theor}
\begin{proof}
    By Corollary~\ref{nc1}, $X\cong \Susp(\KK,f,1,\ldots, 1)\times \KK$. Consider monomials $$M_j=\prod\limits_{s=1}^j y_s\cdot \prod\limits_{s=j+1}^m y_s^{2} ,\text{ where }0\leq j\leq m-1.$$
    Denote by $f_i$ the polynomial
    $$f_i=xM_i+\sum\limits_{j=i+1}^{m} M_j-f(z)\in\KK[x,y_1,\ldots, y_m,z],\text{ where }0\leq i\leq m,$$ and let 
    $Y_i$ be the zero locus of $f_i$. We put $X_i=Y_i\times \KK$. Then 
    $$
    Y_{m}=\{xy_1\ldots  y_m=f(z)\}=\Susp(\KK,f,1,\ldots,1, 1).
    $$
    So, by Corollary~\ref{nc1} we have $X_{m}\cong X$. Each $Y_i$ is a modification of $$\Susp(\KK,f,1,\ldots,1,2,\ldots,2).$$ Therefore, by Lemma~\ref{stmodl} we have
    \begin{multline*}
    X_i=Y_i\times\KK=\left\{xM_i+\sum\limits_{j=i+1}^{m} M_j=f(z)\right\}\times \KK\cong\\
    \cong \left\{xM_{i+1}+\sum\limits_{j=i+1}^{m} M_j=f(z)\right\}\times \KK=\left\{(x+1)M_{i+1}+\sum\limits_{j=i+2}^{m} M_j=f(z)\right\}\times \KK\cong\\
    \left\{xM_{i+1}+\sum\limits_{j=i+2}^{m} M_j=f(z)\right\}\times \KK=X_{i+1}
    \end{multline*}
    Thus, all $X_i$ are isomorphic to $X$.          

    Now, let us prove that in $\Aut(Y_i)$ there exists a maximal torus of dimension~$i$. This means that in $\Aut(X)$ there exists a maximal torus of dimension~$i+1$.
    We can assume that $f(z)=z^d+c_{d-2}z^{d-2}+\ldots+c_0=z^{a}f(z^b)$, where $a\in \{0,1\}$ and $b$ is the maximum possible value. The variety $Y_i$ is a Danielewski variety in sence of~\cite{Ga}, i.e. variety of the form 
    $$
    Y=\{xy_1^{k_1}\ldots y_m^{k_m}=P(y_1,\ldots, y_m, z)\}, 
    $$
    where
    $$
    P(y_1,\ldots, y_m, z)=z^d+s_{d-1}(y_1,\ldots, y_m)z^{d-1}+\ldots+s_{0}(y_1,\ldots, y_m),\qquad d\geq 2.
    $$
    In \cite[Theorem~7.15]{Ga} automorphism groups of such varieties are described. Now we recall this description. Denote by $\delta$ the LND of $\KK[Y]$ given by $\delta(x)=\frac{\partial P}{\partial z}$, $\delta(z)=y_1^{k_1}\ldots y_m^{k_m}$, $\delta(y_j)=0$. Then $\mathbb{U}(\partial)=\{\exp (f\partial)\mid f\in \mathrm{Ker}\,\partial\}$ is a commutative subgroup of $\Aut(Y)$. Since $x=\frac{P(y_1,\ldots, y_m, z)}{y_1^{k_1}\ldots y_m^{k_m}}$ in $\KK[Y]$, we have an embedding $\phi$ of $\KK[Y]$ in the field $\KK(y_1,\ldots, y_m, z)$.   Consider the $\mathrm{S}_m$-action on $\KK(y_1,...,y_m,z)$ by permutations of coordinates $y_1,\ldots, y_m$. By $S(Y) \subseteq \mathrm{S}_m$ we denote the stabilizer of the monomial $y_1^{k_1}\ldots y_m^{k_m}$. That is $S(Y)$ permutes only variables with coinciding $k_i$. Consider the action of the semidirect product 
    $G= S(Y) \rightthreetimes (\KK^\times)^{m+1}$ on $\KK(y_1,...,y_m,z)$ given by
$$
(\sigma,t_1,\ldots,t_{m+1})\cdot (y_1,\ldots,y_m,z) = (t_1y_{\sigma(1)},...,t_my_{\sigma(m)},t_{m+1}z).
$$
Let $\mathbb{G}\subseteq G$ be the stabilizer
of $\phi(\KK[Y])$. Then the group $\mathbb{G}$ is called {\it the canonical group} of $Y$, and we have a $\mathbb{G}$-action on $Y$. By \cite[Theorem~7.15]{Ga}, $\Aut(Y)\cong\mathbb{G}\rightthreetimes \mathbb{U}(\delta)$. 
It is easy to see that $\mathbb{G}$ is a finite extension of the torus $\mathbb{G}^0$. Then $\mathbb{G}^0$ is a maximal torus in $\Aut(Y)$, moreover, each  maximal torus in $\Aut(Y)$ is conjugate to $\mathbb{G}^0$. 

It is easy to see that in case $Y=Y_i$ we have $\mathbb{G}=\mathrm{S}_{i}\times (\KK^\times )^{i}\times \mathcal{C}_b$, where $\mathcal{C}_b=(\{w\in\KK\mid w^b=1\},\cdot)$ is the cyclic group of order $b$. The group $\mathbb{G}$ acts on $Y_i$ by 
\begin{multline*}
(\sigma, t_1,\ldots, t_i, \varepsilon)\cdot (x,y_1,  \ldots, y_m,z)=\\
=(t_1^{-1}\ldots t_i^{-1}\varepsilon^{-a}x,t_1 y_{\sigma(1)},\ldots, t_iy_{\sigma(i)},\varepsilon^ay_{i+1} \ldots, y_m,\varepsilon z).
\end{multline*}
So, the dimension of the unique up to conjugation maximal torus in $\Aut(Y_i)$ is~$i$. Since $i\in \{0,1,\ldots, m\}$, we obtain maximal tori of dimensions $1,2,\ldots, m$ in $\Aut(X)$. 
\end{proof}

\section{Bhatwadekar’s technique for $m$-suspensions of special type}
\label{SB}

Let $D$ be an LND of a $\mathbb{K}$-domain $B=\KK[X]$. Denote by $A$ the kernel of $D$. If $D$ has a slice $s$, then $B=A[s]$ and each LND $\delta$ on $A$ can be naturally lifted to an LND of $B$.
We introduce a technique which allows to extend an LND $E\in \mathrm{LND}(A)$ satisfying some conditions to an LND $\widetilde{E}\in \mathrm{LND}(B)$. We call it Bhatwadekar’s technique.

Suppose for some non-zero $h\in \Ker E$ there exists $v\in B$ such that $D(v)=h$, i.e. $h\in \mathrm{pl}(D)$. Then $B_h=A_h[v]$. Hence, we can extend $E$ to an LND $\widehat{E}\in\LND(B_h)$ 
via $\widehat{E}(v)=0$. 
Since $B$ is finitely generated, we can take such a positive integer $N$ that $h^N\widehat{E}$ takes each generator of $B$ to $B$. 
Therefore, the restriction of $h^N\widehat{E}$ to $B$ gives an LND $\widetilde{E}\in\LND(B)$.

\begin{re}
     Geometrically, since $h\in \mathrm{pl}(D)$, the principle open subset $X_h=\mathrm{Spec}\,B_h$ is a cylinder $Y\times \KK$, where $Y=\mathrm{Spec}\, A_h$. The LND $E$ gives a structure of a cylinder on a principle open subset $Y_f$ for some $f$. So, the principle open subset $X_{h^kf}$ have a structure of $\KK^2$-cylinder. Here we take $k$ such that $h^kf$ is regular on $X$. In particular, it has a structure of $\KK$-cylinder extending the structure of $\KK$-cylinder on the base $Y_f$ given by $E$. So, we can extend $E$ to an LND on $X_{h^kf}$. This structure of a cylinder on a principle open subset gives an LND on $X$.
\end{re}

The advantage of this technique is that if $a\in A$ is not in the kernel of $E$, then it is not in the kernel of $\widetilde{E}$, so it is not included in $\ML(X)$. More precisely, we get the following statement.

\begin{prop}\label{mlprop}
    Let $D$ be an LND of a $\mathbb{K}$-domain $B=\KK[X]$ and $A=\Ker D$. Then $$\ML(X)\subseteq \bigcap_{E\in \LND(A),\  \mathrm{Ker} E\cap \mathrm{pl}(D)\neq \{0\}}\Ker E.$$
\end{prop}
\begin{proof}
    By definition, $\ML(X)\subseteq \Ker D\cap \Ker \widetilde{E}=\Ker E$. 
\end{proof}
\begin{ex}
    Note that the condition $\Ker E\cap \mathrm{pl}(D)\neq \{0\}$ in Proposition~\ref{mlprop} can not be omitted. Indeed, let $X=\{xy^2=z^2-1\}$. It is well-known that $\ML(X)=\KK[y]$. Consider $D\in\LND(X)$ given by $(x,y,z)\rightarrow (2z,0,y^2)$. Then $A=\Ker D=\KK[y]$. We put $E=\frac{\partial}{\partial y}\in\LND(A)$. Then $\Ker E=\KK$. So, we obtain $\ML(X)=\KK[y]\not\subset\KK=\Ker E$. This do not contradict to Proposition~\ref{mlprop} because $\mathrm{pl}(D)=(y^2)$ and $\Ker E\cap \mathrm{pl}(D)=\{0\}$.
\end{ex}

We would like to apply this idea to the particular case where $X$ is a cylinder over an $m$-suspension of some special form. Firstly, we state two corollaries of Proposition~\ref{2pp1} and Corollary~\ref{2pc1} which are similar to Corollaries~\ref{nt0} and~\ref{nc0}.
Let $V$ be an irreducible affine variety and
let $\xi$ be an LND of $\KK[V]$. Fix $f\in \KK[V]$ such that $\xi(f)\neq 0$. 
Denote 
$$
Z=\Susp(V,f,1,k_1,\ldots, k_m)=\{xy_1^{k_1}\ldots y_m^{k_m}=f\}\subseteq \mathbb{A}^{m+1}\times V,\qquad k_i\in\mathbb{Z}_{>0}.
$$
Consider $\delta=\overline{\xi}\in \mathrm{LND}(Z)$; see Lemma~\ref{lprl}. Put $x=y_1, y=y_2$, and $D=y\frac{\partial}{\partial u}+\delta\in \mathrm{LND}(Z\times \KK)$.

\begin{cor}\label{nt0n}
 Suppose that $pf+q\xi(f)=h$ for some $p$, $q$ in $\KK[V]$, and $h$ in $\Ker\xi$. Then $h\in \mathrm{pl}(D)$. 
\end{cor}
\begin{proof}
    Conditions $(*)$ are satisfied. We have $g=\xi(f)$, $\widehat{h_1}=f$. The condition $h\in \Ker\xi$ implies $h\in \Ker D$.   So, by Corollary~\ref{2pc1} we obtain $h\in \mathrm{pl}(D)$.
\end{proof}

\begin{cor}\label{l7}
    Let $r$ be a local slice of $\xi$. If we take $f=\alpha r^k+\beta$, where $\alpha,\beta\in\Ker\xi$, $\beta\neq 0$, then $h=\beta\delta(r)\in \mathrm{pl}(D)$. 
\end{cor}
\begin{proof}
    Take $p=\delta(r)$, $q=-\frac{1}{k}r$. Let's do a straightforward check. 
    $$
    pf+q\delta(f)=\delta(r)(\alpha r^k+\beta)- \frac{1}{k}r\alpha k r^{k-1}\delta(r)=\delta(r)\beta=h.
    $$
    The conditions of Corollary~\ref{nt0n} are satisfied.
\end{proof}

A straightforward consequence of the previous assertion is the following statement. 
\begin{cor}\label{nc0two}
Suppose $W$ is an irreducible affine variety. Consider 
$$f=\alpha z^n+\beta\in\KK[W\times \KK],$$ where $\alpha,\beta\in\KK[W]$, $\beta\neq 0$, and $z$ is the coordinate on $\KK$. 
Denote $$
    Z=\Susp(W\times \KK,f,1,k_1,\ldots, k_m)
    $$
    and consider $\xi=\frac{\partial}{\partial z}\in \LND(W\times \KK)$. Then for $D=\frac{\partial}{\partial u}+\overline{\xi}\in LND(W\times \KK)$ we have $\beta \in \mathrm{pl}(D)$.
\end{cor}

\begin{re}\label{replpl}
    Proposition~\ref{2pp1} implies that Corollaries~\ref{nt0n},~\ref{l7},~and~\ref{nc0two} remains true we replace the suspension $Z$ by its modification; see Definition~\ref{nd316}.
\end{re}

\begin{ex}\label{ee11}
Let $Z=\{xy^2=z^2+w^3\}=\Susp(\KK^2,z^2+w^3, 1,2)$. Here $z$ and $w$ are coordinate functions on $V=\KK^2=W\times \KK$, where $W\cong \KK$. We can consider $\xi=\frac{\partial}{\partial z}\in\LND(\KK[\KK^2])$.
If we put $D=x\frac{\partial}{\partial u}+\xi\in\LND(Z\times \KK)$ as above, then 
 $f=z^2+w^3$, i.e. $\alpha=1, \beta=w^3$. By Corollary~\ref{nc0two} we obtain  $w^3\in\mathrm{pl}(D)$.

By Example~\ref{prtrnkr} we have 
$$\Ker D= \KK[\widetilde{x},y,\widetilde{z},w]/(\widetilde{x}y-\widetilde{z}^2-w^3).$$
Consider $E\in\LND(\Ker D)$ given by 
$$
E(\widetilde{x})=0, \qquad E(y)=2\widetilde{z}, \qquad E(\widetilde{z})=\widetilde{x}, \qquad E(w)=0.
$$
Now we have 
$w^3\in \Ker E\cap\mathrm{pl}(D)$. So, by Proposition~\ref{mlprop} we obtain 
$$\ML(Z\times \KK)\subseteq \Ker E=\KK[\widetilde{x},w].$$

Considering $E'\colon (\widetilde{x},y,\widetilde{z},w)\rightarrow (2\widetilde{z},0,y,0)$, one can prove that 
$$\ML(Z\times \KK)\subseteq \Ker E'=\KK[y,w].$$
Therefore, $\ML(Z\times \KK)\subseteq\KK[w]$. 

Analogically one can show that $\ML(Z\times \KK)\subseteq\KK[z]$.
Therefore, 
$\ML(Z\times \KK)=\KK$.
\end{ex}

Of course, in the previous example we have used the fact that $Z$ is 2-suspension with $k_1=1,k_2=2$. However, in the next section we extend the technique to the case of $m$-suspensions with $k_1=1$.

Now we would like to apply the technique above to computing Makar-Limanov invariant of modifications of $m$-suspensions. The following example give the result of~\cite{D}. It was shown to authors (with another technique for describing the kernel of $D$) by N.~Gupta and P.~Ghosh. 

\begin{ex}\label{ee22}
    Let $Z$ be the Koras-Russell threefold $\{x+x^2y+z^2+w^3=0\}$. It is isomorphic to 
    $\{y+xy^2=z^2+w^3\}$. Let us proceed with the variety of the second form. 
    Following Example~\ref{KKRC} we say that $Z$ is a modification of the 2-suspension $xy^2=z^2+w^3$. It was shown in Example~\ref{KKRC} that for $\delta\colon(x,y,z,w)\rightarrow (2z,0,y^2,0)\in \LND(Z)$ and $D=y\frac{\partial}{\partial u}+\delta\in \LND(Z\times\KK)$ we obtain 
$$
\Ker D=\KK[\widetilde{x}, y,\widetilde{z}, w]/(\widetilde{x}y-\widetilde{z}^2-w^3).
$$
Remark~\ref{replpl} provides $w^3\in\mathrm{pl}(D)$. Therefore, by the same way as in the previous example we obtain $\ML(Z\times \KK)\subseteq\KK[\widetilde{x},w]\cap\KK[y,w]=\KK[w]$. Analogously, $\ML(Z\times \KK)\subseteq \KK[z]$. Therefore, $\ML(Z\times \KK)=\KK$.
\end{ex}

\section{Iterative application of generalized Bhatwadekar’s technique}\label{IAGBT}

Let $D$ be an LND of a domain $B$ and $A=\Ker D$. The Bhatwadekar’s technique described in Section~\ref{SB} allows to construct an LND on $B$ having an LND of $A$, satisfying some conditions. In this section we would like to introduce a similar construction, which we call generalized Bhatwadekar’s technique, but starting from LND $E$ of the algebra of polynomials $A[z]$ over $A$.

Let $v$ be a local slice of $D$. So, $D(v)=h\neq 0\in A$. Then $B_h=A_h[v]$, where $v$ is transcendental over $A_h$. Suppose $E$ is an LND of $A[z]$ such that $E(h)=0$. Then $E$ can be extended to an LND $\widehat{E}$ of $A_h[z]\cong B_h$. Therefore, $h^N\widehat{E}$ takes $B$ to $B$. The restriction $\widetilde{E}$ of $h^N\widehat{E}$ to $B$ is an LND of $B$. Moreover, $\widetilde{E}|_{A}=h^NE|_{A}$.

\begin{re}\label{rekerco}
Taking into account $A[z]\subseteq A_h[z]\cong B_h\subseteq \mathrm{Quot}(B)$ we can assume that $A[z]$ is a subring of $\mathrm{Quot}(B)$. (This embedding is not unique since the isomorphism $A_h[z]\cong B_h$ is not unique. To obtain an explicit embedding one should fix $v\in B$ such that $E(v)=h$.) Then $E$ can be extended to a derivation of $\mathrm{Quot}(B)$. The kernel of this extension coincides with the kernel of the extension of $\widetilde{E}$ to a derivation of $\mathrm{Quot}(B)$.
\end{re}

\begin{re}
If we consider $E=\frac{\partial}{\partial z}$, then $E|_A=0$ and $\widetilde{E}$ is equivalent to $D$. This do not give us a new LND. So, to obtain nontrivial LND one should use an LND $E$ which is not zero on $A$. 
\end{re}

\begin{re}
     Geometrically, since $h\in \mathrm{pl}(D)$, the principle open subset $X_h=\mathrm{Spec}\,B_h$ is a cylinder $Y\times \KK$, where $Y=\mathrm{Spec}\, A_h$. The LND $E$ gives another structure of a cylinder on $Y\times \KK$. This structure of a cylinder on a principle open subset gives an LND on $X$.
\end{re}

\begin{re}
The Bhatwadekar’s technique described in Section~\ref{SB} is a particular case of the generalized Bhatwadekar’s technique when $E(z)=0$ and hence, $E$ is obtained from an LND of $A$. Geometrically, this means that the new structure of a cylinder on $Y\times \KK$ is given by a structure of a cylinder on the base $Y$. 
\end{re}

The generalized Bhatwadekar’s technique can be applied iteratively. Let us describe this construction. We start with an LND $D_1$ of $B$, with the kernel $A_1$. And we have $h_1=D_1(v_1)\in A_1$. Then we consider an LND $D_2$ of $A_1[z_1]$ with the kernel $A_2$ and $h_2=D_2(v_2)\in A_2$ and so on. Each $D_i$ for $i\geq 2$ is an LND on $A_{i-1}[z_{i-1}]$ with $h_i=D_i(v_i)\in A_i=\Ker D_i$. Suppose we have an LND $E$ of $A_{n}[z_n]$ with $E(h_n)=0$. Then we can obtain by the generalized Bhatwadekar’s technique an LND $E_1=\widetilde{E}$ of $A_{n-1}[z_{n-1}]$. If $E_1(h_{n-1})=0$, by the generalized Bhatwadekar’s technique  we obtain 
$E_2=\widetilde{E_1}$ of $A_{n-2}[z_{n-2}]$ and so on. Finaly, if for each $i$ we have $E_{i}(h_{n-i})=0$, we obtain an LND $E_{n}$ of $B$. 

The simplest case of application of this technique is the case when $h_1=\ldots=h_n=h$. Then $B_h=(A_1)_h[v_1]=\ldots=(A_n)_h[v_n]$. In this case $E_n=h^N\widehat{E}$ for some positive integer~$N$, where $\widehat{E}$ is the extension of $E\in\LND(A_n[v_n])$ to $(A_n)_h[v_n]=B_h$. 
If we have fixed $D_1,\ldots, D_n$ as above such that $h=D_i(v_i)\in A_i$, we say that we are in conditions ($\#$).

\begin{prop}\label{fmlcap}
In conditions ($\#$) we have 
$$\FML(X)\subseteq \bigcap_{E\in \LND(A_n[v_n]),\  E(h)=0} \Quot(\Ker E).$$
\end{prop}
\begin{proof}
    By Remark~\ref{rekerco} we have $\Quot(\Ker E)=\Quot(\Ker E_1)=\ldots=\Quot(\Ker E_n)$. Since $E_n$ is an LND on $B$, we have $\FML(X)\subseteq \Quot(\Ker E_n)$.
\end{proof}

\begin{prop}\label{tpp}
 Let $\xi$ be a nonzero LND on a variety $V$. Suppose $r$ is a local slice of $\xi$. Let $f=\alpha r^k+\beta$, where $\alpha,\beta\in \Ker\xi$, $\beta\neq 0$, $k\in \mathbb{Z}_{>0}$. Suppose that there exists a factorial $G$-grading on $\KK[V]$ for some abelian group $G$, such that $f\in\KK[V]$ is homogeneously irreducible, $\xi$ is homogeneous, and $\gcd(f,\xi(f))=1$. 
 Denote $$X=\Susp(V,f,1,k_1,\ldots,k_m)\times \KK.$$ Then 
 $\FML(X)\subseteq\Quot(\Ker\xi)\subseteq \KK(V).$
\end{prop}
\begin{proof}
We have
 $$
    \KK[X]=\KK[V][y_1,\ldots, y_m, x, u]/(xy_1^{k_1}\ldots y_m^{k_m}-f).
    $$
    Since $k>0$, the function $f$ is not invertible in $\KK[V]$. Therefore, $y_1$ is not invertible in $\KK[X]$. 
    By Lemma~\ref{facgrlem}, $\KK[X]$ is factorially $(\mathbb{Z}^{m}\times G)/\langle-k_1,\ldots,-k_m,w\rangle$-graded, where $w$ is the weight of $f$. 
   Consider the LND $D_1=y_1\frac{\partial}{\partial u}+\overline{\xi}$ of $\KK[X]$. Then $g=D_1(x)=\xi(f)$. Since $\gcd(f,\xi(f))=1$, we have $\gcd(y_1,\xi(f))=1$. By Lemma~\ref{pl3} and Remark~\ref{pr2}, we have
    $$
    A_1=\Ker D_1\cong
    \KK[V][x,y_1,\ldots, y_m]/(xy_1^{k_1-1}\ldots y_m^{k_m}-f).
    $$
     By Corollary~\ref{l7}, there exists $v_1\in\KK[X]$ such that  $D_1(v_1)=h=\beta\xi(r)\neq 0$.

    
       Similarly, we have an LND $D_2=y_1\frac{\partial}{\partial z_1}+\overline{\xi}$ of $A_1[z_1]$ with $D_2(v_2)=h=\beta\xi(r)$ and 
    $$
    \Ker D_2\cong\KK[V][y_1,\ldots, y_m, x]/(xy_1^{k_1-2}\ldots y_m^{k_m}-f).
    $$
    We proceed by the same way and obtain $D_1,\ldots, D_n$, where $n=k_1+\ldots+k_m-m$. Here $D_i$ is an LND of $A_{i-1}[z_{i-1}]$, $D_i(v_i)=h$ and 
    $$
    A_{n}=\Ker D_{n}\cong \KK[V][y_1,\ldots, y_m, x]/(xy_1\ldots y_m-f). 
    $$
    So, $B_h=A_1[v_1]_h=\ldots=A_{n}[v_{n}]_h$ and we are in conditions ($\#$). 

    We see that $A_n=\Susp(V,f,1\ldots,1)$. Let us consider the LNDs $\overline{\xi_i}$ of $A_n$; see Section~\ref{secsusp}. We can extend them to LNDs $E^{(i)}$ of $A_n[z_n]$ by $E^{(i)}(z_n)=0$, $E^{(i)}|_{A_n}=\overline{\xi_i}$. Then $E^{(i)}(h)=0$. By  Proposition~\ref{fmlcap} we have
    $$
    \FML(X)\subseteq \bigcap \Quot(\Ker E^{(i)})\cap\Quot\left(\Ker \frac{\partial}{\partial z_n}\right)=\Quot(\Ker\xi);
    $$
    see Lemma~\ref{intker}.
\end{proof}

\begin{ex}
 Let $Z$ be a Danielewski surface $x^ny=z^k+1$. Then $$Z=\Susp(\KK,z^k+1, 1, n).$$ Denote $X=Z\times\KK$.
 Consider $G=\mathbb{Z}/2\mathbb{Z}$-grading on $\KK[z]$, where $\deg z=1\in\mathbb{Z}/2\mathbb{Z}$.
 The variety $\KK$ is $G$-factorial, $z^2-1=0$ ia a homogeneously irreducible element of degree zero, $\deg\frac{\partial}{\partial z}=1$, $\gcd(z^2-1,z)=1$. Since $z$ is a local slice for $\xi=\frac{\partial}{\partial z}$, we have $\FML(X)\subseteq \Quot(\Ker\xi)=\KK$. So, $\FML(X)=\KK$.

 Of cause, this fact easily follows from 
 $$
 \{x^ny=f(z)\}\times\KK\cong 
 \{xy=f(z)\}\times\KK.
 $$
\end{ex}

\begin{re}
    Using Corollary~\ref{nt0n} instead Corollary~\ref{l7}, we can replace the condition $f=\alpha r^k+b$ in Proposition~\ref{tpp} by $h=af+b\delta(f)r\neq 0\in\Ker\delta $. 
\end{re}

\begin{cor}\label{jal}
    Let $Z$ be a trinomial hypersurface of the form
    $$
    Z=\{xy_1^{k_1}\ldots y_m^{k_m}=z_1^{b_1}\ldots z_p^{b_p}+w_1^{c_1}\ldots w_q^{c_q}\},
    $$
    where $p\geq 1$, $q\geq 0$. Then the field Makar-Limanov invariant of the cylinder $X=Z\times \KK$ is trivial, i.e. $\FML(X)=\KK$.
\end{cor}
\begin{proof}
    Let us put $V=\KK^{p+q}$. Then $$X=\Susp(V,z_1^{b_1}\ldots z_p^{b_p}+w_1^{c_1}\ldots w_q^{c_q},1,k_1,\ldots,k_m)\times \KK.$$ 
    We consider the finest grading on $\KK[V]$ such that all the generators and $f$ are homogeneous. That is $\mathbb{Z}^{p+q}/\langle (b_1,\ldots,b_p,-c_1,\ldots, -c_q)\rangle$-grading. 
    It is easy to check that for each LND $\xi=\frac{\partial}{\partial z_i}$ or $\xi=\frac{\partial}{\partial w_j}$ all conditions of Proposition~\ref{tpp} are satisfied. (See Lemma~\ref{sfar} for proving that $f$ is homogeneously irreducible.) Therefore, 
    $$
    \FML(X)\subseteq \bigcap_{i=1}^p\Quot\left(\Ker \frac{\partial}{\partial z_i}\right)\cap  \bigcap_{j=1}^q\Quot\left(\Ker \frac{\partial}{\partial w_j}\right)=\KK.
    $$
\end{proof}

Now let us consider cylinders over a chain of $m$-suspensions.  

\begin{theor}\label{mtt}
Suppose $V$ is an irreducible affine algebraic variety and $\xi$ is an LND on it. Let $r_1,\ldots, r_n$ be local slices for $\xi$ and $f_i=\alpha_ir_i^{p_i}+\beta_i$, where $\alpha_i,\beta_i\in\Ker\xi$, $\beta_i\neq 0$. Suppose that there exists a factorial $G$-grading on $\KK[V]$ for some abelian group $G$, such that all $f_i\in\KK[V]$ are homogeneously irreducible and $\gcd(f_i,\xi(f_i))=1$.
 Put $Z_0=V$,  $Z_i=\Susp(Z_{i-1},f_i,1, k_{i1},\ldots,k_{im_i})$ for $i\geq 1$. 
    Denote $X=Z_n\times \KK$. Then $\FML(X)\subseteq \Quot(\Ker \xi)\subseteq \KK(V)$.
\end{theor}
\begin{proof}
    Let us denote $X_i=Z_i\times \KK$. Then $X_{i}=\Susp(X_{i-1}, f_i, k_{i1},\ldots, k_{im_i})$.

We have 
$$
Z_i=\{x_iy_{i1}^{k_{i1}}\ldots y_{im_i}^{k_{im_i}}-f_i\}\subseteq Z_{i-1}\times \KK^{m_i+1}.
$$
Let us consider the chain of LNDs 
$\delta_0=\xi$, $\delta_i=\overline{\delta_{i-1}}$ is an LND on $Z_i$; see Lemma~\ref{lprl}.
Then 
$$\Quot(\Ker\delta_n)=\Quot(\Ker\xi)(y_{ij}\mid 1\leq i\leq n, 1\leq j\leq m_i).$$

Now let us apply Proposition~\ref{tpp} for $Z_{p}=\Susp(Z_{p-1},f_p,1,k_{p1},\ldots, k_{pm_p})$ and the LND $\delta_{p-1}$ on $Z_{p-1}$. It is easy to see that since $r_p$ is a local slice for $\xi$, then it is a local slice for $\delta_{p-1}$. Also $\alpha_{p}$ and $\beta_p$ are in $\Ker\delta_{p-1}$.  We have 
$$
\FML(X_p)\subseteq \Quot(\Ker\delta_{p-1})=\Quot(\Ker\xi) (y_{ij}\mid i<p).
$$
But $X_i=\Susp(X_{i-1},f_i,1,k_{i1},\ldots, k_{im_i})$. By Lemma~\ref{lmlsunov} $$\FML(X_{i})\subseteq \FML(X_{i-1})(y_{i1},\ldots, y_{im_i}).$$ This implies 
$$
\FML(X)=\FML(X_n)\subseteq \FML(X_p)(y_{ij}\mid i>j)\subseteq\Quot(\Ker\xi)(y_{ij}\mid i\neq p).
$$
Since this is true for all $p$, we have $\FML(X)\subseteq\Quot(\Ker\xi)$.

\end{proof}

\begin{re}
One can take as $r_i$ a local slice for $\delta_i$ but not for $\xi$.
\end{re}

\begin{re}
One can take $f_i$ such that $h=af+b\delta(f)r\neq 0\in\Ker\delta$ for some $a,b\in \Ker\xi$ and use Corollary~\ref{nt0n}.
\end{re}

\begin{cor}\label{cc1}
    If in conditions of Theorem~\ref{mtt} suppose we have LNDs $\xi_1,\ldots, \xi_l$ on $V$. Assume that each $f_i$ can be represent as $f_i=\alpha_{ki}r_{ki}^{p_{ki}}+\beta_{ki}$, where $r_{ki}$ is a local slice of $\xi_k$ and $\alpha_{ki},\beta_{ki}\in\Ker\xi_{k}$, $\beta_{ki}\neq 0$. Then $\FML(X)\subseteq\bigcap_{i=1}^l\Quot(\Ker\xi_i)$.
\end{cor}

\begin{ex}
    Consider $Z\subset\KK^{12}$ given by the system
    $$
    \begin{cases}
        xy_1^2y_2^3=(z_1^5+z_2^3z_3)(z_4^3+5z_5^7)(z_6^2+7z_7^7)+1;\\
        vw^5=(z_1^3-2z_3^5)(z_2^3+z_5^4+z_4^2z_7^2)(1-z_6^7)+1.
    \end{cases}
    $$
    Then $\mathrm{FML}(Z\times \KK)=\KK$. Indeed, we can consider $V=\KK^7$ is the affine space with coordinates $z_1,\ldots, z_7$. There are LNDs $\xi_1, \ldots, \xi_7$ on $V$, where $\xi_j=\frac{\partial}{\partial z_j}$. We have 
    $$Z_1=\Susp(V,f_1, 1,2,3), \qquad Z=Z_2=\Susp(Z_1,f_2, 1,5),\text{ where}$$
    $$
    f_1=(z_1^5+z_2^3z_3)(z_4^3+5z_5^7)(z_6^2+7z_7^7)+1,\qquad f_2=(z_1^3-2z_3^5)(z_2^3+z_5^4+z_4^2z_7^2)(1-z_6^7)+1.
    $$
    For each $i,j$ we have 
    $$
    f_i=g(z_1,\ldots, z_{j-1}, z_{j+1},\ldots, z_7)z_j^{p_{ij}}+h(z_1,\ldots, z_{j-1}, z_{j+1},\ldots, z_7)
    $$
    is an irreducible polynomial. Therefore, 
    all conditions of Corollary~\ref{cc1} are satisfied with respect to the trivial grading on $\KK^7$. So, 
    $$\mathrm{FML}(Z\times\KK)\subseteq\bigcap_{j=1}^7\Quot(\Ker\xi_j)=\KK.$$
\end{ex}

\section{Bisuspensions}\label{secbis}

The construction of $m$-suspension being given by a variety $X$ and a regular function $f$ on it builds a new variety, which is subvariety of $X\times \mathbb{K}^m$ given by $y_1^{k_1}\ldots y_m^{k_m}=f$. In this section, we introduce a similar construction which we call {\it bisuspension}.  For a bisuspension we replace the monomial $y_1^{k_1}\ldots y_m^{k_m}$ by a sum of two monomials. 

\begin{de}
    Let $f$ be a nonzero regular function on an affine variety $X$. We fix positive integers $m, n$ and $k_1,\ldots, k_m$, $s_1,\dots, s_n$. By {\it $m,n$-bisuspension} we mean a subvariety in $\KK^{m+n}\times X$ cut by 
    $$
    y_1^{k_1}\ldots y_m^{k_m}-z_1^{s_1}\ldots z_n^{s_n}=f,
    $$
    where $y_1,\ldots, y_m, z_1,\ldots, z_n$ are coordinate functions on $\KK^{m+n}$.
    We denote this variety by $\Bisusp_{m,n}(X,f, k_1,\ldots, k_m, s_1,\dots, s_n)$.
\end{de}
  \begin{re}\label{bissu}
        Of couse, 
        \begin{multline*}
        \Bisusp_{m,n}(X,f, k_1,\ldots, k_m, s_1,\dots, s_n)\cong\\
        \cong
        \Susp(X\times \KK^n,f+z_1^{s_1}\ldots z_n^{s_n},k_1,\ldots, k_m).
        \end{multline*}
        I.e. a bisuspension is just a particular case of an $m$-suspension. But since the variety $X\times \KK^n$ and the function $f+z_1^{s_1}\ldots z_n^{s_n}$ have special forms, a bisuspension has some properties  that an arbitrary $m$-suspension does not have. 
    \end{re}

In this section we are interested on bisuspensions of the form 
    $$Y=\Bisusp_{m,n}(X,f,2, 2p_2,\ldots, 2p_m, 2, 2q_2,\ldots, 2q_n).$$
    Having an LND $\delta$ on $X$ we can extend it to an LND on $Y$ by two different ways. We obtain LNDs $\delta_{+}$ and $\delta_{-}$ respectively. To define these two LNDs we need the following notations 
    $\gamma=y_1y_2^{p_2}\ldots y_m^{p_m}+z_1z_2^{q_2}\ldots z_n^{q_n}$, 
    $\rho=y_1y_2^{p_2}\ldots y_m^{p_m}-z_1z_2^{q_2}\ldots z_n^{q_n}$. 
    So, we have
    $$
    \gamma\rho=y_1^2y_2^{2p_2}\ldots y_m^{2p_m}-z_1^2z_2^{2q_2}\ldots z_n^{2q_n}.
    $$
    By definition
    $$
    \delta_+(y_i)=\delta_+(z_j)=0 \text{ for all } i,j\geq 2;
    $$
    $$
    \delta_+(y_1)=z_2^{q_2}\ldots z_n^{q_n}\delta(f); \qquad\delta_+(z_1)=y_2^{p_2}\ldots y_m^{p_m}\delta(f);
    $$
    $$
    \delta_+(h)=2\rho  y_2^{p_2}\ldots y_m^{p_m}z_2^{q_2}\ldots z_n^{q_n}\delta(h) \text{ for }h\in\KK[X].
    $$
    It is easy to see that $\delta_{+}(\rho)=0$. Using this one can check that $\delta_+$ is a well-defined locally nilpotent derivation on $\KK[Y]$.

Similarly 
 $$
    \delta_-(y_i)=\delta_-(z_j)=0 \text{ for all } i,j\geq 2;
    $$
    $$
    \delta_-(y_1)=z_2^{q_2}\ldots z_n^{q_n}\delta(f); \qquad\delta_-(z_1)=-y_2^{p_2}\ldots y_m^{p_m}\delta(f);
    $$
    $$
    \delta_-(h)=2\gamma  y_2^{p_2}\ldots y_m^{p_m}z_2^{q_2}\ldots z_n^{q_n}\delta(h) \text{ for }h\in\KK[X].
    $$
    Again one can check that $\delta_-$ is a well-defined locally nilpotent derivation on $\KK[Y]$ satisfying $\delta_-(\gamma)=0$.
    \begin{re}
        These extensions of $\delta$ are analogic to some LNDs given in \cite{G} for trinomial hypersurfaces.
    \end{re}

        \begin{lem}\label{vkerl}
    The quotient fields of the kernels of $\delta_+$ and $\delta_-$ are as follows:
    \begin{enumerate}
        \item $\Quot(\Ker\delta_+)=\Quot(\Ker\delta)(y_2,\ldots, y_m, z_2,\ldots, z_n, \rho)$;
        \item
        $\Quot(\Ker\delta_-)=\Quot(\Ker\delta)(y_2,\ldots, y_m, z_2,\ldots, z_n, \gamma)$.
    \end{enumerate}
    \end{lem}
    \begin{proof}
        Let $r$ be a local slice of $\delta$. Let us consider the Dixmier map $\pi_{\delta_+,r}$. We have $\pi_{\delta_+,r}(y_i)=y_i$ and $\pi_{\delta_+,r}(z_j)=z_j$ for $i,j\geq 2$. For $F\in\KK[X]$,
        \begin{multline*}
        \pi_{\delta_+,r}(F)=\sum_{i=0}^\infty\frac{(-1)^i(2\rho  y_2^{p_2}\ldots y_m^{p_m}z_2^{q_2}\ldots z_n^{q_n})^i\delta^i(F)r^i}{i!(2\rho  y_2^{p_2}\ldots y_m^{p_m}z_2^{q_2}\ldots z_n^{q_n}\delta(r))^i}=\\=
        \sum_{i=0}^\infty\frac{(-1)^i\delta^i(F)r^i}{i!(\delta(r))^i}=\pi_{\delta,r}(F);
        \end{multline*}
         \begin{multline*}
         \pi_{\delta_+,r}(y_1)=y_1+\sum_{i=1}^\infty\frac{(-1)^i(2\rho y_2^{p_2}\ldots y_m^{p_m})^{i-1}(z_2^{q_2}\ldots z_n^{q_n})^i\delta^i(f)r^i}{i!(2\rho  y_2^{p_2}\ldots y_m^{p_m}z_2^{q_2}\ldots z_n^{q_n}\delta(r))^i}=\\=
         y_1-\frac{f}{2\rho y_2^{p_2}\ldots y_m^{p_m}}+\frac{1}{2\rho y_2^{p_2}\ldots y_m^{p_m}}\sum_{i=0}^\infty\frac{(-1)^i(2\rho y_2^{p_2}\ldots y_m^{p_m})^{i}(z_2^{q_2}\ldots z_n^{q_n})^i\delta^i(f)r^i}{i!(2\rho  y_2^{p_2}\ldots y_m^{p_m}z_2^{q_2}\ldots z_n^{q_n}\delta(r))^i}=\\=
         y_1-\frac{f}{2\rho y_2^{p_2}\ldots y_m^{p_m}}+\frac{1}{2\rho y_2^{p_2}\ldots y_m^{p_m}}\pi_{\delta,r}(f)=\\
         =\frac{2\rho y_1y_2^{p_2}\ldots y_m^{p_m}-(y_1^2y_2^{2p_2}\ldots y_m^{2p_m}-z_1^2z_2^{2q_2}\ldots z_n^{2q_n})}{2\rho y_2^{p_2}\ldots y_m^{p_m}}+\frac{1}{2\rho y_2^{p_2}\ldots y_m^{p_m}}\pi_{\delta,r}(f)=\\
         =\frac{\rho^2}{2\rho y_2^{p_2}\ldots y_m^{p_m}}+\frac{1}{2\rho y_2^{p_2}\ldots y_m^{p_m}}\pi_{\delta,r}(f)=\frac{\rho}{2 y_2^{p_2}\ldots y_m^{p_m}}+\frac{1}{2\rho y_2^{p_2}\ldots y_m^{p_m}}\pi_{\delta,r}(f).
         \end{multline*}
         In a similar way we obtain 
         $$\pi_{\delta_+,r}(z_1)=-\frac{\rho}{2z_2^{q_2}\ldots z_n^{q_n}}+\frac{1}{2\rho z_2^{q_2}\ldots z_n^{q_n}}\pi_{\delta,r}(f).$$
         So, 
         \begin{multline*}
         \Ker\delta_+=\pi_{\delta_+,r}(\KK[Y])\cap\KK[Y]\subseteq \\\subseteq
         \Ker\delta\left[y_2,\ldots, y_m, z_2,\ldots, z_n,\rho, \frac{1}{y_2^{p_2}\ldots y_m^{p_m}},\frac{1}{z_2^{q_2}\ldots z_n^{q_n}}\right]\cap\KK[Y]\subseteq\\\subseteq 
         \Quot(\Ker\delta)(y_2,\ldots, y_m, z_2,\ldots, z_n, \rho).
         \end{multline*}
         Conversely, $\Ker\delta[y_2,\ldots, y_m, z_2,\ldots, z_n, \rho]\subseteq \Ker\delta_+$, therefore, $$\Quot(\Ker\delta)(y_2,\ldots, y_m, z_2,\ldots, z_n, \rho)\subseteq \Quot(\Ker\delta_+).$$ 
         This proves part (1). Part (2) is analogous. 
    \end{proof}

    A straightforward consequence of the previous lemma is the following statement.
\begin{lem}\label{vlmlsu}
    Suppose $X$ is an irreducible affine variety and $f\in\KK[X]$. Consider 
    $$Y=\Bisusp_{m,n}(X,f,2, 2p_2,\ldots, 2p_m, 2, 2q_2,\ldots, 2q_n).$$ 
    Then $\FML(Y)\subseteq \FML(X)(y_2,\ldots, y_m, z_2,\ldots, z_n)$.
\end{lem}

Let us consider the case when all $p_i$ and $q_j$ equal $1$. So, 
$$
\KK[Y]=\KK[X][y_1,\ldots,y_m, z_1,\ldots, z_n]/(y_1^2\ldots y_m^2- z_1^2\ldots z_n^2-f).
$$
Suppose we are given by nonzero LND $\delta$ of $\KK[X]$. Then we can consider extensions $\delta_+$ and $\delta_-$. But since all exponents are equal to $2$, we can obtain more extensions changing the order of variables. So we obtain $2mn$ extensions of $\delta$, which we denote $\delta_{ij+}$ and $\delta_{ij-}$, where $1\leq i\leq m$ and $1\leq j\leq n$. (Here $\delta_+=\delta_{11+}$ and $\delta_-=\delta_{11-}$.)

\begin{lem}\label{vnlnl}
    In the above notations 
    $$\bigcap_{1\leq i\leq m,1\leq j\leq n}\Quot(\Ker\delta_{ij+})\cap\bigcap_{1\leq i\leq m,1\leq j\leq n} \Quot(\Ker\delta_{ij-})=\Quot(\Ker\delta).$$
\end{lem}
\begin{proof}
By Lemma~\ref{vkerl} we have
$$\Quot(\Ker\delta_{ij+})=\Quot(\Ker\delta)(y_a,z_b,\rho\mid a\neq i, b\neq j).$$ 
Therefore, 
$$\bigcap_{1\leq i\leq m,1\leq j\leq n}\Quot(\Ker\delta_{ij+})=\Quot(\Ker\delta)(\rho).$$
In a similar way
$$\bigcap_{1\leq i\leq m,1\leq j\leq n}\Quot(\Ker\delta_{ij-})=\Quot(\Ker\delta)(\gamma).$$
The assertion follows.
\end{proof}

Let $\delta$ be a nonzero LND on a variety $V$. Denote $$X=\Bisusp_{m,n}(V,f,2, 2p_2,\ldots, 2p_m, 2, 2q_2,\ldots, 2q_n)\times\KK.$$ Let us consider an LND $D=y_2\frac{\partial}{\partial u}+\delta_+$ of $\KK[X]$.
\begin{lem}\label{vl1}
Suppose $p_2>1$. Assume that there exists a factorial $G$-grading (may be a trivial one) on $\KK[V]$ such that $f$ is nonzero homogeneous element. Then $$\mathrm{Spec} (\Ker D)\cong \Bisusp_{m,n}(V,f,2, 2p_2-2,\ldots, 2p_m, 2, 2q_2,\ldots, 2q_n).$$
\end{lem}
\begin{proof}

Note that the derivation $D$ satisfies the conditions ($*$) for $x=y_1$, $y=y_2$.
By Lemma~\ref{sfar}, $\KK[V][z_1,\ldots,z_n]$ is factorially graded by an abelian group such that $f+z_1^2z_2^{2q_2}\ldots z_n^{2q_n}$ is a homogeneously irreducible element.
By Lemma~\ref{facgrlem} and Remark~\ref{bissu}, the algebra 
$$\KK[\Bisusp_{m,n}(V,f,2, 2p_2,\ldots, 2p_m, 2, 2q_2,\ldots, 2q_n)]$$ 
is factorially graded and $y$ is a graded irreducible element. Since $D(x)=z_2^{q_2}\ldots z_n^{q_n}\delta(f)$, we have $\gcd(y,D(x))=1$. Therefore, by Lemma~\ref{pl3} and Remark~\ref{pr2} we are done. 
\end{proof}

\begin{lem}\label{vl2}
Let $\delta$ be a nonzero LND on a variety $V$. Denote $$X=\Bisusp_{m,n}(V,f,2, 2p_2,\ldots, 2p_m, 2, 2q_2,\ldots, 2q_n)\times\KK.$$ Consider LNDs $D=y_2\frac{\partial}{\partial u}+\delta_+$ and
$E=y_2\frac{\partial}{\partial u}+\delta_-$.
Suppose there exists $r\in \KK[V]$ and $a,b\in \Ker\delta$ such that $$h=af+b\delta(f)r\in\Ker\delta.$$ 
Then
\begin{enumerate}
\item There exists an element $v\in \KK[X]$ such that $D(v)=(h+a\rho^2)z_2^{q_2}\ldots z_n^{q_n}$;

\item There exists an element $w\in \KK[X]$ such that $E(w)=(h+a\gamma^2)y_2^{p_2}\ldots y_m^{p_m}$. 
\end{enumerate}

\end{lem}
\begin{proof}

Let us do some computations. Denote by $\pi_u$ the Dixmier map $\pi_{D,u}$.
We have: $\pi_u(y_1)=y_i, \pi_u(z_j)=z_j$ for all $i,j\geq 2$, and 
    \begin{multline*}
    \pi_u(y_1)=y_1-\frac{z_2^{q_2}\ldots z_n^{q_n}\delta(f)u}{y_2}+\frac{2\rho  y_2^{p_2}\ldots y_m^{p_m}z_2^{2q_2}\ldots z_n^{2q_n}\delta^{2}(f)u^2}{2 y_2^2}-\ldots=\\
    =F-\frac{z_2^{q_2}\ldots z_n^{q_n}\delta(f)u}{y_2}=\frac{1}{y_2}\widetilde{y_1},\text{ where }F\in \KK[X];
    \end{multline*}
    $$
    \pi_u(z_1)=z_1-\frac{y_2^{p_2}\ldots y_m^{p_m}\delta(f)u}{y_2}+\frac{2\rho  y_2^{2p_2}\ldots y_m^{2p_m}z_2^{q_2}\ldots z_n^{q_n}\delta^2(f)u^2}{2 y_2^2}-\ldots=\widetilde{z_1}\in\KK[X];
    $$
    $$
    \pi_u(r)=r-\frac{2\rho  y_2^{p_2}\ldots y_m^{p_m}z_2^{q_2}\ldots z_n^{q_n}\delta(r)u}{y_2}+\ldots =\widetilde{r}\in\KK[X].
    $$
So, the functions
$$\widetilde{y_1}=y_2\pi_u(y_1)=z_2^{q_2}\ldots z_n^{q_n}\delta(f)u+y_2A\text{ and }
    \widetilde{r}=\pi_u(r)=r+y_2B,$$ where $A,B\in\KK[X]$, belong to $\Ker D$.
    Let us also compute that
    $$
    \rho^2=(y_1y_2^{p_2}\ldots y_m^{p_m}-z_1z_2^{q_2}\ldots z_n^{q_n})^2=z_1^2z_2^{2q_2}\ldots z_n^{2q_n}+y_2H,\qquad H\in \KK[X].
    $$
So, we have
\begin{multline*}
    y_2(h+a\rho^2)z_2^{q_2}\ldots z_n^{q_n}=D(u(h+a\rho^2)z_2^{q_2}\ldots z_n^{q_n})=\\
=D(u(af+b\delta(f)r+a\rho^2)z_2^{q_2}\ldots z_n^{q_n})
=D(u(af+b\delta(f)r+a\rho^2)z_2^{q_2}\ldots z_n^{q_n}-b\widetilde{y_1}\widetilde{r}) 
=\\
=D(u(af+a\rho^2)z_2^{q_2}\ldots z_n^{q_n}+y_2 C)
=D(u(a(y_1^2y_2^{2p_2}\ldots y_m^{2p_m}-z_1^2z_2^{2q_2}\ldots z_n^{2q_n})+\\+a(z_1^2z_2^{2q_2}\ldots z_n^{2q_n}+y_2H))z_2^{q_2}\ldots z_n^{q_n}+y_2 C)
=D(y_2v).
\end{multline*}
Here $C=b(-z_2^{q_2}\ldots z_n^{q_n}\delta(f)uB-rA+y_2AB)\in\KK[X]$.
Since $X$ is irreducible, and $y_2\neq 0$, we have $D(v)=(h+a\rho^2)z_2^{q_2}\ldots z_n^{q_n}$. This proves part (1). Part (2) is analogous.
\end{proof}

\begin{prop}\label{bsp}
Let $\delta$ be a nonzero LND on a variety $V$. Suppose $r$ is a local slice of $\delta$. Let $f=\alpha r^k+\beta$, where $\alpha,\beta\in \Ker\delta$, $k\in \mathbb{Z}_{>0}$. Suppose that there exists a factorial $G$-grading on $\KK[V]$ for some abelian group $G$, such that $f\in\KK[V]$ is a homogeneous element. 
 Denote $$X=\Bisusp_{m,n}(V,f,2, 2p_2,\ldots, 2p_m, 2, 2q_2,\ldots, 2q_n)\times\KK.$$ Then 
 $\FML(X)\subseteq\Quot(\Ker\delta)\subseteq \KK[V].$
 \end{prop}

 \begin{proof}
By definition,
$$
\KK[X]=\KK[V][y_1,\ldots, y_m,z_1,\ldots, z_n, u]/(y_1^2y_2^{2p_2}\ldots y_m^{2p_m}-z_1^2z_2^{2q_2}\ldots y_n^{2q_n}-f). 
$$
We consider the LND $D_1=y_2\frac{\partial}{\partial u}+\delta_+$ on $\KK[X]$. Note that  By Lemma~\ref{vl1},
$$\mathrm{Spec} (\Ker D_1)\cong \Bisusp_{m,n}(V,f,2, 2p_2-2,\ldots, 2p_m, 2, 2q_2,\ldots, 2q_n).$$
That is 
\begin{multline*}
A_1=\Ker D_1=\\=\pi_u(\KK[V])[\widetilde{y_1},y_2,\ldots, y_m, \widetilde{z_1},z_2,\ldots, z_n]/(\widetilde{y_1}^2y_2^{2p_2-2}\ldots y_m^{2p_m}-\widetilde{z_1}^2z_2^{2q_2}\ldots y_n^{2q_n}-\widetilde{f}).
\end{multline*}
Here $\widetilde{z_1}=\pi_u(z_1)$, $\widetilde{f}=\pi_u(f)$, and $\widetilde{y_1}=y_2\pi_u(y_1)$. 

Consider $h=f-\frac{1}{k}r\delta(r)=\alpha r^k+\beta-\frac{1}{k}r(k\alpha r^{k-1})=\beta\in\Ker\delta_+$. 
By Lemma~\ref{vl2}, $(\beta+\rho^2)z_2^{q_2}\ldots z_n^{q_n}\neq 0\in \mathrm{pl}(D_1)$. Let us denote $\widetilde{\rho}=\widetilde{y_1}y_2^{p_2-1}\ldots y_m^{p_m}-\widetilde{z_1}z_2^{q_2}\ldots y_n^{q_n}$. Then 
$$
    \pi_u(\rho)=\pi_u(y_1y_2^{p_2}\ldots y_M^{p_M}-z_1 z_2^{q_2}\ldots y_N^{q_N})
    =\frac{\widetilde{y_1}}{y_2}y_2^{p_2}\ldots y_m^{p_m}-\widetilde{z_1} z_2^{q_2}\ldots y_n^{q_n}=\widetilde{\rho}. 
$$
Therefore, 
$$
\pi_u((\beta+\rho^2)z_2^{q_2}\ldots z_n^{q_n})=(\pi_u(\beta)+\widetilde{\rho}^2)z_2^{q_2}\ldots z_n^{q_n}.
$$
This shows that if we fix the above isomorphism $$\mathrm{Spec}(A_1)\cong\Bisusp_{m,n}(V,f,2,2p_2-2,2p_3,\ldots, 2p_m,2,2q_2,\ldots, 2q_n),$$ 
it takes $(\beta+\rho^2)z_2^{q_2}\ldots z_n^{q_n}$ (in the initial coordinates) to $(\beta+\rho^2)z_2^{q_2}\ldots z_n^{q_n}$ (in the new coordinates). 

Then we define $D_2=y_2\frac{\partial}{\partial u}+\delta_+$ on $\Ker D_1[u]$ and so on. Finally, we obtain a chain of LNDs $D_1,\ldots, D_{l}$, where $l=p_2+\ldots+p_n-n+1$ and $D_{j+1}\in\LND((\Ker{D_j})[u])$. We have 
$$\mathrm{Spec}(\Ker D_l)=\Bisusp_{m,n}(V,f,2,\ldots,2,2,2q_2,\ldots, 2q_n).$$
For each $D_j$ we have 
$h_j=(\beta+\rho^2)z_2^{q_2}\ldots z_n^{q_n}\in \mathrm{pl}(D_j)$.

By Lemma~\ref{vkerl}, we have $m$ LNDs $\delta_{i1+}$ on $\Ker D_l$ such that 
$$\bigcap_{i=1}^m\Quot(\Ker \delta_{i1+})\subseteq \Quot(\Ker\delta)(\rho, z_2,\ldots, z_n).$$

Let us take one of these LNDs $E=\delta_{i1+}\in \LND(\Ker(D_l)[u])$. We put $E(u)=0$. Then 
$$
E(h_l)=E((\beta+a\rho^2)z_2^{q_2}\ldots z_n^{q_n})=0. 
$$
So, we can apply iterated generalized Bhatwadekar's technique and obtain the LND $E_l\in\LND(B)$ for which $\Quot(\Ker E_l)=\Quot(\Ker E)$. Applying this for all $E=\delta_{i1+}$ we obtain 
$$\FML(X)\subseteq \Quot(\Ker \delta)(\rho,z_1,\ldots, z_n,u).$$
Similarly, considering $\delta_{i1-}$, we can prove that 
$$\FML(X)\subseteq \Quot(\Ker \delta)(\gamma,z_1,\ldots, z_n,u).$$
Hence, 
$$\FML(X)\subseteq \Quot(\Ker \delta)(z_1,\ldots, z_n,u).$$
Similarly,
$\FML(X)\subseteq \Quot(\Ker \delta)(y_1,\ldots, y_m,u)$. Therefore, $$\FML(X)\subseteq \Quot(\Ker\delta)(u).$$ Since $\frac{\partial}{\partial u}$ is an LND on $X$, we have $\FML(X)\subseteq \Quot(\Ker\delta)$. 
\end{proof}
\begin{ex}\label{efmlgf}
    Let $Z=\Bisusp_{m,n}(\mathbb{K}^d, f, 2, 2p_2,\ldots, 2p_m, 2, 2q_2,\ldots, 2q_n)$,  
    where $f=x_{11}^{l_{11}}\ldots x_{1n_1}^{l_{1d_1}}+\ldots+x_{s1}^{l_{s1}}\ldots x_{sn_s}^{l_{sd_s}}$. Here $x_{ij}$ are coordinates on $\mathbb{K}^d$ and $d_1+\ldots+d_s\leq d$, $d_j\geq 0$, and  $s>1$. Denote $X=Z\times \KK$. Then $X$ is generically flexible. Indeed, for $\delta=\frac{\partial}{\partial x_{ij}}$ the conditions of Proposition~\ref{bsp} are satisfied. Therefore, $\FML(X)\subseteq \Quot(\Ker \frac{\partial}{\partial x_{ij}})$. That is $\FML(X)=\KK$.
\end{ex}

The following example is just a particular case of Example~\ref{efmlgf}

\begin{ex}
    Let $Z=\{y_1^2-z_1^2z_2^4=w^3-1\}$. Then $$Z=\Bisusp_{1,2}(\KK,w^{3}-1,2,2,4).$$ 
    Put $X=Z\times \KK$. Then $\FML(X)=\KK$.
\end{ex}

\section{Cylinders over trinomial varieties}\label{cyltrv}

Consider a trinomial variety $X=X(A)=\mathrm{Spec}\,R(A)$; see Definition~\ref{trvarde}. We use notations of Construction~\ref{odyn}. By definition $X(A)$ is isomorphic to a product of another trinomial variety $Y=Y(A)$ and the $\theta$-dimensional affine space. 

If $\theta=1$ and $Y$ is rigid (see the regidity criterion in Proposition~~\ref{EGSh}), then by Makar-Limanov's result~\cite[Theorem~2.24]{F} the variety $X$ is semi-rigid, and therefore, $\ML(X)=\KK[Y]\neq \KK$. The cases 
\begin{itemize}
    \item $Y$ is rigid  and $\theta>1$
    \item $\theta=0$ 
\end{itemize} 
require further investigations; see Section~\ref{lastsec}. Trinomial varieties that are cylinders over non-rigid initial varieties $Y$ turn out to be generically flexible.

\begin{theor}\label{mlml}
    If the variety $Y(A)$ is non-rigid and $\theta>0$, the variety~$X$ is generically flexible. 
\end{theor}

By Proposition~\ref{EGSh}, the variety $Y$ is not rigid if and only if one of the following conditions holds

1) $Y$ can be included as $Y=Y_k$ in the following chain of varieties. The initial variety $Y_1=\KK^{n_0+n_1}$ is the affine space with coordinates $T_{01},\ldots, T_{0n_0}, T_{11},\ldots, T_{1n_1}$. For $j>1$ the variety $Y_j$ is defined iteratively as
$$
Y_j=\Susp(Y_{j-1}, f, 1, l_{j2},\ldots, l_{jn_j}),
$$
where $f=\lambda_j T_0^{l_0}+\mu_j T_1^{l_1}$, for some $\lambda_j, \mu_j\in \KK\setminus \{0\}$. 

2) $Y$ can be included as $Y=Z_{k-1}$ in the following chain of varieties. The initial variety $Z_1=\KK^{n_1}$ is the affine space with coordinates $T_{31},\ldots, T_{3n_3}$. The variety $Z_2$ is defined by 
$$
Z_2=\Bisusp_{n_1,n_2}(Z_1,T_3^{l_3},2, 2p_2,\ldots, 2p_{n_1}, 2, 2q_2,\ldots, 2q_{n_2}).
$$
For $j\geq 2$ the variety $Z_{j+1}$ is defined iteratively as
$$
Z_{j+1}=\Susp(Z_{j}, f, 1, l_{j2},\ldots, l_{jn_j}),
$$
where $f=\lambda_j T_0^{l_0}+\mu_j T_1^{l_1}$, for some $\lambda_j, \mu_j\in \KK\setminus\{0\}$. 

Applying Theorem~\ref{mtt} in case (1) and Proposition~\ref{bsp} and Theorem~\ref{mtt} in case (2) we prove the assertion of the theorem. 
\begin{cor}
    In conditions of Theorem~\ref{mlml} we have $\ML(X)=\KK$.
\end{cor}

\begin{re}
    In a forthcoming paper the second author proves that in conditions of Theorem~\ref{mlml} the trinomial variety $X(A)$ is flexible.
\end{re}

\section{Concluding remarks and questions}
\label{lastsec}

In Section~\ref{secvar} we give an example of an affine variety $X$ of dimension $n$ such that $\Aut(X)$ contains maximal tori of dimensions $1,2,\ldots, n-1$. It is natural to investigate possible sets of dimensions of maximal tori in the automorphism group of an affine variety. It is known that a torus of dimension greater than $n$ can not act on an $n$-dimensional affine space effectively. Also it is obvious that if $\Aut(X)$ admits a maximal torus of dimension zero, it does not admit any other maximal torus. Thus, we obtain the following question.
\begin{q}\label{q1}
    Which subsets of the set $\{1,2,\ldots, n\}$ can be sets of dimensions of maximal tori in automorphism group of $n$-dimensional affine space.
\end{q}
It is proved in~\cite{AG} that a rigid variety admits a unique maximal torus. Using this it is not difficult to construct an $n$-dimensional variety with a unique maximal torus of dimension $k$ for each $0\leq k\leq n$. 
\begin{prop}\label{mns}
    Consider the subvariety $X$ in $\KK^{2n}$ cut by 
    $$
    \begin{cases}
        x_1^3+y_1^2=0;\\
        \ldots\\
        x_k^3+y_k^2=0;\\
        x_{k+1}^3+y_{k+1}^2=1;\\
        \ldots\\
        x_n^3+y_n^2=1.
    \end{cases}
    $$
    Then $X$ is rigid with a unique maximal torus of dimension $k$. 
\end{prop}
\begin{proof}
    AB-theorems imply rigidity of $X$; see~\cite[Theorems~2.48 and~2.50]{F}. It is easy to see that the torus acting by
    $$
    (t_1,\ldots,t_k)\cdot(x_1,y_1,\ldots, x_n, y_n)=
    (t_1^2x_1,t_1^3y_1,\ldots, t_k^2x_k, t_k^3y_k, x_{k+1},y_{k+1},\ldots,x_n,y_n)
    $$
    is maximal.
\end{proof}
\begin{re}
    The variety in Proposition~\ref{mns} is not normal. One can obtain a normal example for $k<n$ using a trinomial equation with sufficiently large exponents instead of the first $k$ equations. For $k=n$ the affine space is a natural candidate for such an example, but it is an open question whether it admits maximal tori of dimensions other than $n$. 
\end{re} 

Using similar ideas we can construct an $n$-dimensional variety with maximal tori of dimensions $1,2,\ldots, k$ for any $k<n$. 

\begin{prop}\label{mnsmns}
Let $Y$ be a $(k+1)$-dimensional variety from Theorem~\ref{raztor} with maximal tori of dimensions $1,2,\ldots, k$ in $\Aut(Y)$. Consider the subvariety $Z$ in $\KK^{2(n-k-1)}$ cut by 
    $$
    \begin{cases}
        x_1^3+y_1^2=1;\\
        \ldots\\
        x_{n-k-1}^3+y_{n-k-1}^2=1.
    \end{cases}
    $$
    Then the set of dimensions of maximal tori in $\Aut(X)$, where $X=Y\times Z$, is $\{1,2,\ldots, k\}$. 
\end{prop}
\begin{proof}
    AB-theorems imply $\ML(X)=\KK[Z]$. Therefore, $\KK[Z]$ is invariant under any action of a torus on $X$. Since $Z$ has a trivial maximal torus, all functions in $\KK[Z]$ are $T$-invariant for any subtorus $T$ of $\Aut(X)$. That is $\dim T\leq k$. Maximal tori in $\Aut(Y)$ give maximal tori in $\Aut(X)$ of the same dimensions. 
\end{proof}

Taking the above into account we can clarify Question~\ref{q1} and ask the following particular cases.

\begin{q}\label{q2}
    Is it possible that $\Aut(X)$ contains maximal tori of different dimensions but it does not admit a maximal torus of dimension one? 
\end{q}

\begin{q}\label{q3}
    Is it possible that the set $S$ of dimensions of maximal tori in $\Aut(X)$ has gaps, i.e., $a,c\in S$ and $b\notin S$ for some $a<b<c$?
\end{q}

A fundamental result by Białynicki-Birula~\cite{BB} implies that the $n$-dimensional affine space does not contain an $(n-1)$-dimensional maximal torus. Therefore, the negative answer to Question~\ref{q3} will imply that $\Aut(\KK^n)$ does not admit any maximal torus of dimensions other that $n$. (The question whether it is true is known as the Linearization Problem for tori.) This implies that non-toric varieties can not be counterexamples to the Zariski Cancellation Problem; see Lemma~\ref{lemum}. Since a toric variety has a unique structure of a toric variety, these varieties are not counterexamples to the Zariski Cancellation Problem, too. Thus, negative answer to Question~\ref{q3} will imply the affirmative answer to the Zariski Cancellation Problem. 

One can use the above arguments to prove that a particular variety is not a counterexample to the Zariski Cancellation Problem. Indeed, any variety $X$ with a torus action of complexity one (i.e. admitting an effective action of a torus of dimension $\dim X-1$) is not a counter-example. Indeed, it is either toric or admits a maximal torus of dimension $\dim X-1$, which by Lemma~\ref{lemum} contradicts Białynicki-Birula's result. If a variety $Y$ is stably isomorphic to 
a variety $X\ncong\KK^n$ with a torus action of complexity one, then $Y$  is not a counterexample either. In particular, all modifications of $m$-suspensions $\Susp(\KK^n, z_1^{s_1}\ldots z_p^{s_p}-1,k_1,\ldots, k_m)$ are not counterexamples to the Zariski Cancellation Problem.

\begin{ex}
    The variety 
    $$
Y=\{xy^2=zw^2+yw+1\}
$$ 
is not a counterexample to the Zariski Cancellation Problem. Indeed, 
$$Y\times\KK\cong \{xy^2=zw^2+1\}\times\KK$$ 
and has a maximal torus of dimension $3$; see Example~\ref{e6}. 
\end{ex}

For Danielewski varieties $\{xy_1^{k_1}\ldots y_m^{k_m}=z^2-1\}$ we proved in Sections~\ref{secconj} and~\ref{secvar} that there exist infinitely many non-conjugate tori of dimension $n-1$ and maximal tori of different dimensions. It seems that one can use the isolated irreducible semi-invariants technique to check whether there are infinitely many non-conjugate tori of other possible dimensions. A natural question is whether this is a general fact.
\begin{q}\label{q4}
Are there implications between the following statements:
\begin{enumerate}
    \item the variety $X$ is non-cancellative;
    \item the group $\Aut(X\times\KK)$ admits infinitely many non-conjugate maximal tori;
    \item the group $\Aut(X\times\KK)$ admits maximal tori of different dimensions.
\end{enumerate}
\end{q}

Section~\ref{cyltrv} is devoted to proving generic flexibility of cylinders over trinomial varieties $X(A)$ with non-rigid initial variety $Y(A)$. It is an intriguing question to compute Makar-Limanov invariant and to obtain criteria of flexibility and generic flexibility for arbitrary trinomial varieties (at least for trinomial hypersurfaces). Some partial results on this topic can be found in~\cite{G} and  \cite{IV}. But for some rather simple varieties we do not know the Makar-Limanov invariant. 

\begin{ex}
    Let $X=\{xyz^2=w^2-1\}\subseteq \mathbb{K}^4$. It is easy to write down two LNDs $\delta_1$ and $\delta_2$ on $X$ such that $\Ker \delta_1\cap\Ker\delta_2=\KK[z]$. We do not know any LND on $X$ such that $z$ is not in the kernel. But it is an open question whether $z\in \ML(X)$. Note that by Corollary~\ref{jal}, $X\times \KK$ is generically flexible, and hence, $\ML(X\times \KK)=\KK$. 
\end{ex}

One more interesting problem in this field is to describe $\ML(Y\times\KK^\theta)$ for rigid $Y$ and $\theta>1$. As we have mentioned in Section~\ref{cyltrv}, when $\theta=1$, then the Makar-Limanov invariant equals $\KK[Y]$. But in~\cite{D} Dubouloz introduced examples of non-isomorphic rigid varieties with isomorphic $\mathbb{K}^2$-cylinders. So, it is a natural question, whether this is possible for trinomial varieties. Therefore, it is not clear, whether the stable Makar-Limanov invariant, i.e. $\mathrm{SML}(X)=\bigcap_{n=1}^\infty \ML(X\times \KK^n)$,  of a rigid trinomial variety $X$ is isomorphic to $\KK[X]$. Note that for some trinomial varieties $X$ this is the case.
\begin{ex}
Let $X$ be the trinomial hypersurface 
    $\{x_1^2\ldots x_l^2+y_1^3\ldots y_m^3=1\}.$
Then by $AB$-Theorem~\cite[Theorem~2.50]{F}, $x_1\ldots x_l$ and $y_1\ldots y_m$ are in $\ML(X\times\KK^n)$. Thus, all $x_i$ and $y_j$ are in $\ML(X\times\KK^n)$. So, $\ML(X\times \KK^n)=\KK[X]$ for all $n$. 
\end{ex}

Let us ask the following question.

\begin{q}
    For which trinomial varieties $X$ is it true that $\mathrm{SML}(X)=\KK[X]$?
\end{q}


\begin{thebibliography}{}
\fontsize{10pt}{10pt}\selectfont

\bibitem{Ar} I. Arzhantsev. {\it On rigidity of factorial trinomial hypersurfaces}. International Journal of Algebra and Computation. 26 (2016), no. 5, 1061-1070


\bibitem{AFKKZ}
I. Arzhantsev, H. Flenner, S. Kaliman, F. Kutzschebauch, and M. Zaidenberg. {\it Flexible varieties and automorphism groups}. Duke Math. J. 162 (2013), no. 4, 767-823

\bibitem{AG} I. Arzhantsev and S. Gaifullin. {\it  The automorphism group of a rigid affine variety}. Math. Nachr. 290 (2017), no. 5-6, 662–671

%

\bibitem{AKZ}
I. Arzhantsev, K. Kuyumzhiyan, and M. Zaidenberg. {\it Flag varieties, toric varieties, and
suspensions: three instances of infinite transitivity}. Sb. Math. 203 (2012), no.~7, 923-949

\bibitem{BB} A. Białynicki-Birula. {\it Linearity of the action of a maximal torus on affine space}. Bull. Acad. Polon. Sci. Sér. Sci. Math. Astronom. Phys. 14 (1966), 569–574

\bibitem{BGSh} I. Boldyrev, S. Gaifullin, and A. Shafarevich. {\it Modified Derksen invariant}. arXiv:2312.08421
%
\bibitem{BG} V. Borovik and S. Gaifullin. {\it Isolated torus invariants and automorphism groups of rigid varieties}. J. of Algebra.  666 (2025), 821-839



\bibitem{Da} W. Danielewski. {\it On a cancellation problem and automorphism groups of affine algebraic varieties}.
Preprint, Warsaw, 1989.

\bibitem{DGO} A.K. Dutta, N. Gupta, and N. Onoda. {\it Some patching results on algebras over two-dimensional factorial domains}. J. Pure Appl. Algebra. 216 (2012), 1667– 1679

\bibitem{Du} A. Dubouloz. {\it Additive group actions on Danielewski varieties and the cancellation problem}. Math. Z.
255 (2007), no. 1, 77–93

\bibitem{D} A. Dubouloz, \emph{The cylinder over the Koras-Russell cubic threefold has a trivial Makar-Limanov invariant}. Transform.\ Groups. 14 (2009), 531-539 

\bibitem{Du2} A. Dubouloz. {\it Rigid affine surfaces with isomorphic $\mathbb{A}^2$-cylinders}.  arXiv:1507.05802, 2015


\bibitem{DF} A. Dubouloz and D. R. Finston. {\it On exotic affine 3-spheres}. J. Algebraic Geom. 23 (2014), no. 3, 445-469

\bibitem{DG} A. Dubouloz and P.Ghosh. {\it Algebraic families of higher dimensional $\mathbb{A}^1$-contractible affine varieties non-isomorphic to affine spaces}. arXiv:2501.09613 



\bibitem{Ess} A. van den Essen. Polynomial Automorphisms and the Jacobian Conjecture. Birkhauser, Boston, 2000 


\bibitem{E-G-S}
P. Evdokimova, S. Gaifullin and A. Shafarevich. {\it Rigid trinomial varieties}, arXiv:2307.06672, 2023

\bibitem{Fi} K.-H. Fieseler. {\it On complex affine surfaces with $\mathbb{C}^{+}$-action.} Commentarii Mathematici Helvetici. 69 (1994), 5-27

\bibitem{FKZ} H. Flenner, Sh. Kaliman, and M. Zaidenberg. {\it Cancellation for Surfaces Revisited}. Memoirs of the American Mathematical Society 278, 2022


\bibitem{F} G. Freudenburg.  Algebraic Theory of Locally Nilpotent Derivations, Second Edition. Springer-Verlag, Berlin, Heidelberg, 2017


\bibitem{Ga} S. Gaifullin. {\it Automorphisms of Danielewski varieties}. J. of Algebra.  573 (2021), 364-392

\bibitem{G} S. Gaifullin. {\it On Rigidity of Trinomial Hypersurfaces and Factorial Trinomial Varieties}. arXiv:1902.06136

\bibitem{Ga2} S. Gaifullin. {\it Generically flexible affine varieties with invariant divisors}. arxiv:2507.14745

\bibitem{GSh} S. Gaifullin and A. Shafarevich.
{\it Modified Makar-Limanov and Derksen invariants}.
J. Pure and Appl. Algebra,
228 
(2024), no. 6, 107616

\bibitem{Gu2} N. Gupta. {\it On the cancellation problem for the aﬃne space $\mathbb{A}^3$ in characteristic~$p$}.
Inventiones Mathamaticae. 195 (2014) 279–288

\bibitem{Gu} N. Gupta. {\it On Zariski’s Cancellation Problem in positive characteristic}. Adv. Math.  264 (2014), 296–307

\bibitem{GS} N. Gupta and S. Sen. {\it On Double Danielewski Surfaces and the Cancellation Problem}. J of Algebra. 533 (2019), 25-43

\bibitem{H-H}
J. Hausen and E. Herppich. \textit{Factorially graded rings of complexity one}. Torsors, $\acute{\mathrm{e}}$tale homotopy
and applications to rational points, 414–428, London Math. Soc. Lecture Note Ser., 405,
Cambridge Univ. Press, Cambridge, 2013.

\bibitem{H-W}
J. Hausen and M. Wrobel. \textit{Non-complete rational T-varieties of complexity one}. Math. Nachr. 290 (2017), no. 5-6, 815-826



\bibitem{IV} M. Ignatev and T. Vilkin. {\it On Flexibility of Trinomial Varieties}. Mediterr. J. Math. 23 (2026), 70

\bibitem{KML} S. Kaliman and L. Makar-Limanov. {\it On the Russell-Koras contractible threefolds}. J. Algebraic Geom.  6 (1997), 247–268

\bibitem{KZ}
Sh. Kaliman and M. Zaidenberg. {\it Affine modifications and affine hypersurfaces with a very transitive automorphism group}. Transform. Groups. 4 (1999), 53-95
%

\bibitem{KrZ} H. Kraft and M. Zaidenberg. {\it Automorphism groups of affine varieties and their Lie algebras}. arXiv:2403.12489



\bibitem{ML96} L. Makar-Limanov. \emph{On the hypersurface $x + x^2y + z^2 + t^3 = 0$  in $\mathbb{C}^4$ or a $\mathbb{C}^3$-like threefold which is not $\mathbb{C}^3$}.  Israel J. Math.  96 (1996), 419-429

\bibitem{ML01} L. Makar-Limanov. {\it On the group of automorphisms of a surface $x^ny = P(z)$}. Israel J. Math. 121
(2001), 113-123

\bibitem{MJ} L. Moser-Jauslin. \emph{Automorphism groups of Koras-Russell threefolds of the first kind}. Affine algebraic geometry, 261-270, CRM Proc. Lecture Notes, 54, Amer.
Math. Soc., Providence, RI, 2011

\bibitem{P} C. Petitjean.  \emph{Automorphism groups of Koras–Russell threefolds of the second kind}.  Beitr. Algebra Geom. 57 (2016), 599–605

\bibitem{Po} V. Popov. {Tori in the Cremona groups}. Izvestiya: Mathematics, 77 (2013), no. 4, 742–771

\bibitem{Re} R. Rentschler. {\it Operations du groupe additif sur le plane affine}. C. R. Acad. Sci. 267 (1968), 384-387


\end{thebibliography}
\end{document}